\documentclass[12pt]{article}
\usepackage{pdfrender,xcolor}
\usepackage{graphicx,color}
\usepackage{amsmath,amsfonts,amssymb,amsthm}
\usepackage{mathabx}
\usepackage[hidelinks,colorlinks=true,linkcolor=blue,citecolor=blue]{hyperref}
\usepackage{rsfso}
\usepackage{comment}
\usepackage{enumerate}

\usepackage[numbers]{natbib}
\newcommand{\vc}[1]{{\boldsymbol #1}}

\newcommand{\sr}[1]{{\mathcal #1}}
\newcommand{\dd}[1]{\mathbb{#1}}

\newcommand{\ol}{\overline}
\newcommand{\ul}{\underline}
\newcommand{\os}{\overset}

\newcommand{\eq}[1]{(\ref{eq:#1})}
\newcommand{\lem}[1]{Lemma~\ref{lem:#1}}
\newcommand{\cor}[1]{Corollary~\ref{cor:#1}}
\newcommand{\thr}[1]{Theorem~\ref{thr:#1}}

\newcommand{\dfn}[1]{Definition~\ref{dfn:#1}}

\newcommand{\rem}[1]{Remark~\ref{rem:#1}}

\newcommand{\app}[1]{Appendix~\ref{app:#1}}
\newcommand{\sectn}[1]{Section~\ref{sec:#1}}

\newcommand{\lemt}[1]{\ref{lem:#1}}

\newcommand{\thrt}[1]{\ref{thr:#1}}

\newcommand{\sect}[1]{\ref{sec:#1}}

\newcommand{\pend}{\hfill \thicklines \framebox(6.6,6.6)[l]{}}

\newenvironment{proof*}[1]{\noindent {\sc  #1} \rm}{\pend}
\newenvironment{keyword}[1]{ {\bf #1} \rm}{}

\newtheorem{theorem}{Theorem}[section]
\newtheorem{lemma}{Lemma}[section]

\newtheorem{remark}{Remark}[section]

\newtheorem{corollary}{Corollary}[section]
\newtheorem{definition}{Definition}[section]

\newenvironment{mylist}[1]{\begin{list}{}
{\setlength{\itemindent}{#1mm}}
{\setlength{\itemsep}{0ex plus 0.2ex}}
{\setlength{\parsep}{0.5ex plus 0.2ex}}
{\setlength{\labelwidth}{10mm}}
}{\end{list}}

\newcommand{\setnewcounter} {
\setcounter{subsection}{0}
\setcounter{equation}{0}
\setcounter{figure}{0}
\setcounter{conjecture}{0}
\setcounter{assumption}{0}
\setcounter{conjecture}{0}
\setcounter{question}{0}
\setcounter{definition}{0}
\setcounter{theorem}{0}
\setcounter{corollary}{0}
\setcounter{lemma}{0}
\setcounter{proposition}{0}
\setcounter{remark}{0}
}

\begin{document}
 \title{\Large \bf Heavy-traffic limit of stationary distributions of\\ a state-dependent queue}

\author{Masahiro Kobayashi\\ Tokai University \and Masakiyo Miyazawa\\ Tokyo University of Science \and Yutaka Sakuma \\National Defense Academy of Japan}
\date{\today, R1-3\_12}

\maketitle

\begin{abstract}

Inspired by the work of \citet{AtarMiya2026} for a multi-level queue as well as applications to energy-saving problems, we are interested in the heavy-traffic limit of the stationary queue length distribution, which is not addressed in \cite{AtarMiya2026}. In this paper, we consider this heavy-traffic limit for the single server queue which has the most general possible state-dependence. Namely, arrival and service processes are generated by renewal processes, and their time scales, called speeds, may take any values depending on the queue length, not limited to specific values such as levels in the multi-level queue. Here, the terminology, heavy-traffic limit, stands for a diffusion-scaled limit in heavy-traffic for processes, distributions and modeling primitives.

This general model is referred to as a state-dependent queue. There are two motivations for this generalization. One is interest in the state-dependent queue itself because it allows finer control of service speed in application. Another is making it clear how the heavy-traffic limit is obtained under what conditions for the state-dependent queue.

Thus, we start to study basic properties of this state-dependent queue, including its stability. We then take the sequence of the stationary distributions of its diffusion scaled queue-length processes. We have three main results for this sequence. We first show that it is tight if the heavy-traffic limit of their drifts exists and is negative as the queue length goes to infinity, where a drift is the arrival speed minus the service speed. We next assume the condition that the limit of every vaguely convergent subsequence has a density, which is referred to as a density condition, and show that the heavy-traffic limit of the stationary distributions is obtained in a closed form if and only if that negative drift condition holds. We then show that the density condition is always satisfied for the multi-level queue, so the problem is nicely solved for the multi-level queue.

Our tool to prove those results is the BAR approach developed for the 2-level $GI/G/1$ queue in \citet{Miya2025}, but the present work {requires} many auxiliary results for this approach, which are of independent interest as they show what is needed to widen the applicability of the BAR approach.

\end{abstract}

\keyword{Keywords:}{single-server queue, state-dependent arrivals and services, multi-level queue, stationary distribution, heavy-traffic, diffusion approximation, Palm distribution.}


\section{Introduction}
\label{sec:introduction}
\setnewcounter

We are interested in the diffusion approximation of a single-server queue with an unlimited buffer whose arrival and service processes are queue-length dependent, motivated by the optimal control of service speed. For example, such control problems have been studied for service stations in communication networks. The diffusion approximation is known as a useful tool to study them because it {effectively} captures the variability of arrivals and service times, which is hard to study by a Markovian queue, that is, a model which {is described by a continuous-time Markov chain}.

Among such single-server queues, we are inspired by the multi-level queue introduced by \citet{AtarMiya2026}. This multi-level queue is a simpler version of the 2-level $GI/G/1$ queue studied in \cite{Miya2025}. It has a single server and finitely many disjoint intervals constituting a partition of $\dd{R}_{+} \equiv [0,\infty)$, and its service speed depends on such an interval to which the queue length belongs. Those intervals are referred to as level sets. This {service} dynamics is exactly the same as {in} the 2-level $GI/G/1$ queue. However, their arrival processes are different. In the multi-level queue of \cite{AtarMiya2026}, the arrival times of customers are generated by a single renewal process whose time scale, which is called an arrival speed, {changes} similarly to the service speed, while different renewal processes are chosen {depending} on the queue length in the 2-level $GI/G/1$ queue. This causes analytical difficulty in the study of the 2-level $GI/G/1$ queue, while the multi-level queue has much simpler dynamics but is still tractable for applications.

For the multi-level queue, it is shown in \cite{AtarMiya2026} that the heavy-traffic limit of the queue length process is a reflected diffusion with a drift and a diffusion coefficient which discontinuously change at the boundaries of level sets. These boundary states are called levels. However, the heavy-traffic limit of the stationary distribution of its queue length is not addressed in \cite{AtarMiya2026}. We are particularly interested in this distributional limit because of its importance in applications. Furthermore, we extend the multi-level queue so that it has countably many levels which are not accumulated at a finite point. We still call this extended model a multi-level queue for convenience.

In this paper, we study the heavy-traffic limit of the stationary queue length distribution for the single server queue whose arrival and service speeds have the most general possible state-dependence. Namely, arrival and service speeds may take any positive values depending on the queue length, not limited to specific values such as levels in the multi-level queue. There are two motivations for this generalization. One is interest in the state-dependent queue itself because it allows finer control of service speed in application. Another is making it clear how the heavy-traffic limit is obtained under what conditions for the state-dependent queue. In this paper, a heavy-traffic limit stands for a diffusion-scaled limit for processes, distributions and modeling primitives in heavy-traffic. 

Thus, our first task is to precisely describe the state-dependent queue, and study its basic properties, including its stability. We then take the sequence of the stationary distributions of its diffusion scaled queue-length processes, and introduce heavy-traffic conditions for them. Assuming these heavy traffic conditions, we consider the weak limit of this sequence, and have three main results in this paper, which are listed below.

\begin{enumerate}[i)]
\item  The sequence of those stationary distributions is tight if the heavy-traffic limit of their drifts exists and is negative as the diffusion scaled queue length goes to infinity, which is referred to as a \emph{negative drift condition}, where a drift is the arrival speed minus the service speed (see \thr{tight-2}). Here, tightness is known as a key condition for for the weak convergence of distributions.

\item We introduce the condition that the vague limit has a density for any vaguely convergent subsequence of the sequence of the stationary distributions, which is referred to as a \emph{density condition}. Assuming this density condition, we prove that the heavy-traffic limit of the stationary distributions exists and is obtained in a closed form if and only if the negative drift condition holds (see \thr{SDQ-main-1}).

\item For the multi-level queue, it is proved that the density condition always holds, and therefore its heavy-traffic limit of the stationary distributions is obtained if and only if the negative drift condition holds (see \thr{MLQ-main-1}).
\end{enumerate}

Our tool to prove those results is the BAR approach based on Palm distributions{,} which {was} developed for the 2-level $GI/G/1$ queue in \cite{Miya2025}, but we need many auxiliary results for applying this approach to the state-dependent queue. {These results are of independent interest as} they show what is needed to widen the applicability of the BAR approach.

As for the process limit for the state-dependent queue, no results are available except for {the multi-level queue of \cite{AtarMiya2026}}. Thus, it is interesting to study the heavy-traffic limit of the queue length process for the state-dependent queue. Intuitively, it must exist as a reflected diffusion, but its proof is quite technical even for the level-dependent queue in \cite{AtarMiya2026}, so we leave it in future study. Based on this process-limit and \thr{tight-2}, we conjecture that the density condition in \thr{SDQ-main-1} is not needed (see \rem{SDQ-main}).

This paper is {organized into} five sections. In \sectn{preliminaries}, we introduce the sequence of indexed state-dependent queues, present the heavy-traffic assumptions, and discuss a stationary framework including Palm distributions. The main results, \thr{tight-2}, Theorems \thrt{SDQ-main-1} and \thrt{MLQ-main-1}, are presented in \sectn{main}. Auxiliary results to prove these theorems are prepared in \sectn{asym-BAR}, and the main results are proved in \sectn{p-main}. In the Appendix, proofs of some lemmas are provided.

\section{Preliminaries and assumptions}
\label{sec:preliminaries}
\setnewcounter

This section consists of six subsections. In \sectn{framework}, we introduce a stochastic framework to be used in this paper, and discuss basic facts on vague and weak convergences. Then, in Sections \sect{state-dependent} and \sect{Markov}, we define a state-dependent queue and describe it by a Markov process. Here, {``}state-dependent{''} means that arrival and service speeds depend on the current queue length. This state-dependent queue can be considered as a generalization of the level-dependent queue of \cite{AtarMiya2026} in which the speeds are similarly changed but only at the time when the queue length crosses one of finitely many fixed levels.

We then present heavy-traffic conditions on this state-dependent queue for diffusion approximation in \sectn{HT-c}, and discuss its time evolution in \sectn{time-evolution}. Finally, in \sectn{stationary}, we introduce the stationary framework for the diffusion approximation of the stationary distribution, following \cite{Miya2024,Miya2025}.

We use standard notations such as $a \vee b = \max(a,b)$, $a \wedge b = \min(a,b)$ and {$a^{+} = \max(a, 0)$} for $a,b \in \dd{R}$, where $\dd{R}$ is the set of all real numbers. Similarly, let $\dd{R}_{+}, \dd{Z}_{+}$ and $\dd{N}$ be the sets of all nonnegative real numbers, all nonnegative integers and all positive integers, respectively. Denote {by $\sr{B}_{+}$} the Borel field on $\dd{R}_{+}$. As usual, all processes are assumed to be right-continuous with left limits unless stated otherwise. A function $f$ is also denoted by $f(\cdot)$ when emphasized to be a function. For a real-valued function $f$ on $\dd{R}_{+}$ which has left-limits $f(t-)$ at $t > 0$, we define the difference operator $\Delta$ by
\begin{align*}
  \Delta f(t) = f(t) - f(t-),
\end{align*}
where this difference is defined at $t=0$ when $f(0-)$ is defined. We also use the small order notation $o(\cdot)$. Namely, we write $f(n) = o(g(n))$ as $n \to \infty$ for real-valued functions $f$ and $g$ if $f(n)/g(n) \to 0$ as $n \to \infty$.

\subsection{Stochastic framework and basic facts}
\label{sec:framework}

We introduce the stochastic framework which will be used in this paper. Let $(\Omega,\sr{F},\dd{P})$ be a probability space, and let $\dd{F} \equiv \{\sr{F}_{t}; t \ge 0\}$ be a filtration, where $\sr{F}_{t}$ is a $\sigma$-field on $\Omega$ which is a subset of $\sr{F}$ and nondecreasing in $t \ge 0$, then the quartet $(\Omega,\sr{F},\dd{F},\dd{P})$ is called a stochastic basis. On this stochastic basis, we define a time-shift operator semigroup $\Theta_{\bullet} \equiv \{\Theta_{t}; t \ge 0\}$ on $\Omega$, {where} $\Theta_{t}$ is an operator on $\Omega$ such that
\begin{align*}
 & \Theta_{s+t} (\omega) = \Theta_s (\Theta_{t} (\omega)), \qquad s,t \ge 0, \omega \in \Omega, \nonumber\\
 & \Theta_{t}(A^{-1}) \equiv \{\omega \in \Omega; \Theta_{t} (\omega) \in A\} \in \sr{F}, \qquad t \ge 0, A \in \sr{F}.
\end{align*}

Throughout the paper, we denote by $X(\cdot)$ a stochastic process $\{X(t); t \ge 0\}$. Then, $X(\cdot)$ on the stochastic basis{, $(\Omega,\sr{F},\dd{F},\dd{P})$,} is said to be consistent with $\Theta_{\bullet}$ if
\begin{align*}
 X(s)(\Theta_{t} (\omega)) = X(s+t) (\omega), \qquad s,t \ge 0, \omega \in \Omega.
\end{align*}
Similarly, the consistency of a counting process is defined. The stochastic basis with the time shift operator $\Theta_{\bullet}$ will be used to define Palm distributions {with respect to} counting processes in \sectn{stationary}, following \cite{Miya2024}.

As mentioned in \sectn{introduction}, the tightness of a sequence of probability measures is {key} to {proving} their weak convergence, where probability measure is the synonym of distribution. Here, a sequence of probability measures $\pi_{n}$ on $(\dd{R}_{+},\sr{B}_{+})$ for $n \ge 1$ is said to be tight if, for any $\varepsilon > 0$, there is an $x \in \dd{R}_{+}$ such that
\begin{align*}
 \pi_{n}([0,x]) > 1 - \varepsilon, \qquad \forall n \ge 1.
\end{align*}

To present a basic property of tightness, we recall two modes of convergence of probability measures. A sequence of probability measures $\{\pi_{n}; n \ge 1\}$ on $(\dd{R}_{+},\sr{B}_{+})$ is said to {converge vaguely} if there is a sub-probability measure $\pi$ on $(\dd{R}_{+},\sr{B}_{+})$ such that $\pi_{n}((a,b]) \to \pi((a,b])$ as $n \to \infty$ for any $a, b \in \dd{R}_{+}$ satisfying $a < b$, which is denoted by $\pi_{n} \os{v}{\rightarrow} \pi$, where $\pi$ is called a sub-probability measure if $\pi(\dd{R}_{+}) \le 1$. In particular, if this $\pi$ is a probability measure, $\pi_{n}$ is said to {converge weakly} to $\pi$ as $n \to \infty$, which is denoted by $\pi_{n} \os{w}{\rightarrow} \pi$. The following results {are} well known (e.g., see Theorem 5.19 and Proposition 5.21 of \cite{Kall2001} and Theorem 4.3.4 of \cite{Chun2001}).

\begin{lemma}[Helly's selection Theorem]
 \label{lem:Helly}
 For any {subsequence} of probability measures $\{\pi_{n}; n \ge 1\}$ on $(\dd{R}_{+},\sr{B}_{+})$, there {exists a further} subsequence $\{\pi_{n_{k}}; k \ge 1\}$ {that converges vaguely}, that is, $\pi_{n_{k}} \os{v}{\rightarrow} \pi$ as $k \to \infty${,} for some sub-probability measure $\pi$.
\end{lemma}

\begin{lemma}
 \label{lem:tight-1}
 A sequence of probability measures $\{\pi_{n}; n \ge 1\}$ is tight if and only if, for any subsequence of $\{\pi_{n}\}$, there {exists a} further subsequence {that converges weakly}. In particular, if this weak limit, denoted by $\pi$, is unique, then $\pi_{n} \os{w}{\to} \pi$ as $n \to \infty$.
\end{lemma}

\subsection{Sequence of state-dependent queues}
\label{sec:state-dependent}

We introduce the state-dependent queue. This is a single-server queue with an unlimited buffer in which customers bring $i.i.d.$ amounts of work for service, which are independent of arrival times, and are served in the first-come first-served manner. Its arrival process is generated by a renewal process in such a way that the residual arrival time decreases at a rate that depends on the current queue length. We call this rate an arrival speed. Similarly, its service speed depends on the queue length.

To present heavy-traffic, we use the sequence of the state-dependent queues which are indexed by positive integers $n$. For each $n \ge 1$, denote the arrival and service speeds when the queue length is $\ell$ by $\lambda^{(n)}(\ell)$ and $\mu^{(n)}(\ell)$, respectively. For convenience, we assume $\mu^{(n)}(0) = \lambda^{(n)}(0)$. This assumption has no influence on this queueing model because there is no service when the system is empty. Those speeds are defined on $\dd{Z}_{+}$. However, it is convenient to extend them to be defined on $\dd{R}_{+}$ because their diffusion-scaled limit is real valued in general. To this end, we redefine functions $\lambda^{(n)}(\cdot)$ and $\mu^{(n)}(\cdot)$ as
\begin{align}
 \label{eq:i-to-u}
 \lambda^{(n)}(u) = \lambda^{(n)}(\ell), \qquad \mu^{(n)}(u) = \mu^{(n)}(\ell), \qquad u \in (\ell-1,\ell], \ell \ge 1,
\end{align}
where $\lambda^{(n)}(0)$ and $\mu^{(n)}(0)$ are as they are. Obviously, these functions are left-continuous with right limits in $u$. Define the drift of the $n$-th state-dependent queue as
\begin{align}
\label{eq:drift-n}
  \gamma^{(n)}(u) = \lambda^{(n)}(u) - \mu^{(n)}(u) 1(u > 0), \qquad u \in \dd{R}_{+}.
\end{align}
This drift and its heavy-traffic limit play an important role in our analysis.

We are also interested in the multi-level queue of \cite{AtarMiya2026}, which is a special case of the state-dependent queue. We introduce a slightly extended version of this multi-level queue below.

\begin{definition}[Multi-level queue]
\label{dfn:MLQ-1}
For the $n$-th state-dependent queue, if the arrival and service speeds $\lambda^{(n)}(\cdot)$ and $\mu^{(n)}(\cdot)$ are step functions on $\dd{R}_{+}$ such that
\begin{mylist}{0}
\item [(M1)] they have finitely many discontinuous points in each finite interval,
\item [(M2)] the limits of $\lambda^{(n)}(u)$ and $\mu^{(n)}(u)$ as $u \to \infty$ exist and positive.
\end{mylist}
Let $K^{(n)}$ be the total number of the discontinuous points in (M1), which may be infinite. Define $\ell^{(n)}_{0} = 0$, and inductively define $\ell^{(n)}_{i}$ for $i \le K^{(n)}$ as
\begin{align}
\label{eq:level-i}
  \ell^{(n)}_{i} = \inf\{ u > \ell^{(n)}_{i-1}; |\Delta \lambda^{(n)}(u)| + |\Delta \mu^{(n)}(u)| > 0\},
\end{align}
then $\ell^{(n)}_{i}$ is called the $i$th level. Let $\ell^{(n)}_{\infty} = \ell^{(n)}_{K^{(n)}}$ if $K^{(n)} < \infty$, and $\ell^{(n)}_{\infty} = \infty$ otherwise. We call the state-dependent queue satisfying (M1) annd (M2) the $n$-th multi-level queue. 
\end{definition}

\begin{remark}
\label{rem:MLQ-1}
(i) For the multi-level queue, $\lambda^{(n)}(u)$ and $\mu^{(n)}(u)$ are constants for $u$ in either $[0,\ell^{(n)}_{1}]$ or $(\ell^{(n)}_{i-1},\ell^{(n)}_{i}]$ for $2 \le i < \infty$ or $(\ell^{(n)}_{K^{(n)}}, \infty)$ if $K^{(n)} < \infty$.\\
(ii) We will assume that $K^{(n)}$ is independent of $n$ in heavy-traffic conditions, and denote it by $K$ (see (M3) in \sectn{HT-c}). However, in this and next subsections, we do not assume $K^{(n)} = K$ because basic properties are studied for the $n$-th model for a fixed $n$.
\end{remark}

This definition of the multi-level queue extends that of \cite{AtarMiya2026} in which the total numbers of the levels are finite. As we will see, the multi-level queue will help us to better understand the state-dependent queue because of its simpler state-dependence.

For the state-dependent queue, we assume the following condition.
\begin{mylist}{3}
 \item [(\sect{preliminaries}.a)] For each $n \ge 1$, $\lambda^{(n)}(\cdot)$ and $\mu^{(n)}(\cdot)$ are measurable nonnegative functions such that
 \begin{align*}
  & \lambda^{(n)}_{\sup} \equiv \sup_{u \in \dd{R}_{+}} \lambda^{(n)}(u) < \infty, \qquad \mu^{(n)}_{\sup} \equiv \sup_{u \in \dd{R}_{+}} \mu^{(n)}(u) < \infty,\\
  & \lambda^{(n)}_{\inf} \equiv \inf_{u \in \dd{R}_{+}} \lambda^{(n)}(u) > 0, \qquad \mu^{(n)}_{\inf} \equiv \inf_{u \in \dd{R}_{+}} \mu^{(n)}(u) > 0.
 \end{align*}
\end{mylist}
The condition (\sect{preliminaries}.a) is obviously satisfied by the multi-level queue.

We next introduce common counting processes $A(\cdot)$ and $S(\cdot)$ which generate the {arrivals} and service completions of the $n$-th state-dependent queue for each $n \ge 1$. Following \cite{AtarMiya2026}, denote the $j$-th counting times of $A$ and $S$ by $t_{A,j}$ and $t_{S,j}$, respectively, for $j=1,2,\ldots$. Let $t_{A,0} = t_{S,0} = 0$, and let
\begin{align*}
 T_{A}(j) = t_{A,j+1} - t_{A,j}, \qquad T_{S}(j) = t_{S,j+1} - t_{S,j}, \qquad j=0, 1, \ldots.
\end{align*}
Then, we have
\begin{align*}
 A(t) = \inf \left\{k \ge 0; t_{A,k+1} > t \right\}, \qquad S(t) = \inf \left\{ k \ge 0; t_{S,k+1} > t \right\}, \qquad t \ge 0.
\end{align*}

Assume the following conditions for these counting processes.

\begin{mylist}{3}
 \item [(\sect{preliminaries}.b)] $A(\cdot)$ and $S(\cdot)$ are delayed renewal processes and independent of each other. Namely, $T_{A}(j) > 0$, $j=1,2,\ldots$ are $i.i.d.$ random variables, while $T_{A}(0) \ge 0$ is independent of $T_{A}(j)$ for $j \ge 1$. Similar {independence is} assumed for $T_{S}(j)$ for $j=0,1,\ldots$.
 
 \item [(\sect{preliminaries}.c)]  Denote by $T_{A}$ and $T_{S}$ random variables subject to the common distributions of $T_{A}(j)$ and $T_{S}(j)$ for $j \ge 1$, respectively. These $T_{A}$ and $T_{S}$ have unit means and finite third moments, namely,
 \begin{align}
  \label{eq:moment-c}
  & \dd{E}[T_{A}] = 1, \quad \dd{E}[T_{A}^{3}] < \infty, \qquad \dd{E}[T_{S}]= 1, \quad \dd{E}[T_{S}^{3}] < \infty.
 \end{align}
 while $T_{A}(0)$ and $T_{S}(0)$ do not necessarily have the same distributions as $T_{A}$ and $T_{S}$, respectively, but have finite means and variances. Furthermore, it is assumed that $\sigma_{A}^{2} + \sigma_{S}^{2} > 0$, where $\sigma_{A}^{2} = \dd{E}[(T_{A} - 1)^{2}]$ and $\sigma_{S}^{2} = \dd{E}[(T_{S} - 1)^{2}]$, namely, the variances of $T_{A}$ and $T_{S}$, respectively, which are finite by \eq{moment-c}.
 
\end{mylist}

\begin{remark}
 \label{rem:moment-c}
 {The finiteness of the second moments of $T_{A}$ and $T_{S}$ is the standard assumption for diffusion approximations, but we here assume the stronger condition that their third moments are finite.} However, these finiteness conditions are crucial in the present BAR approach similarly to \cite{Miya2025}.
\end{remark}

\subsection{Markov process for the state-dependent queue}
\label{sec:Markov}

In this section, we construct a Markov process which describes the state-dependent queue indexed by positive integer $n$, called the $n$-th state-dependent queue, and give conditions for this Markov process to be positive recurrent.

We first inductively define the queue length process $L^{(n)}(\cdot) \equiv \{L^{(n)}(t); t \ge 0\}$ for the $n$-th state-dependent queue. To this end, define
\begin{align*}
 & U^{(n)}(L^{(n)},t) = \int_{0}^{t} \lambda^{(n)}(L^{(n)}(s)) ds, \;\; V^{(n)}(L^{(n)},t) = \int_{0}^{t} \mu^{(n)}(L^{(n)}(s)) 1(L^{(n)}(s) > 0) ds, \;\; t \ge 0,
\end{align*}
and let
\begin{align}
\label{eq:ZLn-t}
 & Z^{(n)}(L^{(n)},t) = A(U^{(n)}(L^{(n)},t)) - S(V^{(n)}(L^{(n)},t)).
\end{align}
Then, if $L^{(n)}(\cdot)$ is well defined, $A(U^{(n)}(L^{(n)},t))$ and $S(V^{(n)}(L^{(n)},t))$ {represent} the total numbers of {arrivals} and departures, respectively, up to time $t$. Hence, we must have
\begin{align}
 \label{eq:Ln-t1}
 & L^{(n)}(t) = L^{(n)}(0-) + Z^{(n)}(L^{(n)},t), \qquad t \ge 0,
\end{align}
where $L^{(n)}(0-)$ is the number of customers at time $0$ which does not include an exogenous arrival and service completion at time $0$ if they are. Conversely, we have the following lemma, which is intuitively obvious, but its proof requires to solve the non-linear functional equation \eq{Ln-t1}, so we prove it in \app{Ln-unique}.
\begin{lemma}
\label{lem:Ln-unique}
For a given $L^{(n)}(0-)$, $L^{(n)}(\cdot)$ is uniquely determined by \eq{Ln-t1}.
\end{lemma}

 Using the $L^{(n)}(\cdot)$ of \lem{Ln-unique}, define the counting processes $N^{(n)}_{e}(\cdot)$ and $N^{(n)}_{d}(\cdot)$ by
\begin{align*}
  N^{(n)}_{e}(t) = A(U^{(n)}(L^{(n)},t)), \qquad N^{(n)}_{d}(t) = S(V^{(n)}(L^{(n)},t)), \qquad t \ge 0.
\end{align*}
 Obviously, they count the arrivals and departures of the $n$-th system.
 
 Thus, the joint process $\{(L^{(n)}(t), N^{(n)}_{e}(t), N^{(n)}_{d}(t)); t \ge 0\}$ is well defined. However, this joint process is not Markov in general. So, we next introduce supplementary variables for $L^{(n)}(\cdot)$ to be a component of a Markov process. 

For these supplementary variables, we take $R^{(n)}_{e}(t)$ and $R^{(n)}_{d}(t)$ which are defined as
\begin{align}
 \label{eq:Re-t1}
 & R^{(n)}_{e}(t) = T_{A}(0) + \sum_{i=1}^{N^{(n)}_{e}(t)} T_{A}(i) - U^{(n)}(L^{(n)},t), \\
 \label{eq:Rd-t1}
 & R^{(n)}_{d}(t) = T_{S}(0) + \sum_{i=1}^{N^{(n)}_{d}(t)} T_{{S}}(i) - V^{(n)}(L^{(n)},t), \qquad t \ge 0.
\end{align}
We call $R^{(n)}_{e}(t)$ and $R^{(n)}_{d}(t)$ nominal residual times for arrivals and service completions, respectively.

Note that {$R^{(n)}_{e}(t)$} is determined by $L^{(n)}(\cdot)$ up to time $t$ and the sequence of the inter-arrival times $\{T_{A}(\ell); \ell =0,1,\ldots\}$ of $A$, and different from the actual residual arrival time $\widehat{R}^{(n)}_{e}(t)$ which is defined as
\begin{align*}
  \widehat{R}^{(n)}_{e}(t) = \inf\{s \ge 0; N^{(n)}_{e}(t+s) > N^{(n)}_{e}(t)\}, \qquad t \ge 0.
\end{align*}
Because the actual residual arrival time $\widehat{R}^{(n)}_{e}(t)$ depends on the evolution $\{L^{(n)}(s); s > t\}$, it cannot be used for constructing a Markov process. The same applies to the actual residual service time. 

Let $t^{(n)}_{e,j}$ and $t^{(n)}_{d,j}$ be the $j$-th counting time of $N^{(n)}_{e}$ and $N^{(n)}_{d}$, respectively, then, from the definitions \eq{Re-t1} and \eq{Rd-t1}, we have
\begin{align}
 \label{eq:Re-t2}
 & R^{(n)}_{e}(t^{(n)}_{e,j}) = T_{A}(j), \quad \frac {d^{+}}{dt} R^{(n)}_{e}(t) = - \lambda^{(n)}(L^{(n)}(t)),\\
 \label{eq:Rd-t2}
 & R^{(n)}_{d}(t^{(n)}_{d,j}) = T_{S}(j), \quad \frac {d^{+}}{dt} R^{(n)}_{d}(t) = - \mu^{(n)}(L^{(n)}(t)) 1(L^{(n)}(t) > 0), \qquad t \ge 0,
\end{align}
where $\frac {d^{+}}{dt}$ stands for taking a derivative from right. From these facts and their definitions, we can see that $R^{(n)}_{e}(t-) = 0$ if and only if $t$ is the counting time of $N^{(n)}_{e}(\cdot)$. Similarly, $R^{(n)}_{d}(t-) = 0$ and $R^{(n)}_{d}(t) > 0$ if and only if $t$ is the counting time of $N^{(n)}_{d}(\cdot)$. Hence, $N^{(n)}_{e}(\cdot)$ and $N^{(n)}_{d}(\cdot)$ are uniquely determined by the nominal residual times $R^{(n)}_{e}(\cdot)$ and $R^{(n)}_{d}(\cdot)$.

Define processes $R^{(n)}(\cdot) \equiv \{R^{(n)}(t); t \ge 0\}$ and $X^{(n)}(\cdot) \equiv \{X^{(n)}(t); t \ge 0\}$ as
\begin{align*}
 R^{(n)}(t) = (R^{(n)}_{e}(t), R^{(n)}_{d}(t)), \qquad X^{(n)}(t) = (L^{(n)}(t), R^{(n)}_{e}(t), R^{(n)}_{d}(t)), \qquad t \ge 0,
\end{align*}
and define the filtration $\dd{F}^{(n)} \equiv \{\sr{F}^{(n)}_{t}; t \ge 0\}$ by
\begin{align}
\label{eq:Fn-t}
 \sr{F}^{(n)}_{t} = \sigma(\{X^{(n)}(s); s \in [0,t]\}),
\end{align}
that is, the $\sigma$-field generated by $X^{(n)}(s)$ for all $s \in [0,t]$. Further, define the counting process $N^{(n)}(\cdot)$ as
\begin{align}
 \label{eq:Nn-t}
 N^{(n)}(t) = \sum_{0<s \le t} 1(\Delta N^{(n)}_{e}(s) + \Delta N^{(n)}_{d}(s) \ge 1), \qquad t \ge 0,
\end{align}
where recall that $\Delta N^{(n)}_{v}(t) = N^{(n)}_{v}(t) - N^{(n)}_{v}(t-)$ for $v=e,d$. Then all the discontinuous times of $X^{(n)}(\cdot)$ are counted by $N(\cdot)$. Hence, by \eq{Re-t2} and \eq{Rd-t2}, $X^{(n)}(\cdot)$ is a piecewise linear Markov process of \cite{Davi1993}. Thus, we have the next lemma.

\begin{lemma}
\label{lem:Xn-defined}
$X^{(n)}(\cdot)$ is a strong Markov process with respect to the filtration $\dd{F}^{(n)}$, and $N^{(n)}_{e}(\cdot)$ and $N^{(n)}_{d}(\cdot)$ are adapted to $\dd{F}^{(n)}$.
\end{lemma}

Let $\dd{F}$ be the filtration generated by $\dd{F}^{(n)}$ for all $n \ge 1$, then it follows from the construction of $X^{(n)}(\cdot)$ and the filtration $\dd{F}^{(n)}$ that there is a time-shift operator $\Theta_{\bullet}$ on stochastic basis $(\Omega,\sr{F},\dd{F},\dd{P})$ such that $X^{(n)}(\cdot)$ is consistent with $\Theta_{\bullet}$ for all $n \ge 1$ because $\{X^{(n)}(\cdot); n \ge 1\}$ is a collection of countably many stochastic processes. Throughout the paper, we consider this $X^{(n)}(\cdot)$ for $n \ge 1$.

As for the stability of $X^{(n)}(\cdot)$, we assume the following two conditions.
\begin{mylist}{3}
 \item [(\sect{preliminaries}.d)]
 $\dd{P}[T_{A} \ge x] > 0$ for all $ x > 0$, and there is a measurable nonnegative function $p(\cdot)$ and a positive integer $k$ such that
 \begin{align*}
  & \dd{P}\left[ a \le \sum_{i=1}^{k} T_{A}(i) \le b\right] \ge \int_a^{b} p(x) dx, \quad \forall a,b,\ 0 \le a < b, \qquad \int_{0}^{\infty} p(x) dx > 0.
 \end{align*}
 
 \item [(\sect{preliminaries}.e)] $\gamma^{(n)}_{\infty} \equiv \lim_{x \to \infty} \gamma^{(n)}(x)$ exists and $\gamma^{(n)}_{\infty} < 0$ for $n \ge 1$, where recall the definition \eq{drift-n} of the drift $\gamma^{(n)}(\cdot)$.
 
\end{mylist}

\begin{lemma}
 \label{lem:Xn-stability}
 For the $n$-th state-dependent queue satisfying (\sect{preliminaries}.a)--(\sect{preliminaries}.c), the strong Markov process $X^{(n)}(\cdot)$ is positive recurrent if (\sect{preliminaries}.d) and (\sect{preliminaries}.e) hold.
\end{lemma}

\begin{remark}
 \label{rem:Xn-stability}
 The condition (\sect{preliminaries}.d) is typically {assumed} for a Markov process describing a queueing model to be irreducible {in the literature} (e.g., see \cite[(1.4) and (1.5)]{Dai1995}).
\end{remark}

Intuitively, \lem{Xn-stability} is clear because (\sect{preliminaries}.e) suggests that the fluid limit for $L^{(n)}(\cdot)$ has a negative drift. However, the proof is a bit technical, so we detail it in \app{Xn-stability} in the Appendix.

\subsection{Heavy-traffic conditions}
\label{sec:HT-c}

We introduce heavy-traffic conditions for the sequence of the state-dependent queues indexed by $n$. First{,} recall some of the modeling primitives of the $n$-th system:
\begin{align*}
 & \lambda^{(n)}(u), \quad \mu^{(n)}(u), \quad u \in \dd{R}_{+},\\
 & \dd{E}[T_{A}] = 1, \quad \sigma_{A}^{2} = \dd{E}[(T_{A} - 1)^{2}], \qquad \dd{E}[T_{S}]= 1, \quad \sigma_{S}^{2} = \dd{E}[(T_{S} - 1)^{2}].
\end{align*}
Using these notations, we define, for $u \in \dd{R}_{+}$,
\begin{align}
 \label{eq:bsber-n}
 \begin{split}
  & b^{(n)}(u) = n^{1/2} (\lambda^{(n)}(u) - \mu^{(n)}(u)), \qquad \sigma^{(n)}(u) = (\lambda^{(n)}(u) \sigma_{A}^{2} + \mu^{(n)}(u) \sigma_{S}^{2})^{1/2},\\
  & \beta^{(n)}(u) = \frac {2 b^{(n)}(u)} {(\sigma^{(n)})^{2}(u)}, \qquad \rho^{(n)}(u) = \lambda^{(n)}(u)/\mu^{(n)}(u),
 \end{split}
\end{align}
where $\sigma^{(n)}(u) > 0$ for any $u \in \dd{R}_{+}$ by (\sect{preliminaries}.a) and (\sect{preliminaries}.b), so $\beta^{(n)}(u)$ is well defined. Note that $b^{(n)}(u)$ is slightly different from the drift $\gamma^{(n)}(u)$ of \eq{drift-n}.

For heavy-traffic approximation, we scale the queue length $L^{(n)}(t)$ by $n^{-1/2}$ but do not scale time because we are only concerned with the stationary distribution of the process $X^{(n)}(\cdot)$. Thus, define stochastic processes $\widehat{L}^{(n)}(\cdot)$ and $\widehat{X}^{(n)}(\cdot)$ for heavy-traffic approximation as
\begin{align}
\label{eq:h-LX}
 \widehat{L}^{(n)}(t) = n^{-1/2} L^{(n)}(t), \qquad \widehat{X}^{(n)}(t) = (\widehat{L}^{(n)}(t), R^{(n)}(t)), \qquad t \ge 0,
\end{align}
where $\widehat{L}^{(n)}(t)$ may not be integer-valued, but $n^{1/2} \widehat{L}^{(n)}(t)$ must be integer-valued. To present this integer, we define
\begin{align}
\label{eq:qn-x}
  q^{(n)}_{x} = [n^{1/2} x], \qquad x \in \dd{R}_{+}, n \ge 1,
\end{align}
where $[a]$ is the largest integer not greater than $a \in \dd{R}$, then
\begin{align}
\label{eq:hLn-Ln}
  \{\widehat{L}^{(n)}(t) \le x\} = \{L^{(n)}(t) \le n^{1/2} x\} = \{L^{(n)}(t) \le q^{(n)}_{x}\}.
\end{align}
To describe the dynamics of $\widehat{X}(\cdot)$, we modify the arrival and service speeds as
\begin{align}
 \label{eq:la-mu-hn}
 & \widehat{\lambda}^{(n)}(u) = \lambda^{(n)}(n^{1/2}u), \qquad \widehat{\mu}^{(n)}(u) = \mu^{(n)}(n^{1/2}u), \qquad u \in \dd{R}_{+},
\end{align}
{then} it is easy to see that
\begin{align*}
 \widehat{\lambda}^{(n)}(\widehat{L}^{(n)}(t)) = \lambda^{(n)}(L^{(n)}(t)), \qquad \widehat{\mu}^{(n)}(\widehat{L}^{(n)}(t)) = \mu^{(n)}(L^{(n)}(t)).
\end{align*}
Note that the hat notations, $\widehat{\lambda}^{(n)}(u)$ and $\widehat{\mu}^{(n)}(u)$, work nicely for $u = \widehat{L}^{(n)}(t)$ in heavy-traffic because $\widehat{L}^{(n)}(t)$ is likely finite as $n \to \infty$, while this is not the case for $\lambda^{(n)}(u)$ and $\mu^{(n)}(u)$ for $u = L^{(n)}(t)$ because $L^{(n)}(t)$ diverges in heavy-traffic as $n \to \infty$.

We assume the following condition for heavy-traffic approximation.
\begin{mylist}{3}
 \item [(\sect{preliminaries}.f)] There are positive-valued functions $\lambda(\cdot)$, $\lambda^{*}(\cdot)$, $\mu(\cdot)$ and $\mu^{*}(\cdot)$ satisfying the following conditions.
 \begin{itemize}
  \item [(f1)] They are bounded from above and away from 0, and left-continuous with right limits at their discontinuous points.
  \item [(f2)] For each finite interval, the set of all their discontinuous points is finite. Namely, for each $t > 0$, the set $\bigcup_{f = \lambda, \lambda^{*}, \mu,\mu^{*}} \{u \in (0,t]; f(u) \not= f(u+)\}$ is finite, where $f(t+)$ is the right limit of function $f$ at $t \ge 0$.
  \item [(f3)] $\lambda(\cdot) = \mu(\cdot)$, and
  \begin{align}
   \label{eq:heavy-t-la1}
   & \lim_{n \to \infty} \sup_{u \in \dd{R}_{+}} \left|n^{1/2} (\widehat{\lambda}^{(n)}(u) - \lambda(u)) - \lambda^{*}(u)\right| = 0,\\
   \label{eq:heavy-t-mu1}
   & \lim_{n \to \infty} \sup_{u \in \dd{R}_{+}} \left|n^{1/2} (\widehat{\mu}^{(n)}(u) - \mu(u) ) - \mu^{*}(u)\right| = 0.
  \end{align}
 \end{itemize}
 \end{mylist}

The intuitive meanings of $\lambda(u)$, $\mu(u)$, $\lambda^{*}(u)$, $\mu^{*}(u)$ and related quantities can be seen by \lem{h-limit} below. In particular, if those functions are constant, then (f3) is a typical condition for heavy-traffic in the literature. Similar to (f3), we have asymptotic behaviors for other characteristics. To see them, we introduce further notations for \eq{bsber-n} similarly to \eq{la-mu-hn}. They are defined as
\begin{align}
 \label{eq:bs-hn}
 & \widehat{b}^{(n)}(u) =b^{(n)}(n^{1/2} u), \qquad \widehat{\sigma}^{(n)}(u) = \sigma^{(n)}(n^{1/2}u), \qquad \widehat{\beta}^{(n)}(u) = \beta^{(n)}(n^{1/2}u).
\end{align}
We call $\widehat{b}^{(n)}(\cdot)$ and $\widehat{\sigma}^{(n)}(\cdot)$ a diffusion-scaled drift and a diffusion-scaled variance, respectively. Define, for $u \in \dd{R}_{+}$,
\begin{align}
 \label{eq:bsber-u}
 & b(u) = \lambda^{*}(u) - \mu^{*}(u), \quad \sigma(u) = (\lambda(u) \sigma_{A}^{2} + \mu(u) \sigma_{S}^{2})^{1/2}, \quad \beta(u) = \frac {2 b(u)} {\sigma^{2}(u)}.
\end{align}
Note that $\sigma^{2}(u)$ is bounded from above and away from $0$ by (f1) of (\sect{preliminaries}.f). Then we have the following lemma, which easily follows from (\sect{preliminaries}.f), but its proof is outlined in \app{h-limit} for completeness.
\begin{lemma}
 \label{lem:h-limit}
 Under the conditions (\sect{preliminaries}.a) and (\sect{preliminaries}.f), we have
 \begin{align}
  \label{eq:limit-la-mu-n}
  & \lim_{n \to \infty} \sup_{u \in \dd{R}_{+}} \left|\widehat{\lambda}^{(n)}(u) - \lambda(u)\right| = 0, \quad \lim_{n \to \infty} \sup_{u \in \dd{R}_{+}} \left| \widehat{\mu}^{(n)}(u) - \mu(u)\right| = 0,\\
  \label{eq:limit-b-sg-n}
  & \lim_{n \to \infty} \sup_{u \in \dd{R}_{+}} \left|\widehat{b}^{(n)}(u) - b(u) \right| = 0, \quad \lim_{n \to \infty} \sup_{u \in \dd{R}_{+}} |\widehat{\sigma}^{(n)}(u) - \sigma(u)| = 0,
 \end{align}
 which yield
 \begin{align}
  \label{eq:beta-limit}
  &  \lim_{n \to \infty} \sup_{u \in \dd{R}_{+}} |\widehat{\beta}^{(n)}(u) - \beta(u)| = 0.
 \end{align}
 \end{lemma}

Note that, by (f1) of (\sect{preliminaries}.f), $|b(u)|$ and $|\beta(u)|$ are locally integrable, that is,
\begin{align}
 \label{eq:bs-locint}
 & \int_{0}^{x} |b(u)| du < \infty, \qquad \int_{0}^{x} |\beta(u)| du < \infty, \qquad \forall x > 0.
\end{align}
The heavy-traffic condition (\sect{preliminaries}.f) also allows for limiting operations to {work nicely} for the diffusion-scaled process $\widehat{L}^{(n)}(\cdot)$. We present it as a lemma.

\begin{lemma}
 \label{lem:limit-hbn}
 Assume the condition (\sect{preliminaries}.f) and that $\widehat{L}^{(n)}(\cdot)$ has the stationary distribution $\nu^{(n)}$ for each $n \ge 0$ and $\nu^{(n)}$ converges vaguely to a sub-probability measure $\nu$ as $n \to \infty$. Then, for $0 \le x < y \le \infty$,
 \begin{align}
  \label{eq:limit-hbn}
   & \lim_{n \to \infty} \int_{x}^{y} \widehat{b}^{(n)}(u) d\nu^{(n)}(u) = \int_{x}^{y} b(u) d\nu(u),\\
  \label{eq:limit-hsn}
   & \lim_{n \to \infty} \int_{x}^{y} (\widehat{\sigma}^{(n)})^{2}(u) d\nu^{(n)}(u) = \int_{x}^{y} \sigma^{2}(u) d\nu(u),
 \end{align}
 where $\int_{x}^{y}$ stands for the integration over $(x,y] \cap [0,\infty)$ for $0 < x < y \le \infty$, and $\int_{0}^{y}$ stands for the integration over $[0,y] \cap [0,\infty)$ for $0 < y \le \infty$.
\end{lemma}

This lemma is easily proved, but we give its proof in \app{limit-hbn} because of its importance. For the multi-level queue of \cite{AtarMiya2026}, the heavy traffic conditions are given in a different way. In the rest of this section, we show that they are equivalent to (\sect{preliminaries}.f).

Consider the sequence of the multi-level queues of \dfn{MLQ-1}. Let $J_{i,j} = \{i \in \dd{Z}_{+}; i \le k \le j\}$ by $J_{i,j}$ for $i,j \in \dd{Z}_{+} \cup \{\infty\}$, and recall that $K^{(n)}$ is the total number of the levels of the $n$-th multi-level queue . For the multi-level queue, we assume:
\begin{mylist}{3}
\item [(M3)] For each $n \ge 1$, $K^{(n)}$ is independent of $n$, and denoted by $K$.
\item [(M4)] There are a strictly increasing sequence of positive real numbers $\ell_{1}, \ell_{2}, \ldots$ and such that
 \begin{align}
  \label{eq:level-ni}
  \ell^{(n)}_{i} = n^{1/2} \ell_{i}, \qquad i \in J_{1,K}, \qquad n \ge 1.
 \end{align}
\end{mylist}

In \cite{AtarMiya2026}, this $\ell^{(n)}_{i}$ is defined as $\ell^{(n)}_{i} = n^{1/2} \ell_{i} + o(n^{1/2})$. Hence, our assumption \eq{level-ni} is slightly stronger, but this difference is minor in application, while it simplifies our analysis. Note that $\ell_{i}$ and $\ell_{i}^{(n)}$ may not be integers while $q^{(n)}_{x}$ is an integer for any $x \in \dd{R}_{+}$, and
\begin{align*}
  \ell_{i}^{(n)}-1 \le q^{(n)}_{\ell_{i}} = [n^{1/2} \ell_{i}] \le \ell_{i}^{(n)}.
\end{align*}

Define $S_{i}$ and $S^{(n)}_{i}$ for $i \in J_{1,K+1}$ as
 \begin{align}
  \label{eq:l-set-i}
  & S_{1} = [0, \ell_{1}], \qquad S_{i} = (\ell_{i-1}, \ell_{i}], \qquad i \in J_{2,K}, \qquad S_{K+1} = (\ell_{K}, \infty),\\
  \label{eq:ln-set-i}
  & S^{(n)}_{1} = [0,\ell^{(n)}_{1}], \quad S^{(n)}_{i} = (\ell^{(n)}_{i-1}, \ell^{(n)}_{i}], \quad i \in J_{2,K}, \quad S^{(n)}_{K+1} = (\ell^{(n)}_{K}, \infty).
 \end{align}
Note that $S_{K+1} = S^{(n)}_{K+1} = \emptyset$ for $K=\infty$ because $\ell_{K} = \ell^{(n)}_{K} = \infty$, and therefore a proposition on $S_{K+1}$ can be ignored for $K=\infty$. We call $S^{(n)}_{i}$ the $i$-th level set for the $n$-th multi-level queue. Note that the collection of all the level sets is a partition of $\dd{R}_{+}$. 

By the assumption of (M1) and the definition of $\ell^{(n)}_{i}$, there are constants $\lambda^{(n)}_{i}$ and $\mu^{(n)}_{i}$ such that
 \begin{align}
  \label{eq:constant-la-mu}
  \lambda^{(n)}(u) = \lambda^{(n)}_{i}, \qquad \mu^{(n)}(u) = \mu^{(n)}_{i}, \qquad \forall u \in S^{(n)}_{i}, \qquad i \in J_{1,K+1}.
 \end{align}
 
Thus, the multi-level queue has a more specific structure than the state-dependent queue. In particular, from \eq{constant-la-mu},
\begin{align*}
   \widehat{\lambda}^{(n)}(u) = \lambda^{(n)}(n^{1/2} u) = \lambda^{(n)}_{i}, \qquad n^{1/2} u \in S^{(n)}_{i}\; \mbox{ equivalently }\; u \in S_{i}.
\end{align*}
Hence, define $\widehat{\lambda}^{(n)}_{i} = \widehat{\lambda}^{(n)}(u)$ for $u \in S_{i}$, then $\widehat{\lambda}^{(n)}_{i} = \lambda^{(n)}_{i}$ for $i \in J_{1,K+1}$. Similarly, define
 \begin{align*}
  \widehat{\mu}^{(n)}_{i} = \widehat{\mu}^{(n)}(u), \qquad \widehat{b}^{(n)}_{i} = \widehat{b}^{(n)}(u), \qquad \widehat{\sigma}^{(n)}_{i} = \widehat{\sigma}^{(n)}(u), \qquad u \in S_{i},\ i \in J_{1,K+1},
\end{align*}
then $\widehat{\mu}^{(n)}_{i} = \mu^{(n)}_{i}$, $\widehat{b}^{(n)}_{i} = b^{(n)}_{i}$ and $\widehat{\sigma}^{(n)}_{i} = \sigma^{(n)}_{i}$ for $i \in J_{1,K+1}$.

\begin{remark}
\label{rem:MLQ-2}
The arrival and service speeds (called rates) in the $n$-th multi-level queue of \cite{AtarMiya2026} are $n$ times larger than those of the multi-level queue of \dfn{MLQ-1}. Namely, $\lambda^{n}_{i}$ and $\mu^{n}_{i}$ of \cite{AtarMiya2026} are equal to $n\, \lambda^{(n)}_{i}$ and $n\, \mu^{(n)}_{i}$, respectively. This is because the time scaling for diffusion approximation is included in $\lambda^{n}_{i}$ and $\mu^{n}_{i}$ of \cite{AtarMiya2026}.
\end{remark}

By this remark, $\lambda^{n}_{i}$ and $\mu^{n}_{i}$ of \cite{AtarMiya2026} are related to $\lambda^{(n)}_{i}$ and $\mu^{(n)}_{i}$ of the present paper as
\begin{align*}
  \lambda^{n}_{i} = n \lambda^{(n)}_{i} = n \widehat{\lambda}^{(n)}_{i}, \qquad \mu^{n}_{i} = n \mu^{(n)}_{i} = n \widehat{\mu}^{(n)}_{i}, \qquad i \in J_{1,K+1},
\end{align*}
where $K$ of \cite{AtarMiya2026} is replaced $K+1$. Therefore, the heavy-traffic conditions of \cite[(2.14)]{AtarMiya2026} can be written as (M3), (M4) and
\begin{mylist}{3}
\item [(M5)] There are $\lambda_{i} > 0, \; \lambda^{*}_{i} \in \dd{R}$ for $i \in J_{1,K+1}$ and $\mu_{i} > 0, \; \mu^{*}_{i} \in \dd{R}$ for $i \in J_{1,K+1}$ such that $\lambda_{i} = \mu_{i}$ and
  \begin{align*}
   \widehat{\lambda}^{(n)}_{i} = \lambda_{i} + \lambda^{*}_{i} n^{-1/2} + o(n^{-1/2}) > 0, \qquad \widehat{\mu}^{(n)}_{i} = \mu_{i} + \mu^{*}_{i} n^{-1/2} + o(n^{-1/2}) > 0.
  \end{align*}
\end{mylist}
 
 Let $b_{i} = \lambda^{*}_{i} - \mu^{*}_{i}$, $\sigma_{i} = \sigma_{i}(u)$ and $\beta_{i} = \beta_{i}(u)$ for $u \in S_{i}$, then, from the condition (M5), we have
 \begin{align}
  \label{eq:2g-heavy}
\begin{split}
 & n^{1/2} (\widehat{\lambda}^{(n)}_{i} - \lambda_{i}) \to \lambda^{*}_{i}, \qquad n^{1/2} (\widehat{\mu}^{(n)}_{i} - \mu_{i}) \to \mu^{*}_{i}, \\
 & \widehat{b}^{(n)}_{i} \to b_{i}, \qquad \widehat{\sigma}^{(n)}_{i} \to \sigma_{i}, \qquad \widehat{\beta}^{(n)}_{i} \to \beta_{i}, 
\end{split} \qquad n \to \infty.
 \end{align}
 
 \begin{lemma}\rm
\label{lem:MLQ-heavy}
For the multi-level queue, the heavy traffic conditions (M3), (M4) and (M5) are equivalent to (\sect{preliminaries}.f).
\end{lemma}

\begin{proof}
Note that we have defined continuous variable functions such as $\lambda^{(n)}(u)$ and $\mu^{(n)}(u)$ for the $n$-th multi-label queue, but not defined them for $\lambda_{i}, \mu_{i}, \lambda^{*}_{i}$ and $\mu^{*}_{i}$. So, define
 \begin{align*}
  & \lambda(u) = \sum_{i=1}^{K} \lambda_{i} 1(u \in S_{i}), \qquad \mu(u) = \sum_{i=1}^{K} \mu_{i} 1(u \in S_{i}),\\
  & \lambda^{*}(u) = \sum_{i=1}^{K} \lambda^{*}_{i} 1(u \in S_{i}), \qquad \mu^{*}(u) = \sum_{i=1}^{K} \mu^{*}_{i} 1(u \in S_{i}).
 \end{align*}
Assume that (M3), (M4) and (M5) hold. Then, from the definitions of $\widehat{\lambda}^{(n)}(u)$, $\lambda(u)$ and $\lambda^{*}(u)$, we have
 \begin{align*}
  & \left|n^{1/2} (\widehat{\lambda}^{(n)}(u) - \lambda(u)) - \lambda^{*}(u)\right| = \left|\sum_{i=1}^{K} \left(n^{1/2} (\widehat{\lambda}^{(n)}_{i} - \lambda_{i}) - \lambda^{*}_{i} \right) 1(u \in S_{i})\right| \nonumber\\
  & \quad \le \max_{i \in J_{1,K+1}} \left|n^{1/2} (\widehat{\lambda}^{(n)}_{i} - \lambda_{i}) - \lambda^{*}_{i} \right|.
 \end{align*}
 This implies \eq{limit-la-mu-n} because the last term in the above inequality does not depend on $u \in \dd{R}_{+}$ and vanishes as $n \to \infty$ by \eq{2g-heavy}. Similarly, \eq{limit-b-sg-n} is proved. Hence, we have (\sect{preliminaries}.f). Conversely, assume (\sect{preliminaries}.f), then, by \lem{h-limit}, we have the heavy-traffic conditions (M3), (M4) and (M5).
\end{proof}

\subsection{Time evolution of the $n$-th state-dependent queue}
\label{sec:time-evolution}

Note that $X^{(n)}(\cdot)$ is uniquely {determined} by the following data{:}
\begin{align}
 \label{eq:SDQ-primitive}
 A, S, \{(\lambda^{(n)}(u), \mu^{(n)}(u)); u \in \dd{R}_{+}\},
\end{align}
where recall that $\mu^{(n)}(0) = \lambda^{(n)}(0)$. We refer to these data as the modeling primitives of the $n$-th system.

We next describe the time evolution of $X^{(n)}(\cdot)$ using a test function. For this, we introduce some notations. We choose $\dd{R}_{+}^{3}$ {as} the state space of $X^{(n)}(\cdot)$. Let $D(\dd{R}_{+}^{3})$ be the set of all measurable functions from $(u,v_{1},v_{2}) \in \dd{R}_{+}^{3}$ to $\dd{R}$ which are continuous in $(v_{1},v_{2}) \in \dd{R}_{+}^{2}$. Let $D_{b}(\dd{R}_{+}^{3})$ be the set of all the functions in $D(\dd{R}_{+}^{3})$ which are bounded, and let $D_{b}^{p}(\dd{R}_{+}^{3})$ be the set of all the functions in $D_{b}(\dd{R}_{+}^{3})$ whose 2nd and 3rd componens have bounded partial derivatives from the right with respect to each variable. For $f \in D_{b}^{p}(\dd{R}_{+}^{3})$, denote the partial derivatives of $f(u,v_{1},v_{2})$ {with respect to} $v_{i}$ from the right by $\frac {\partial^{+}}{\partial v_{i}}$ for $i=1,2$.

Define operator $\sr{H}^{(n)}$ which maps $f \in D_{b}^{p}(\dd{R}_{+}^{3})$ to $\sr{H}^{(n)} f \in D_{b}(\dd{R}_{+}^{3})$ by
\begin{align*}
 \sr{H}^{(n)} f(u,\vc{v}) & = - \lambda^{(n)}(u) \frac {\partial^{+}}{\partial v_{1}} f(u,\vc{v}) - \mu^{(n)}(u) 1(u>0) \frac {\partial^{+}}{\partial v_{2}} f(u,\vc{v}), \quad (u,\vc{v}) \equiv {(u, v_{1}, v_{2})} \in \dd{R}_{+}^{3}.
\end{align*}
This $\sr{H}^{(n)}$ describes the instantaneous change of $X^{(n)}(t)$ when $L^{(n)}(t) = u$.

To describe the discontinuous changes of $X^{(n)}(t)$, recall the difference operator $\Delta$ which is defined as $\Delta g(t) = g(t) - g(t-)$ for a function $g$ from $\dd{R}_{+}$ to $\dd{R}$ which is right-continuous with {left limits}. These discontinuous changes are counted by $N^{(n)}(\cdot)$ of \eq{Nn-t}, and an arrival and a service completion may simultaneously occur. To describe them precisely, we assume that arrival occurs first then departure occurs. Namely, arrival and departure are sequentially ordered when both of then occur simultaneously, and an intermediate state is introduced between them.

Thus, we define an $\dd{F}$-adapted process $X^{(n)}_{1}(\cdot) \equiv \{(L^{(n)}_{1}(t), R^{(n)}_{e,1}(t), R^{(n)}_{d,1}(t)); t \ge 0\}$ as
\begin{align*}
 & X^{(n)}_{1}(t) =
 \begin{cases}
  X^{(n)}(t-), & \Delta N^{(n)}_{e}(t) = 0,\\
  (L^{(n)}(t-) + 1,  T_{A}({N^{(n)}_{e}}(t)), R^{(n)}_{d}(t-)), & \Delta N^{(n)}_{e}(t) = 1,
 \end{cases}
\end{align*}
and call $X^{(n)}_{1}(t)$ an intermediate state. Then, just after time $t$, we have
\begin{align*}
 X^{(n)}(t) = \begin{cases}
  (L^{(n)}_{1}(t) -1, R^{(n)}_{e,1}(t), T_{S}(N^{(n)}_{d}(t))), & \Delta N^{(n)}_{d}(t) = 1,\\
  X^{(n)}_{1}(t), & \Delta N^{(n)}_{d}(t) = 0.
 \end{cases}
\end{align*}
We can see that the arrival-first assumption has no influence on the queue length.

To describe state changes through the intermediate state, we define other difference operators $\Delta_{e}$ and $\Delta_{d}$ as
\begin{align*}
 & \Delta_{e} f(X^{(n)})(t) = f(X^{(n)}_{1}(t)) - f(X^{(n)}(t-)),\\
 & \Delta_{d} f(X^{(n)})(t) = f(X^{(n)}(t)) - f(X^{(n)}_{1}(t)),
\end{align*}
where the function $f(X^{(n)})$ is defined by $f(X^{(n)})(t) = {f(X^{(n)}(t))}$ for $t \ge 0$. Recall the definition \eq{Nn-t} of $N^{(n)}(\cdot)$, and assume that
\begin{align}
 \label{eq:test-f}
 \mbox{$f(X^{(n)})(t)$ is continuous in $t \ge 0$ if $\Delta N^{(n)}(t) = 0$.}
\end{align}
Then, $\Delta f(X^{(n)})(t) = \Delta_{e} f(X^{(n)})(t) + \Delta_{d} f(X^{(n)})(t)$, so, from an elementary observation of the time evolution of $X^{(n)}(t)$, we have
\begin{align}
 \label{eq:fXn-t}
 f(X^{(n)}(t)) & = f(X^{(n)}(0)) + \int_{0}^{t} \sr{H}^{(n)} f(X^{(n)})(s) ds \nonumber\\
 & \quad + \sum_{s \in (0,t]} [\Delta_{e} f(X^{(n)})(s) + \Delta_{d} f(X^{(n)})(s)], \qquad f \in D_{b}^{p}(\dd{R}_{+}^{3}),
\end{align}
where the summation is well defined because $\Delta_{e} f(X^{(n)})(s)$ and $\Delta_{d} f(X^{(n)})(s)$ vanish except for finitely many $s$ in $(0,t]$.

\subsection{Stationary framework and Palm distributions}
\label{sec:stationary}

Our main interest is the stationary distributions of the indexed systems and {their} diffusion-scaled limit in heavy-traffic. By \lem{Xn-stability}, the {strong} Markov process $X^{(n)}(\cdot)$ is positive recurrent under the conditions (\sect{preliminaries}.a)--(\sect{preliminaries}.e).
We assume these conditions in this section. Then, for each $n \ge 1$, the Markov process $X^{(n)}(\cdot)$ has the stationary distribution. Hence, {it can} be a stationary process. In what follows, {$X^{(n)}(\cdot)$ is assumed to be} a stationary strong Markov process for all $n \ge 1$ unless otherwise stated. Recall the definition \eq{Fn-t} of $\sr{F}^{(n)}_{t}$, and let $\sr{F}^{(n)}_{\infty} = \sigma(\cup_{t \ge 0} \sr{F}^{(n)}_{t})$. Denote the stationary distribution of  $X^{(n)}(\cdot)$ by $\pi^{(n)}$, and define  probability measure $\dd{P}_{\pi^{(n)}}$ on $(\Omega,\sr{F}^{(n)}_{\infty})$ as
\begin{align*}
 \dd{P}_{\pi^{(n)}}(\Theta_{t}^{-1} A) = \dd{P}_{\pi^{(n)}}(A), \qquad A \in \sr{F}^{(n)}_{\infty},\ t \ge 0.
\end{align*}
Since we are only concerned with the sequence of $X^{(n)}(\cdot)$ for $n \ge 1$, we can assume that
\begin{align*}
 \dd{P}(\Theta_{t}^{-1} A) = \dd{P}(A), \qquad A \in \sr{F}, t \ge 0,
\end{align*}
by appropriately choosing $(\Omega,\sr{F},\dd{F})$ (for example, choosing $\sr{F}_{t} = \sigma(\cup_{n \ge 1} \sr{F}^{(n)}_{t})$), and consider $\dd{P}_{\pi^{(n)}}$ as the restriction of $\dd{P}$ to $(\Omega,\sr{F}^{(n)}_{\infty})$. In this setting, we denote the expectation of a random variable $X$ {on $(\Omega,\sr{F}^{(n)}_{\infty})$} simply by $\dd{E}(X)$, instead of $\dd{E}_{\pi^{(n)}}(X)$.

Thus, we have the stationary framework on the stochastic basis $(\Omega,\sr{F},\dd{F},\dd{P})$. {In} this framework, $X^{(n)}(\cdot)$ is a stationary process consistent with the time-shift $\Theta_{\bullet}$. Then, counting processes $N^{(n)}_{e}(\cdot)$ and $N^{(n)}_{d}(\cdot)$ are also stationary processes consistent with $\Theta_{\bullet}$ from their {construction}. Let
\begin{align*}
 \alpha^{(n)}_{e} = \dd{E}[N^{(n)}_{e}(1)], \qquad \alpha^{(n)}_{d} = \dd{E}[N^{(n)}_{d}(1)].
\end{align*}
We note the following simple fact.
\begin{lemma}
 \label{lem:alpha 1}
 Under the conditions (\sect{preliminaries}.a)--(\sect{preliminaries}.e), $\alpha^{(n)}_{e}$ and $\alpha^{(n)}_{d}$ are bounded for each $n \ge 1$. If the condition (\sect{preliminaries}.f) is further assumed, then they are uniformly bounded for all $n \ge 1$.
\end{lemma}

\begin{proof}
 By the definition of $t^{(n)}_{e,j+1}$ and (\sect{preliminaries}.a), we have
 \begin{align*}
  t^{(n)}_{e,j+1} - t^{(n)}_{e,j} & = \inf\left\{t > 0; \int_{t^{(n)}_{e,j}}^{t^{(n)}_{e,j}+t} \lambda^{(n)}(L^{(n)}(s)) ds \ge T_{A}(j)\right\} \le  \frac 1{\lambda^{(n)}_{\sup}} T_{A}(j) < \infty,
 \end{align*}
by (\sect{preliminaries}.a). Hence, $t_{e,j} \ge (\lambda^{(n)}_{\sup})^{-1} \sum_{i=0}^{j-1} T_{A}(i)$, and $\lim_{j \to \infty} \frac 1j \sum_{i=0}^{j-1} T_{A}(i) = \dd{E}[T_{A}] = 1$ by the law of large numbers. Since $N^{(n)}_{e}(\cdot)$ is a stationary counting process, these facts yield
 \begin{align*}
  \alpha^{(n)}_{e} = \dd{E}[N^{(n)}_{e}(1)] & = \lim_{t \to \infty} {N^{(n)}_{e}}(t)/t = \lim_{j \to \infty} {N^{(n)}_{e}}(t^{(n)}_{e,j})/t^{(n)}_{e,j} = \lim_{j \to \infty} j/t^{(n)}_{e,j} \le \lambda^{(n)}_{\sup} < \infty.
 \end{align*}
 This implies $\alpha^{(n)}_{d} \le \alpha^{(n)}_{e} < \infty$ because $N^{(n)}_{d}(t) \le L^{(n)}(0) + N^{(n)}_{e}(t)$. We next assume (\sect{preliminaries}.f). Then, we can see that $\sup_{n \ge 1} \lambda^{(n)}_{\sup} < \infty$ because, for $\varepsilon> 0$, there is some $n_{0}$ such that
\begin{align*}
  \lambda^{(n)}_{\sup} = \sup_{u \in \dd{R}_{+}} \widehat{\lambda}^{(n)}(u) \le \varepsilon + \sup_{u \in \dd{R}_{+}} \lambda(u) < \infty, \qquad \forall n \ge n_{0},
\end{align*}
by \eq{limit-la-mu-n} and $\lambda^{(n)}_{\sup} < \infty$ for $1 \le n < n_{0}$ by (\sect{preliminaries}.a). Similarly, $\sup_{n \ge 1} \mu^{(n)}_{\sup} < \infty$ is proved.
\end{proof}

Since $\alpha^{(n)}_{e}$ and $\alpha^{(n)}_{d}$ are finite, we can define probability distributions $\dd{P}^{(n)}_{e}$ and $\dd{P}^{(n)}_{d}$ on $(\Omega,\sr{F})$ by
\begin{align}
 \label{eq:Palm 1}
 \dd{P}^{(n)}_{v}(A) = (\alpha^{(n)}_{v})^{-1} \dd{E}\left[ \int_{(0,1]} 1_{\Theta_s^{-1} A} N^{(n)}_{v}(ds) \right], \qquad A \in \sr{F}^{(n)}, v=e,d,
\end{align}
where $1_{A}$ is the indicator function of the event $A \in \sr{F}^{(n)}$. For $v=e,d$, $\dd{P}^{(n)}_{v}$ is called a Palm distribution {with respect to} the counting process $N^{(n)}_{v}$. The intuitive meaning of Palm distributions is briefly explained in \cite{Miya2025} (see \cite{Miya2024} for details).

We use the Palm distributions {to derive} the stationary equations as shown below. {We first note that} the stationarity of $X^{(n)}(\cdot)$ and the definition of the Palm distribution yield
\begin{align*}
 & \dd{E}\left[\int_{0}^{1} \sr{H}^{(n)} f(X^{(n)}(s)) ds\right] = \int_{0}^{1} \dd{E}\left[\sr{H}^{(n)} f(X^{(n)}(s))\right] ds = \dd{E}\left[\sr{H}^{(n)} f(X^{(n)}(0))\right],\\
 & \dd{E}\left[\sum_{s \in (0,1]} \Delta_{v} f(X^{(n)})(s)\right] = \alpha^{(n)}_{v} \dd{E}^{(n)}_{v}\left[\Delta_{v} f(X^{(n)})(0) \right], \qquad v= e,d,
\end{align*}
where $\dd{E}^{(n)}_{v}$ stands for the expectation under $\dd{P}^{(n)}_{v}$ for $v=e,d$. {Hence, taking the expectation of \eq{fXn-t} for $t=1$ with respect to $\dd{P}$}, we have a stationary equation:
\begin{align}
 \label{eq:BAR1}
 & \dd{E}\left[ \sr{H}^{(n)}f(X^{(n)}(0)) \right]  \nonumber\\
 & \quad + \alpha^{(n)}_{e} \dd{E}^{(n)}_{e}\left[\Delta_{e} f(X^{(n)})(0) \right] + \alpha^{(n)}_{d} \dd{E}^{(n)}_{d}\left[\Delta_{d} f(X^{(n)})(0) \right] = 0, \qquad f \in D_{b}^{p}(\dd{R}_{+}^{3}),
\end{align}
which is called a pre-limit BAR.

Since $X^{(n)}(\cdot)$ is a stationary process, we simply write $\dd{E}\left[ \sr{H}^{(n)}f(X^{(n)}(0)) \right]$ as $\dd{E}\left[ \sr{H}^{(n)}f(X^{(n)}) \right]$. The formula \eq{BAR1} is a special case of the rate conservation law (e.g., see \cite{Miya1994}), and referred to as a basic adjoint relationship, BAR for short.

\section{Main results}
\label{sec:main}
\setnewcounter

We are now ready to present {our} main results, which will be proved in \sectn{p-main}. Recall the partially diffusion-scaled process $\widehat{X}^{(n)}(\cdot)$ of \eq{h-LX}. In this section, we always assume the conditions (\sect{preliminaries}.a)--(\sect{preliminaries}.f). Under these conditions, $\widehat{X}^{(n)}(\cdot)$ has the unique stationary distribution for each $n \ge 0$ by \lem{Xn-stability}. Denote by $\widehat{X}^{(n)} \equiv (\widehat{L}^{(n)}, R^{(n)}_{e}, R^{(n)}_{d})$ a random vector subject to this stationary distribution.

Define the probability distribution $\nu^{(n)}$ on $\dd{R}_{+}$ as
\begin{align}
 \label{eq:hnu-B}
 \nu^{(n)}(B) = \dd{P}(\widehat{L}^{(n)} \in B), \qquad B \in \sr{B}_{+}.
\end{align}
We refer to this $\nu^{(n)}$ as the diffusion-scaled stationary distribution of the $n$-th model, and consider its sequence:
\begin{align*}
  \sr{S}_{sd} \equiv \{\nu^{(n)}; n \ge 1\},
\end{align*}
under the modeling conditions (\sect{preliminaries}.a)--(\sect{preliminaries}.e) and the heavy traffic condition (\sect{preliminaries}.f).

Recall $\beta(\cdot)$ of \eq{bsber-u}. Namely,
\begin{align*}
 \beta(u) = \frac {2 b(u)} {\sigma^{2}(u)}, \qquad u \in \dd{R}_{+},
\end{align*}
which is a key characteristic in our main results below. To consider the tightness of $\sr{S}_{sd}$, we assume the following condition in addition to the conditions (\sect{preliminaries}.a)--(\sect{preliminaries}.f).
\begin{itemize}
 \item [(\sect{preliminaries}.g)] There is a constant $b_{\infty} \in \dd{R}$ such that $\lim_{u \to \infty} b(u) = b_{\infty}$.
\end{itemize}
This condition is not hard to checked if the primitive data $\{\lambda^{(n)}(\cdot)); n \ge 1\}$ and $\{\mu^{(n)}(\cdot); n \ge 1\}$ are given, and one may guess that $b_{\infty} < 0$ implies the tightness of $\sr{S}_{sd}$. This is indeed true as shown below.

\begin{theorem}
 \label{thr:tight-2}
 For the state-dependent queue, assume the conditions (\sect{preliminaries}.a)--(\sect{preliminaries}.g), then $\sr{S}_{sd}$ is tight if $b_{\infty} < 0$.
\end{theorem}

\thr{tight-2} is proved in \sectn{p-tight-2}. For the state-dependent queue, we assume one more condition to find the weak limit of $\nu^{(n)}$ as $n \to \infty$. Namely,
\begin{itemize}
\item [(\sect{preliminaries}.h)]
If $\{\nu^{(n_{k})}; k \ge 1\}$ is a vaguely convergent subsequence of $\sr{S}_{sd}$, then its vague limit has a density.
\end{itemize}
This is the condition which we call the density condition in \sectn{introduction}.

\begin{theorem}
 \label{thr:SDQ-main-1}
 For the state-dependent queue satisfying the conditions (\sect{preliminaries}.a)--(\sect{preliminaries}.g), if (\sect{preliminaries}.h) holds, then $\nu^{(n)}$ converges weakly to the $\nu$ which has the unique density $h$ given by \eq{SDQ-h} below if and only if $b_{\infty} < 0$.
 \begin{align}
  \label{eq:SDQ-h}
  h(u) = \frac 1{C\sigma^{2}(u)} \exp\left(\int_{0}^{u} \beta(v) dv\right), \qquad u \in \dd{R}_{+},
 \end{align}
 where $C$ is a positive constant given by
\begin{align}
  \label{eq:C}
  C = \int_{0}^{\infty} \frac 1{\sigma^{2}(u)} \exp\left(\int_{0}^{u} \beta(v) dv\right) du.
 \end{align}
\end{theorem}

As discussed in \rem{SDQ-main}, we conjecture that the density condition (\sect{preliminaries}.h) can be removed from this theorem. For the level-dependent queue, this is the case, and much nicer results are obtained as shown below.

\begin{theorem}
 \label{thr:MLQ-main-1}
 For the multi-level queue satisfying the conditions (\sect{preliminaries}.a)--(\sect{preliminaries}.g), (\sect{preliminaries}.h) always holds, and therefore $\nu^{(n)}$ converges weakly to the $\nu$ of \thr{SDQ-main-1} if and only if $b_{\infty} < 0$.
\end{theorem}

Theorems \thrt{SDQ-main-1} and \thrt{MLQ-main-1} are proved in Sections \sect{p-SDQ-main-1} and \sect{p-MLQ-main-1}, respectively. 

\begin{remark}
\label{rem:SDQ-h}
From the proof of \thr{SDQ-main-1}, we can see that Theorems \thrt{SDQ-main-1} and \thrt{MLQ-main-1} hold without the condition (\sect{preliminaries}.g) if the condition $b_{\infty} < 0$ is replaced by $C < \infty$, where $C$ is defined by \eq{C}. However, the finiteness of $C$ may not be easily checked, so we have assumed (\sect{preliminaries}.g).
\end{remark}
 
\begin{remark}
 \label{rem:MLQ-main}
 For the multi-level queue which has finitely many levels, \citet{AtarMiya2026} prove that the process limit of the diffusion-scaled queue length process $\widetilde{L}^{(n)}(\cdot) \equiv \{n^{-1/2} L^{(n)}(nt); t \ge 0\}$ is the reflected diffusion $Z(\cdot)$ on $[0,\infty)$ which is obtained as the unique solution of the following stochastic {differential} equation, called SDE{:}
 \begin{align}
  \label{eq:SIE-Z}
  Z(t) = Z(0) + \int_{0}^{t} b(Z(s)) ds + \int_{0}^{t} \sigma(Z(s)) dW(s) + Y(t) \ge 0, \qquad t \ge 0,
 \end{align}
 where $W(\cdot)$ is the one-dimensional standard Brownian motion and $Y(\cdot)$ is the nondecreasing process {satisfying} $\int_{0}^{t} Z(s) dY(s) = 0$ for $t \ge 0$ {and $Y(0)=0$}. On the other hand, the stationary distribution of $Z(\cdot)$ is obtained in Theorem 3.2 of \cite{Miya2024b}, which is identical to the $\nu$ in \thr{SDQ-main-1} above. Hence, we have the limit interchange for the multi-level queue of \cite{AtarMiya2026}.
 \pend
\end{remark}
 
\begin{remark}\rm
\label{rem:SDQ-main}
From the proofs in \cite{AtarMiya2026}, we guess that the heavy-traffic limit of the queue length process for the state-dependent queue exists and is the solution of $Z(\cdot)$ of SDE \eq{SIE-Z}. Based on this heavy-traffic limit and the tightness of \thr{tight-2}, we expect that \thr{SDQ-main-1} is proved without the density condition (\sect{preliminaries}.h). However, their verifications require various techniques for process limits as in \cite{AtarMiya2026}. This is beyond the scope of this paper, so we leave them as conjectures.
\end{remark}

\section{Asymptotic BAR}
\label{sec:asym-BAR}
\setnewcounter

In this section, we derive the BAR for the $n$-th state-dependent queue, and consider its asymptotic form as $n$ goes to infinity, which is called an asymptotic BAR, and will be used to prove Theorems \thrt{SDQ-main-1} and \thrt{MLQ-main-1}. For deriving the asymptotic BAR, we take four steps, refining the BAR approach of \cite{Miya2025}. The first step considers some basic facts, the second step proves auxiliary lemmas for asymptotic expansions of several quantities, and the third and {fourth} steps compute the continuous and discontinuous components of the asymptotic BAR, respectively. Those four steps are given in Sections \sect{basic}, \sect{test}, \sect{continuous} and \sect{discontinuous}. In this section, $n$ is a fixed positive integer unless stated otherwise.

\subsection{Basic facts}
\label{sec:basic}

We consider some properties of $\alpha^{(n)}_{e}$ and $L^{(n)}(\cdot)$ under the Palm distributions. Similarly to Lemma 4.1 of \cite{Miya2025}, we have the following lemma. Recall that all the sample paths are right-continuous with left limits, and $X^{(n)} = (L^{(n)}, R^{(n)}_{e}, R^{(n)}_{d})$ is a random vector {subject to} the stationary distribution of $X^{(n)}(\cdot)$.
\begin{lemma}
 \label{lem:SDQ-basic 1}
 Under {the conditions} (\sect{preliminaries}.a)--(\sect{preliminaries}.f), we have $\alpha^{(n)}_{e} = \alpha^{(n)}_{d}$, and
 \begin{align}
  \label{eq:SDQ-balance 1}
  & \dd{P}^{(n)}_{e}[L^{(n)}(0-) = \ell] = \dd{P}^{(n)}_{d}[L^{(n)}(0) = \ell], \qquad \ell \ge 0,\\
  \label{eq:SDQ-alpha 1}
  & \alpha^{(n)}_{e} = \sum_{\ell = 0}^{\infty} \lambda^{(n)}(\ell) \dd{P}\left[L^{(n)} = \ell \right],\\
  \label{eq:SDQ-alpha 2}
  & \alpha^{(n)}_{d} = \sum_{\ell = 1}^{\infty} \mu^{(n)}(\ell) \dd{P}\left[L^{(n)} = \ell \right].
 \end{align}
\end{lemma}

\begin{proof}
 For each integer $\ell \ge 0$, let $f(\vc{x}) = x_{1} \wedge (\ell+1)$ for $\vc{x} \in \dd{R}_{+}^{3}$, then obviously $f \in D_{b}^{p}(\dd{R}_{+}^{3})$. Applying this $f$ to \eq{BAR1}, we have
 \begin{align}
  \label{eq:SDQ-alpha}
  \alpha^{(n)}_{e} \dd{P}^{(n)}_{e}[L^{(n)}(0-) \le \ell] = \alpha^{(n)}_{d} \dd{P}^{(n)}_{d}[L^{(n)}(0) \le \ell], \qquad \ell \ge 0{,}
 \end{align}
 because $\sr{H}^{(n)}f(X^{(n)}(0)) = 0$ and
 \begin{align*}
  \Delta_{e} f(X^{(n)})(0) & = L_{1}^{(n)}(0) \wedge (\ell+1) - L^{(n)}(0-) \wedge (\ell+1) \\
  & = 1(L^{(n)}(0-) \le \ell, \Delta N^{(n)}_{e}(0) = 1),\\
  \Delta_{d} f(X^{(n)})(0) & = L^{(n)}(0) \wedge (\ell+1) - L^{(n)}_{1}(0) \wedge (\ell+1) \\
  & = - 1(L^{(n)}(0) \le \ell, \Delta N^{(n)}_{d}(0) = 1).
 \end{align*}
 Letting $\ell \to \infty$ in \eq{SDQ-alpha}, we have $\alpha^{(n)}_{e} = \alpha^{(n)}_{d}$. Applying this fact to \eq{SDQ-alpha}, we have \eq{SDQ-balance 1}. We next let $f(\vc{x}) = x_{2}$, and apply it to \eq{BAR1}, then
 \begin{align*}
  - \sum_{\ell = 0}^{\infty} \lambda^{(n)}(\ell) \dd{P}\left[L^{(n)}(0) = \ell\right] + \alpha^{(n)}_{e} = 0
 \end{align*}
 because $\dd{E}[T_{A}] = 1$. Hence, we have \eq{SDQ-alpha 1} because $L^{(n)}(0)$ has the same distribution as $L^{(n)}$ under $\dd{P}$. Similarly, we have \eq{SDQ-alpha 2} from \eq{BAR1} for $f(\vc{x}) = x_{3}$.
\end{proof}

In what follows, we always replace $\alpha^{(n)}_{d}$ by $\alpha^{(n)}_{e}$ because $\alpha^{(n)}_{{d}} = \alpha^{(n)}_{e}$ by \lem{SDQ-basic 1}. This simplifies the BAR \eq{BAR1}.

\subsection{Test functions and pre-limit BAR}
\label{sec:test}

Recall that the Markov process $X^{(n)}(\cdot)$ describes the $n$-th state-dependent queue. Recall that $X^{(n)} \equiv (L^{(n)}, R^{(n)}_{e}, R^{(n)}_{d})$ is a random vector subject to the stationary distribution of $X^{(n)}(\cdot)$, and $\widehat{X}^{(n)} \equiv (\widehat{L}^{(n)}, R^{(n)}_{e}, R^{(n)}_{d})$ is defined with $\widehat{L}^{(n)} = n^{-1/2} L^{(n)}$. To prove \thr{SDQ-main-1}, we compute a BAR for $\widehat{X}^{(n)}(\cdot)$, called the $n$-th pre-limit BAR, and derive the weak limit of its stationary distribution $\nu^{(n)}$ of $\widehat{L}^{(n)}$ as $n \to \infty$. 

A key to this BAR approach is a suitable choice of a test function for the BAR \eq{BAR1}. To this end, we take an exponential-type test function following \cite{Miya2025}, which is originally introduced in \cite{Miya2017}. For this test function, we first define a function $g^{(n)}_{\theta}$ for $\theta \in \dd{R}$ as
\begin{align*}
 & g^{(n)}_{\theta}(\vc{v}) = e^{ - \eta^{(n)}(\theta) (v_{1} \wedge n^{1/2}) - \zeta^{(n)}(\theta) (v_{2} \wedge n^{1/2})}, \qquad \vc{v}= (v_{1},v_{2}) \in \dd{R}_{+}^{2},
\end{align*}
where $\eta^{(n)}(\theta)$ {and} $\zeta^{(n)}(\theta)$ are the solutions of the following equations for each $\theta \in \dd{R}${:}
\begin{align}
 \label{eq:SDQ-boundary1}
 e^{\theta} \dd{E}\left[e^{-\eta^{(n)}(\theta) (T_{A} \wedge n^{1/2})}\right] = 1, \qquad e^{-\theta} \dd{E}\left[e^{-\zeta^{(n)}(\theta) (T_{S} \wedge n^{1/2})}\right] = 1.
\end{align}
These solutions exist and are finite because $\eta^{(n)}(\theta)$ is the inverse function of the Laplace transform of the finite positive random variable $T_{A} \wedge n^{1/2}${,} which takes the value $e^{-\theta}$. Similarly, $\zeta^{(n)}(\theta)$ is determined.

Let us consider the asymptotic behavior of $\eta^{(n)}(\theta)$ and $\zeta^{(n)}(\theta)$. Since they are {infinitely differentiable} functions of $\theta$, by twice differentiating the equations in \eq{SDQ-boundary1}, the following lemma can be obtained by their Taylor expansions around the origin. Its proof can be found in \cite[Lemma 5.8]{BravDaiMiya2024}.
\begin{lemma}
 \label{lem:SDQ-eta-zeta1}
 Under {the conditions} (\sect{preliminaries}.b)--(\sect{preliminaries}.c), we have the following expansions as $n^{-1/2} \theta  \to 0${:}
 \begin{align}
  \label{eq:SDQ-eta1}
  & \eta^{(n)}(n^{-1/2}\theta) = n^{-1/2} \theta  + \frac 12 \sigma_{A}^{2} n^{-1} \theta^{2} + o(n^{-1} \theta^{2}),\\
  \label{eq:SDQ-zeta1}
  & \zeta^{(n)}(n^{-1/2}\theta) = - n^{-1/2} \theta  + \frac 12 \sigma_{S}^{2} n^{-1} \theta^{2} + o(n^{-1} \theta^{2}).
 \end{align}
 Furthermore, there are constants $d_{A}, d_{S}, a > 0$ such that, for $n \ge 1$ and $\theta \in \dd{R}$ satisfying $n^{-1/2}|\theta| < a$,
 \begin{align}
  \label{eq:SDQ-eta-zeta1}
  \big|\eta^{(n)}(n^{-1/2}\theta) (u_{1} \wedge n^{1/2}) & + \zeta^{(n)}(n^{-1/2}\theta) (u_{2} \wedge n^{1/2}) \big| \nonumber\\
  & \le |\theta| \left(d_{A} (n^{-1/2} u_{1} \wedge 1) + d_{S} (n^{-1/2}u_{2} \wedge 1)\right).
 \end{align}
\end{lemma}

By this lemma, $g^{(n)}_{n^{-1/2}\theta}(R^{(n)})$ is bounded by $e^{|\theta|(d_{A} + d_{S})}$ for sufficiently large $n$ for each $\theta \in \dd{R}$, where {we} recall that $R^{(n)} = (R^{(n)}_{e},R^{(n)}_{d})$. Hence, from the fact that
\begin{align*}
  \lim_{n \to \infty} g^{(n)}_{n^{-1/2}\theta}(R^{(n)}) & = \lim_{n \to \infty} e^{ - \eta^{(n)}(n^{-1/2}\theta) (R^{(n)}_{e} \wedge n^{1/2}) - \zeta^{(n)}(n^{-1/2}\theta) (R^{(n)}_{d} \wedge n^{1/2})} \nonumber\\
  & = \lim_{n \to \infty} e^{ - \theta (n^{-1/2}R^{(n)}_{e} \wedge 1) + \theta (n^{-1/2} R^{(n)}_{d} \wedge 1)} = 1,
\end{align*}
we have the following lemma.

\begin{lemma}
 \label{lem:SDQ-gn-limit}
 $g^{(n)}_{n^{-1/2}\theta}(R^{(n)})$ converges to $1$ as $n \to \infty$ for each $\theta \in \dd{R}$, and
 \begin{align}
  \label{eq:SDQ-Egn-limit}
  & \lim_{n \to \infty} \dd{E}\left[g^{(n)}_{n^{-1/2}\theta}(R^{(n)})\right] = 1, \qquad \lim_{n \to \infty} \dd{E}_{u}\left[g^{(n)}_{n^{-1/2}\theta}(R^{(n)})\right] = 1, \quad u = e, d.
 \end{align}
\end{lemma}

We now define the test functions $f^{(n)}_{\theta,I}$ for $\theta \in \dd{R}$ and interval $I \subset \dd{R}_{+}$.
\begin{align*}
   & f^{(n)}_{\theta,I}(u,\vc{v}) = e^{\theta u} g^{(n)}_{\theta}(v_{1},v_{2}) 1(n^{-1/2} u \in I), \qquad (u,\vc{v}) = (u,v_{1},v_{2}) \in \dd{R}_{+}^{3},
\end{align*}
and introduce the following notation for the Palm expectation parts for $f = f^{(n)}_{\theta,I}(u,\vc{v})$ in the BAR \eq{BAR1}.
\begin{align}
\label{eq:EnDI}
 & E^{(n)}_{\Delta,I}(\theta) = \dd{E}^{(n)}_{e}\left[\Delta_{e} f^{(n)}_{n^{-1/2}\theta,I}(X^{(n)})(0) \right] + \dd{E}^{(n)}_{d}\left[\Delta_{d} f^{(n)}_{n^{-1/2}\theta,I}(X^{(n)})(0) \right].
\end{align}

We will use these function and notation for $I = [0,x]$ and $I = (x,y]$ for $x, y \in \dd{R}_{+}$ satisfying $x < y$. Note that the diffusion-scaled $u$, that is, $n^{-1/2} u$, is limited to $u \in I$ in $f^{(n)}_{\theta,I}(u,\vc{v})$. These classes of test functions without the limitations by $I$ are typically used in the BAR approach (e.g., see \cite{BravDaiMiya2017,Miya2017}). We will see that this limitations is quite useful for the state-dependent queue.

Our first task is to compute the BAR \eq{BAR1} for the test function $f = f^{(n)}_{n^{-1/2}\theta,I}$ with $I=[0,x]$ and $I=(x,y]$ for $0 \le x < y$. Namely,
\begin{align}
 \label{eq:SDQ-BAR1}
 \dd{E}\left[ \sr{H}^{(n)}f^{(n)}_{n^{-1/2}\theta,[0,x]}(X^{(n)}(0)) \right] + \alpha^{(n)}_{e} E^{(n)}_{\Delta,[0,x]}(\theta) = 0, \qquad \theta \in \dd{R}.
\end{align}
This pre-limit BAR is the starting point of our analysis. We derive its asymptotic form. For this, we use the following simplified notations. For $n \ge 1$ and a nonnegative random variable $X$, let $\widecheck{X} = X \wedge n^{1/2}$. For example,
\begin{align*}
 & \widecheck{T}^{(n)}_{A} = T_{A} \wedge n^{1/2}, \qquad \widecheck{T}^{(n)}_{S} = T_{S} \wedge n^{1/2}, \\
 & \widecheck{R}^{(n)}_{e}(0) = R^{(n)}_{e}(0) \wedge n^{1/2}, \qquad \widecheck{R}^{(n)}_{d}(0-) = R^{(n)}_{d}(0-) \wedge n^{1/2}.
\end{align*}

Then, we have an asymptotic BAR in the following lemma, which will be used in the proofs of Theorems \thrt{SDQ-main-1} and \thrt{MLQ-main-1}. Recall the definition \eq{qn-x} of $q^{(n)}_{x}$.
 
\begin{lemma}[Asymptotic BAR]
 \label{lem:asym-BAR1}
 The following asymptotic formula is obtained uniformly in $x$ in each finite interval on the nonnegative half-line. For each $\theta \le 0$ and $x \ge 0$, as $n \to \infty$,
 \begin{align}
  \label{eq:asym-BAR=<1}
  & \frac 12 \dd{E}\left[ (\widehat{\sigma}^{(n)})^{2}(\widehat{L}^{(n)}) \left(\widehat{\beta}^{(n)}(\widehat{L}^{(n)}) + {\theta} \right) e^{\theta \widehat{L}^{(n)}} 1(\widehat{L}^{(n)} \le x) \right] + \mu^{(n)}(0) n^{1/2} \dd{P}[\widehat{L}^{(n)} = 0] \nonumber\\
  & \quad = - e^{n^{-1/2} q^{(n)}_{x} {\theta}} \{\dd{E}[\widehat{b}^{(n)}(\widehat{L}^{(n)}) 1(\widehat{L}^{(n)} > x)] + \theta H^{(n)}_{x}\}+ \theta o(1),\\
  \label{eq:asym-BAR>1}
  & \frac 12 \dd{E}\left[ (\widehat{\sigma}^{(n)})^{2}(\widehat{L}^{(n)}) \left(\widehat{\beta}^{(n)}(\widehat{L}^{(n)}) + {\theta} \right) e^{\theta \widehat{L}^{(n)}} 1(\widehat{L}^{(n)} > x) \right] \nonumber\\
  & \quad = e^{n^{-1/2} q^{(n)}_{x} {\theta}} \{\dd{E}[\widehat{b}^{(n)}(\widehat{L}^{(n)}) 1(\widehat{L}^{(n)} > x)] + \theta H^{(n)}_{x}\}+ \theta o(1),
 \end{align}
 where
 \begin{align}
  \label{eq:Hn-x}
  H^{(n)}_{x} = 2^{-1} \alpha^{(n)}_{e} \big\{&\dd{E}^{(n)}_{e}[(\sigma_{S}^{2} \widecheck{R}^{(n)}_{d}(0-) - (\widecheck{R}^{(n)}_{d}(0-))^{2}) 1(L^{(n)}(0-) = q^{(n)}_{x})] \nonumber\\
  & -\dd{E}^{(n)}_{d}[((\sigma_{A}^{2} + 2) \widecheck{R}^{(n)}_{e}(0) - (\widecheck{R}^{(n)}_{e}(0))^{2}) 1(L^{(n)}(0) = q^{(n)}_{x})].
 \end{align}
\end{lemma}

This lemma is immediate from Lemmas \lemt{BAR=<P1} and \lemt{BAR=<P2} in Sections \sect{continuous} and \sect{discontinuous}, respectively, which correspond to the continuous and discontinuous parts of the pre-limit BAR \eq{SDQ-BAR1}, respectively. To prove these  lemmas, we prepare two auxiliary lemmas on asymptotic properties of $R^{(n)}_{z}$ and $\widecheck{R}^{(n)}_{z}$ for $z=e,d$ and $\widecheck{T}^{(n)}_{z}$ for $z=A,S$. 

\begin{lemma}
 \label{lem:SDQ-moment}
 For each fixed $n \ge 1$, 
 \begin{align}
  \label{eq:SDQ-Re 1}
  & \dd{P}[R^{(n)}_{e} > y] \le (\lambda^{(n)}_{\inf})^{-1} \alpha^{(n)}_{e} \dd{E}[(T_{A}-y) 1(T_{A} > y)] = o(y^{-2}),\\
  \label{eq:SDQ-Rd 1}
  & \dd{P}[R^{(n)}_{d} > y] \le (\mu^{(n)}_{\inf})^{-1} \left( \mu^{(n)}(0) \dd{P}[T_{S} > y] + \alpha^{(n)}_{e} \dd{E}[(T_{S}-y) 1(T_{S} > y)]\right) = o(y^{-2}),
 \end{align}
 and therefore, for $j =1,2$ and $n \ge 1$,
 \begin{align}
  \label{eq:SDQ-Re 2}
  & (j+1) \dd{E}[(R^{(n)}_{e})^{j}] \le (\lambda^{(n)}_{\inf})^{-1} \alpha^{(n)}_{e} \dd{E}[(T_{A})^{j+1}], \\
  \label{eq:SDQ-Rd 2}
  & (j+1) \dd{E}[(R^{(n)}_{d})^{j}] \le (j+1) \dd{E}[T_{S}^{j}] + (\mu^{(n)}_{\inf})^{-1} \alpha^{(n)}_{e} \dd{E}[T_{S}^{j+1}].
 \end{align}
 In particular, $\{(R^{(n)}_{e})^{j}; n \ge 1\}$ and $\{(R^{(n)}_{d})^{j} 1(L^{(n)}(0) \ge 1); n \ge 1\}$ are uniformly integrable for $j=1,2$ by \eq{SDQ-Re 2} and \eq{SDQ-Rd 2}.
\end{lemma}
This lemma easily follows from \lem{alpha 1} and the BAR \eq{BAR1} by appropriately chosen test functions. For example, we choose $f(u,\vc{y}) = (y_{1} - w) 1(y_{1} \ge w)$, which satisfies {the condition} \eq{test-f}. Applying this $f$ to \eq{BAR1}, we have
\begin{align*}
 \dd{E}[\lambda^{(n)}(L^{(n)}) 1(R^{(n)}_{e} > w)] = \alpha^{(n)}_{e} \dd{E}[(T_{A}-w) 1(T_{A} > w)],
\end{align*}
which yields \eq{SDQ-Re 1}. Similarly, \eq{SDQ-Rd 1} is proved. 

The following lemma is essentially the same as Lemma 4.4 of \cite{Miya2025}.

\begin{lemma}
 \label{lem:SDQ-t-moment}
 For $j=0,1,2$, as $n \to \infty$,
 \begin{align}
  \label{eq:SDQ-Tu 1}
  & |\dd{E}[(\widecheck{T}^{(n)}_{z})^{j+1}] - \dd{E}[T_{z}^{j+1}]| \le \dd{E}[T_{z}^{j+1} 1(T_{z} > n^{1/2})] = o(n^{-(2-j)/2}), \quad z = A,S,\\
  \label{eq:SDQ-Rv 1}
  & |\dd{E}[(\widecheck{R}^{(n)}_{z})^{j}] - \dd{E}[(R^{(n)}_{z})^{j}]| \le \dd{E}[(R^{(n)}_{z})^{j} 1(R^{(n)}_{z} > n^{1/2})] = o(n^{-(2-j)/2}), \quad z = e, d.
 \end{align}
\end{lemma}

\subsection{Continuous part of the pre-limit BAR}
\label{sec:continuous}

In this section, we compute the continuous part $\dd{E}\left[ \sr{H}^{(n)}f^{(n)}_{n^{-1/2}\theta,[0,x]}(X^{(n)}(0))\right]$ of the pre-limit BAR \eq{SDQ-BAR1}, and derive its asymptotic form. Let
\begin{align}
\label{eq:r-ed}
  r^{(n)}_{e} = 1(R^{(n)}_{e} > n^{1/2}), \qquad r^{(n)}_{d} = 1(R^{(n)}_{d} > n^{1/2}), 
\end{align}
then
\begin{align}
 \label{eq:SDQ-H1}
 & \dd{E}\left[ \sr{H}^{(n)}f^{(n)}_{n^{-1/2}\theta,[0,x]}(X^{(n)})\right] \nonumber\\
 & \quad = \dd{E} \Big[\Big(\lambda^{(n)}(L^{(n)}) (1-r^{(n)}_{e}) \eta^{(n)}(n^{-1/2}\theta) \nonumber\\
 & \hspace{9ex} + \mu^{(n)}(L^{(n)}) (1-r^{(n)}_{d}) \zeta^{(n)} (n^{-1/2}\theta)\Big) e^{n^{-1/2} \theta  L^{(n)}} 1(n^{-1/2} L^{(n)} \le x) g^{(n)}_{n^{-1/2}\theta}(R^{(n)}) \Big] \nonumber\\
 & \qquad - \mu^{(n)}(0) \zeta^{(n)}(n^{-1/2}\theta) \dd{E}\left[(1-r^{(n)}_{d}) 1(L^{(n)}=0) g^{(n)}_{n^{-1/2} \theta}(R^{(n)}) \right].
\end{align}
Since $\widehat{L}^{(n)} = n^{-1/2} L^{(n)}$, $\widehat{\lambda}^{(n)}(\widehat{L}^{(n)}) = \lambda^{(n)}(L^{(n)})$ and $\widehat{\mu}^{(n)}(\widehat{L}^{(n)}) = \mu^{(n)}(L^{(n)})$, \eq{SDQ-H1} can be written as
\begin{align}
 \label{eq:SDQ-H2}
 & \dd{E}\left[ \sr{H}^{(n)}f^{(n)}_{n^{-1/2}\theta,[0,x]}(X^{(n)})\right] \nonumber\\
 & \quad = \dd{E} \Big[\Big(\widehat{\lambda}^{(n)}(\widehat{L}^{(n)}) (1-r^{(n)}_{e}) \eta^{(n)}(n^{-1/2}\theta) \nonumber\\
 & \hspace{9ex} + \widehat{\mu}^{(n)}(\widehat{L}^{(n)}) (1-r^{(n)}_{d}) \zeta^{(n)} (n^{-1/2}\theta)\Big) e^{\theta \widehat{L}^{(n)}} 1(\widehat{L}^{(n)} \le x) g^{(n)}_{n^{-1/2}\theta}(R^{(n)}) \Big] \nonumber\\
 & \qquad - \mu^{(n)}(0) \zeta^{(n)}(n^{-1/2}\theta) \dd{E}\left[(1-r^{(n)}_{d}) 1(\widehat{L}^{(n)}=0) g^{(n)}_{n^{-1/2} \theta}(R^{(n)}) \right], \qquad x \ge 0,
\end{align}
where we have used the fact that $\{L^{(n)} \le q^{(n)}_{x}\} = \{\widehat{L}^{(n)} \le x\}$. Let 
\begin{align}
\label{eq:En-ed}
  \sr{D}^{(n)}_{ed}(\theta) & = \widehat{\lambda}^{(n)}(\widehat{L}^{(n)}) r^{(n)}_{e} \eta^{(n)}(n^{-1/2}\theta) + \widehat{\mu}^{(n)}(\widehat{L}^{(n)}) r^{(n)}_{d} \zeta^{(n)} (n^{-1/2}\theta).
\end{align}
Since {$\widehat{\lambda}^{(n)}(u) - \widehat{\mu}^{(n)}(u) = n^{-1/2} \widehat{b}^{(n)}(u)$}, \lem{SDQ-eta-zeta1} yields, for each $\theta \le 0$,
\begin{align*}
 & \Big(\widehat{\lambda}^{(n)}(\widehat{L}^{(n)}) (1-r^{(n)}_{e}) \eta^{(n)}(n^{-1/2}\theta) + \widehat{\mu}^{(n)}(\widehat{L}^{(n)}) (1-r^{(n)}_{d}) \zeta^{(n)} (n^{-1/2}\theta)\Big) \nonumber\\
 & \quad = \Big(n^{-1/2} \widehat{b}^{(n)}(\widehat{L}^{(n)}) n^{-1/2} \theta - \sr{D}^{(n)}_{ed}(\theta) \nonumber\\
 & \qquad + \frac 12 \left(\widehat{\lambda}^{(n)}(\widehat{L}^{(n)}) \sigma_{A}^{2} + \widehat{\mu}^{(n)}(\widehat{L}^{(n)}) \sigma_{S}^{2}\right) \theta^{2} n^{-1} + o(\theta^{2} n^{-1}) \Big) \nonumber\\
 & \quad = \Big(\frac 12 (\widehat{\sigma}^{(n)})^{2}(\widehat{L}^{(n)}) \left(\widehat{\beta}^{(n)}(\widehat{L}^{(n)}) + \theta\right) + \theta o(1) \Big) \theta n^{-1} - \sr{D}^{(n)}_{ed}(\theta), \quad n \to \infty.
\end{align*}
Hence, from \eq{SDQ-H2}, we have
\begin{align}
 \label{eq:SDQ=<H3}
 & \dd{E}\left[\sr{H}^{(n)}f^{(n)}_{n^{-1/2}\theta,[0,x]}(X^{(n)})\right] + \dd{E}\left[\sr{D}^{(n)}_{ed}(\theta) e^{\theta \widehat{L}^{(n)}} 1(\widehat{L}^{(n)} \le x) g^{(n)}_{n^{-1/2}\theta}(R^{(n)})\right] \nonumber\\
 & \quad = \frac 12 \dd{E}\left[ \left((\widehat{\sigma}^{(n)})^{2}(\widehat{L}^{(n)}) \left(\widehat{\beta}^{(n)}(\widehat{L}^{(n)}) + {\theta} \right) + \theta o(1) \right) e^{\theta \widehat{L}^{(n)}} g^{(n)}_{n^{-1/2} \theta}(R^{(n)}) 1(\widehat{L}^{(n)} \le x)\right] {\theta} n^{-1} \nonumber\\
 & \qquad + \left(n^{-1/2} \theta  - \frac 12 \sigma_{S}^{2} n^{-1} \theta^{2}  + o(\theta^{2} n^{-1}) \right) \mu^{(n)}(0) \nonumber\\
 & \qquad \hspace{15ex} \times \dd{E}\left[(1-r^{(n)}_{d}) 1(L^{(n)}=0) g^{(n)}_{n^{-1/2} \theta}(R^{(n)})\right].
\end{align}
Similarly,
\begin{align}
 \label{eq:SDQ>H3}
 & \dd{E}\left[\sr{H}^{(n)}f^{(n)}_{n^{-1/2}\theta,(x,\infty)]}(X^{(n)})\right] + \dd{E}\left[\sr{D}^{(n)}_{ed}(\theta) e^{\theta \widehat{L}^{(n)}} 1(\widehat{L}^{(n)} > x) g^{(n)}_{n^{-1/2}\theta}(R^{(n)})\right] \nonumber\\
 & \quad = \frac 12 \dd{E}\Big[ \left((\widehat{\sigma}^{(n)})^{2}(\widehat{L}^{(n)}) \left(\widehat{\beta}^{(n)}(\widehat{L}^{(n)}) + {\theta} \right) + \theta o(1) \right) \nonumber\\
 & \qquad \hspace{15ex} \times  e^{\theta \widehat{L}^{(n)}} g^{(n)}_{n^{-1/2} \theta}(R^{(n)}) 1(\widehat{L}^{(n)} \le x)\Big] {\theta} n^{-1}.
\end{align}

Since $r^{(n)}_{e} = o(n^{-1})$ and $r^{(n)}_{d} = o(n^{-1})$ by \eq{SDQ-Re 1} and \eq{SDQ-Rd 1} in \lem{SDQ-moment} and $\widehat{\lambda}^{(n)}(\widehat{L}^{(n)}) = \lambda^{(n)}(L^{(n)})$ and $\widehat{\mu}^{(n)}(\widehat{L}^{(n)}) = \mu^{(n)}(L^{(n)})$ are uniformly bounded by \lem{h-limit},
\begin{align*}
  \dd{E}\left[\sr{D}^{(n)}_{ed}(\theta) e^{\theta \widehat{L}^{(n)}} 1(\widehat{L}^{(n)} \le x) g^{(n)}_{n^{-1/2}\theta}(R^{(n)})\right] = \theta o(n^{-1}).
\end{align*}
by \eq{SDQ-eta-zeta1}. Then, we have the following lemma.

\begin{lemma}[Continuous part of the asymptotic BAR]
 \label{lem:BAR=<P1}
 For $0 \le x < y$, as $n \to \infty$,
\begin{align}
\label{eq:BAR=<P1}
  & \dd{E}\left[\sr{H}^{(n)}f^{(n)}_{n^{-1/2}\theta,[0,x]}(X^{(n)})\right]  \nonumber\\
  & \quad = \frac 12 \dd{E}\left[ \left((\widehat{\sigma}^{(n)})^{2}(\widehat{L}^{(n)}) \left(\widehat{\beta}^{(n)}(\widehat{L}^{(n)}) + {\theta} \right) + \theta o(1) \right) e^{\theta \widehat{L}^{(n)}} (1+o(1)) 1(\widehat{L}^{(n)} \le x)\right] {\theta} n^{-1} \nonumber\\
  & \qquad + \mu^{(n)}(0) (1+ \theta o(1)) n^{1/2} \dd{P}[\widehat{L}^{(n)} = 0] {\theta} n^{-1}, \qquad \theta \in \dd{R}, \\
\label{eq:BAR>P1}
  & \dd{E}\left[\sr{H}^{(n)}f^{(n)}_{n^{-1/2}\theta,(x,\infty)}(X^{(n)})\right]  \nonumber\\
  & \quad = \frac 12 \dd{E}\Big[ \left((\widehat{\sigma}^{(n)})^{2}(\widehat{L}^{(n)}) \left(\widehat{\beta}^{(n)}(\widehat{L}^{(n)}) + \theta \right) + \theta o(1) \right) \nonumber\\
 & \qquad \hspace{20ex} \times  e^{\theta \widehat{L}^{(n)}} (1+o(1)) 1(\widehat{L}^{(n)} > x)\Big] \theta n^{-1}, \qquad \theta \le 0.
\end{align}
\end{lemma}
\begin{proof}
\eq{BAR=<P1} is obtained by applying \lem{SDQ-gn-limit} to \eq{SDQ=<H3}. Similarly, \eq{BAR>P1} is obtained by \eq{SDQ>H3}.
\end{proof}

\subsection{Discontinuous part of the pre-limit BAR}
\label{sec:discontinuous}

We next compute the Palm expectation terms in \eq{SDQ-BAR1} for $f = f^{(n)}_{n^{-1/2}\theta,I}$ with $I = [0,x]$ and $I = (x,y]$. Recall the definition \eq{EnDI} of $E^{(n)}_{\Delta,I}(\theta)$. Note that $\{\widehat{L}^{(n)} \le x\} = \{L^{(n)} \le q^{(n)}_{x}\}$ by \eq{hLn-Ln}. Let $R^{(n)}_{1}(t) = (R^{(n)}_{e,1}(t), R^{(n)}_{d,1}(t))$, which is the intermediate state of the nominal residual times. From the definition $\eta^{(n)}$ of \eq{SDQ-boundary1}, we have, under $\dd{P}^{(n)}_{e}$, $L^{(n)}_{1}(0) = L^{(n)}(0-) + 1$ and
\begin{align}
\label{eq:fn-X1}
  & e^{\theta n^{-1/2} L^{(n)}_{1}(0)} \dd{E}^{(n)}_{e} \left[e^{-\eta^{(n)}(n^{-1/2} \theta) (T_{A} \wedge n^{1/2})}\right] e^{-\zeta^{(n)}(n^{-1/2} \theta) (R^{(n)}_{d,1}(0) \wedge n^{1/2})} \nonumber\\
  & \quad = e^{\theta n^{-1/2} L^{(n)}(0-)} e^{-\zeta^{(n)}(n^{-1/2} \theta) (R^{(n)}_{d}(0-) \wedge n^{1/2})} = e^{\theta n^{-1/2} L^{(n)}(0-)} g_{n^{-1/2} \theta} (R^{(n)}(0-)),\\
\label{eq:fn-X-}
 & 1(L^{(n)}_{1}(0) \le n^{1/2} x, L^{(n)}(0-) \not= q^{(n)}_{x}) = 1(L^{(n)}_{1}(0) \le n^{1/2} x, L^{(n)}(0-) \le n^{1/2} x - 1)  \nonumber\\
 & \quad = 1(L^{(n)}(0-) \le n^{1/2} x - 1) = 1(L^{(n)}(0-) \le n^{1/2} x, L^{(n)}(0-) \not= q^{(n)}_{x}).
\end{align}
Multiplying \eq{fn-X1} and \eq{fn-X-} side by side and taking their expectations under $\dd{P}^{(n)}_{e}$, from the definitions of $\Delta_{e}$ for $f^{(n)}_{n^{-1/2}\theta,[0,x]}$, we have
\begin{align*}
 & \dd{E}^{(n)}_{e}\left[ \Delta_{e} f^{(n)}_{n^{-1/2}\theta,[0,x]}(X^{(n)})(0) 1(L^{(n)}(0-) \not= q^{(n)}_{x})\right] = 0.
\end{align*}
Similarly, 
\begin{align*}
 & \dd{E}^{(n)}_{d}\left[\Delta_{d} f^{(n)}_{n^{-1/2}\theta,[0,x]}(X^{(n)})(0) 1(L^{(n)}(0) \not= q^{(n)}_{x})\right] = 0.
\end{align*}
Thus, for $x \ge 0$, define
\begin{align*}
 & E^{(n)}_{\Delta,e,[0,x]}(\theta) = \dd{E}^{(n)}_{e}\left[\Delta_{e} f^{(n)}_{n^{-1/2}\theta,[0,x]}(X^{(n)})(0) 1(L^{(n)}(0-) = q^{(n)}_{x}) \right],\\
 & E^{(n)}_{\Delta,d,[0,x]}(\theta) = \dd{E}^{(n)}_{d}\left[\Delta_{d} f^{(n)}_{n^{-1/2}\theta,[0,x]}(X^{(n)})(0) 1(L^{(n)}(0) = q^{(n)}_{x}) \right],
\end{align*}
then
\begin{align}
 \label{eq:SDQ=<df}
 E^{(n)}_{\Delta,[0,x]}(\theta) = E^{(n)}_{\Delta,e,[0,x]}(\theta) + E^{(n)}_{\Delta,d,[0,x]}(\theta), \qquad \theta \in \dd{R}, \ x \ge 0.
\end{align}

From the definitions of $X^{(n)}_{1}(0)$ and $\Delta_{e} f^{(n)}_{n^{-1/2}\theta,[0,x]}(X^{(n)})(0)$, we have
\begin{align}
 \label{eq:SDQ=<df1}
 & E^{(n)}_{\Delta,e,[0,x]}(\theta) = \dd{E}^{(n)}_{e}[(f^{(n)}_{n^{-1/2}\theta,[0,x]}(X^{(n)}_{1}(0)) - f^{(n)}_{n^{-1/2}\theta,[0,x]}(X^{(n)}(0-))) 1(L^{(n)}(0-) = q^{(n)}_{x})] \nonumber\\
 & \quad = \dd{E}^{(n)}_{e} \Big[e^{n^{-1/2}\theta (q^{(n)}_{x} + 1)} g^{(n)}_{n^{-1/2}\theta}(R^{(n)}_{1}(0))1(L^{(n)}_{1}(0) \le q^{(n)}_{x}, L^{(n)}(0-) = q^{(n)}_{x})  \nonumber\\
 & \hspace{15ex} - e^{n^{-1/2}{\theta} q^{(n)}_{x}} g^{(n)}_{n^{-1/2}{\theta}}(R^{(n)}(0-)) 1(L^{(n)}(0-) \le q^{(n)}_{x}, L^{(n)}(0-) = q^{(n)}_{x}) \Big] \nonumber\\
 & \quad = - e^{n^{-1/2}{\theta} q^{(n)}_{x}} \dd{E}^{(n)}_{e}[e^{-\zeta^{(n)}(n^{-1/2}{\theta}) \widecheck{R}^{(n)}_{d}(0-)} 1(L^{(n)}(0-) = q^{(n)}_{x})],
\end{align}
where the second equality is obtained applying \eq{SDQ-boundary1}, and the last equality is obtained because $L^{(n)}_{1}(0) = L^{(n)}(0-) + 1$ and $R^{(n)}_{d,1}(0) = R^{(n)}_{d}(0-)$ under $\dd{P}^{(n)}_{e}$. Similarly, we have
\begin{align}
 \label{eq:SDQ=<df2}
 & E^{(n)}_{\Delta,d,[0,x]}(\theta)  = \dd{E}^{(n)}_{d}[(f^{(n)}_{n^{-1/2}\theta,[0,x]}(X^{(n)}(0)) - f^{(n)}_{n^{-1/2}\theta,[0,x]}(X^{(n)}_{1}(0))) 1(L^{(n)}(0) = q^{(n)}_{x})]\nonumber\\
 & \quad = \dd{E}^{(n)}_{d} \Big[e^{n^{-1/2}{\theta} q^{(n)}_{x}} g^{(n)}_{n^{-1/2}{\theta}}(R^{(n)}(0))1(L^{(n)}(0) \le q^{(n)}_{x},L^{(n)}(0) = q^{(n)}_{x})  \nonumber\\
 & \hspace{15ex} - e^{n^{-1/2}\theta (q^{(n)}_{x}+1)} g^{(n)}_{n^{-1/2}\theta}(R^{(n)}_{1}(0)) 1(L^{(n)}_{1}(0) \le q^{(n)}_{x},L^{(n)}(0) = q^{(n)}_{x}) \Big] \nonumber\\
 & \quad = e^{n^{-1/2}{\theta} (q^{(n)}_{x}+1)} \dd{E}^{(n)}_{d}[e^{-\eta^{(n)}(n^{-1/2}{\theta}) \widecheck{R}^{(n)}_{e}(0)} 1(L^{(n)}(0) = q^{(n)}_{x})],
\end{align}
because $L^{(n)}_{1}(0) = L^{(n)}(0) + 1$ and $R^{(n)}_{e}(0) = R^{(n)}_{e,1}(0)$ under $\dd{P}^{(n)}_{d}$. Then, substituting \eq{SDQ=<df1} and \eq{SDQ=<df2} {into} \eq{SDQ=<df}, we have
\begin{align}
 \label{eq:SDQ=<En1}
 E^{(n)}_{\Delta,[0,x]}(\theta) & = e^{n^{-1/2}{\theta} (q^{(n)}_{x}+1)} \dd{E}^{(n)}_{d}[e^{-\eta^{(n)}(n^{-1/2}{\theta}) \widecheck{R}^{(n)}_{e}(0)} 1(L^{(n)}(0) = q^{(n)}_{x})] \nonumber\\
 & \quad - e^{n^{-1/2}{\theta} q^{(n)}_{x}} \dd{E}^{(n)}_{e}[e^{-\zeta^{(n)}(n^{-1/2}{\theta}) \widecheck{R}^{(n)}_{d}(0-)} 1(L^{(n)}(0-) = q^{(n)}_{x})] \nonumber\\
 & = e^{n^{-1/2}{\theta} q^{(n)}_{x}} (e^{n^{-1/2}{\theta}} - 1) \dd{E}^{(n)}_{d}[e^{-\eta^{(n)}(n^{-1/2}{\theta}) \widecheck{R}^{(n)}_{e}(0)} 1(L^{(n)}(0) = q^{(n)}_{x})]  \nonumber\\
 & \quad - e^{n^{-1/2}{\theta} q^{(n)}_{x}} \big(\dd{E}^{(n)}_{e}[e^{-\zeta^{(n)}(n^{-1/2}{\theta}) \widecheck{R}^{(n)}_{d}(0-)} 1(L^{(n)}(0-) = q^{(n)}_{x})] \nonumber\\
 & \hspace{15ex} - \dd{E}^{(n)}_{d}[e^{-\eta^{(n)}(n^{-1/2}{\theta}) \widecheck{R}^{(n)}_{e}(0)} 1(L^{(n)}(0) = q^{(n)}_{x})] \big).
\end{align}

Similarly, for $E^{(n)}_{\Delta,(x,\infty)}(\theta)$, define
\begin{align*}
 & E^{(n)}_{\Delta,e,(x,\infty)}(\theta) = \dd{E}^{(n)}_{e}\left[\Delta_{e} f^{(n)}_{n^{-1/2}\theta,(x,\infty)}(X^{(n)})(0) 1(L^{(n)}(0-) = q^{(n)}_{x}) \right], \qquad x \ge 0,\\
 & E^{(n)}_{\Delta,d,(x,\infty)}(\theta) = \dd{E}^{(n)}_{d}\left[\Delta_{d} f^{(n)}_{n^{-1/2}\theta,(x,\infty)}(X^{(n)})(0) 1(L^{(n)}(0) = q^{(n)}_{x}) \right], \qquad x > 0,
\end{align*}
then
\begin{align}
 \label{eq:SDQ>df}
 E^{(n)}_{\Delta,(x,\infty)}(\theta) = E^{(n)}_{\Delta,e,(x,\infty)}(\theta) + E^{(n)}_{\Delta,d,(x,\infty)}(\theta).
\end{align}
Similarly to \eq{SDQ=<df1},  we have
\begin{align}
 \label{eq:SDQ>df1}
 & E^{(n)}_{\Delta,e,(x,\infty)}(\theta) = \dd{E}^{(n)}_{e}[(f^{(n)}_{n^{-1/2}\theta,(x,\infty)}(X^{(n)}_{1}(0)) - f^{(n)}_{n^{-1/2}\theta,(x,\infty)}(X^{(n)}(0-))) 1(L^{(n)}(0-) = q^{(n)}_{x})] \nonumber\\
 & \quad = \dd{E}^{(n)}_{e} \Big[e^{n^{-1/2}\theta (q^{(n)}_{x} + 1)} g^{(n)}_{n^{-1/2}\theta}(R^{(n)}_{1}(0))1(L^{(n)}_{1}(0) > q^{(n)}_{x},L^{(n)}(0-) = q^{(n)}_{x})  \nonumber\\
 & \hspace{15ex} - e^{n^{-1/2}{\theta} q^{(n)}_{x}} g^{(n)}_{n^{-1/2}{\theta}}(R^{(n)}(0-)) 1(L^{(n)}(0-) > q^{(n)}_{x},L^{(n)}(0-) = q^{(n)}_{x}) \Big] \nonumber\\
 & \quad = e^{n^{-1/2}{\theta} q^{(n)}_{x}} \dd{E}^{(n)}_{e}[e^{-\zeta^{(n)}(n^{-1/2}{\theta}) \widecheck{R}^{(n)}_{d}(0-)} 1(L^{(n)}(0-) = q^{(n)}_{x})] = - E^{(n)}_{\Delta,e,[0,x]}(\theta).
\end{align}
Similarly to \eq{SDQ=<df2}, we have $E^{(n)}_{\Delta,d,(x,\infty)}(\theta)  = - E^{(n)}_{\Delta,d,[0,x]}(\theta)$. Hence, from \eq{SDQ>df}, we have
\begin{align}
 \label{eq:SDQ>En1}
 E^{(n)}_{\Delta,(x,\infty)}(\theta) & = - E^{(n)}_{\Delta,[0,x]}(\theta).
\end{align}

We will divide asymptotic BAR's from \eq{SDQ=<En1} and \eq{SDQ>En1} dividing by $n^{-1}$, and take their limits as $n \to \infty$. For these computations, we will use Lemmas \lemt{SDQ-moment} and \lemt{SDQ-t-moment}. We need one more lemma and its corollary for {evaluating} $\alpha^{(n)}_{e} E^{(n)}_{\Delta,[0,x]}(\theta)$.

\begin{lemma}
 \label{lem:SDQ-Red}
 For $j=1,2,3$ and $n \ge 1$,
 \begin{align}
  \label{eq:SDQ-Rd-1}
  & j \dd{E}[\mu^{(n)}(L^{(n)}) (R^{(n)}_{d})^{j-1} (1-r^{(n)}_{d}) 1(0 < L^{(n)} \le q^{(n)}_{x})] \nonumber\\
  & \quad = \alpha^{(n)}_{e} \Big(-\dd{E}^{(n)}_{e}[(\widecheck{R}^{(n)}_{d}(0-))^{j} 1(L^{(n)}(0-) = q^{(n)}_{x})] + \dd{E}[(\widecheck{T}^{(n)}_{S})^{j}] \dd{P}^{(n)}_{d}[L^{(n)}(0) \le q^{(n)}_x]\Big),\\
  \label{eq:SDQ-Rd-K}
  & j \dd{E}[\mu^{(n)}(L^{(n)}(0)) (R^{(n)}_{d})^{j-1} (1-r^{(n)}_{d}) 1( L^{(n)} > q^{(n)}_{x})]\nonumber\\
  & \quad = \alpha^{(n)}_{e} \Big( \dd{E}^{(n)}_{e}[(\widecheck{R}^{(n)}_{d}(0-))^{j} 1(L^{(n)}(0-) = q^{(n)}_{x})] + \dd{E}[(\widecheck{T}^{(n)}_{S})^{j}] \dd{P}^{(n)}_{d}[L^{(n)}(0) > q^{(n)}_{x}]\Big),
 \end{align}
 where recall that $r^{(n)}_{d} = 1(R^{(n)}_{d} > n^{1/2})$ and $r^{(n)}_{e} = 1(R^{(n)}_{e} > n^{1/2})$, then
 \begin{align}
  \label{eq:SDQ-Re-1}
  & j \dd{E}[\lambda^{(n)}(L^{(n)}) (R^{(n)}_{e})^{j-1} (1-r^{(n)}_{e}) 1(L^{(n)} \le q^{(n)}_{x})]  \nonumber\\
  & \quad = \alpha^{(n)}_{e}\Big( \dd{E}_{d}[(\widecheck{R}^{(n)}_{e}(0))^{j} 1(L^{(n)}(0) = q^{(n)}_{x})] + \dd{E}[(\widecheck{T}^{(n)}_{A})^{j}] \dd{P}^{(n)}_{e}[L^{(n)}(0) \le q^{(n)}_x]\Big),\\
  \label{eq:SDQ-Re-K}
  & j \dd{E}[\lambda^{(n)}(L^{(n)}) (1-r^{(n)}_{e}) 1(L^{(n)} > q^{(n)}_{x})] \nonumber\\
  & \quad = \alpha^{(n)}_{e}\Big(- \dd{E}_{d}[(\widecheck{R}^{(n)}_{e}(0))^{j} 1(L^{(n)}(0) = q^{(n)}_{x})] + \dd{E}[(\widecheck{T}^{(n)}_{A})^{j}] \dd{P}^{(n)}_{e}[L^{(n)}(0) > q^{(n)}_{x}] \Big).
 \end{align}
\end{lemma}
\begin{proof}
 For $j \ge 1$, we first prove \eq{SDQ-Rd-1}. We apply $f(u,\vc{y}) = (y_{2} \wedge n^{1/2})^{j} 1(u \le q^{(n)}_{x})$ to \eq{BAR1}, then we have \eq{SDQ-Rd-1} because $\dd{E}^{(n)}_{d}[(\widecheck{R}^{(n)}_{d}(0-)(0))^{j}] = \dd{E}[(\widecheck{T}^{(n)}_{S})^{j}]$ and
 \begin{align*}
  & \sr{H}^{(n)} f(X^{(n)}(t)) = - j \mu^{(n)}(L^{(n)}(t)) (R^{(n)}_{d}(t))^{j-1} 1(R^{(n)}_{d} < n^{1/2}, 1 \le L^{(n)}(t) \le q^{(n)}_{x}),\\
  & \Delta_{e,x} f(X^{(n)})(t) = - (\widecheck{R}^{(n)}_{d}(0-)(t-))^{j} 1(L^{(n)}(t-) = q^{(n)}_{x}) 1(\Delta N^{(n)}_{e}(t) = 1),\\
  & \Delta_{d,x} f(X^{(n)})(t) = (\widecheck{R}^{(n)}_{d}(0-)(t))^{j} 1(L^{(n)}(t) \le q^{(n)}_{x}) 1(\Delta N^{(n)}_{d}(t) = 1).
 \end{align*}
 Similarly, \eq{SDQ-Rd-K} is obtained from \eq{BAR1} by applying $f(u,\vc{y}) = (y_{2} \wedge n^{1/2})^{j} 1(u > q^{(n)}_{x})$. \eq{SDQ-Re-1} and \eq{SDQ-Re-K} are similarly proved by applying $f(u,\vc{y}) = (y_{1} \wedge n^{1/2})^{j} 1(u \le q^{(n)}_{x})$ and $f(u,\vc{y}) = (y_{1} \wedge n^{1/2})^{j} 1(u > q^{(n)}_{x})$, respectively.
\end{proof}

From Lemmas \lemt{SDQ-moment} and \lemt{SDQ-Red}, we have the following corollary.

\begin{corollary}
 \label{cor:SDQ-Re-Rs 1}
 For each $j=1,2,3$ and $x \in \dd{R}_{+}$, 
\begin{align}
\label{eq:wRnd}
  & \sup_{n \ge 1} \dd{E}_{e}[(\widecheck{R}^{(n)}_{d}(0-))^{j} 1(L^{(n)}(0-) = q^{(n)}_{x})] < \infty,\\
\label{eq:wRne}
  & \sup_{n \ge 1}  \dd{E}_{d}[(\widecheck{R}^{(n)}_{e}(0))^{j} 1(L^{(n)}(0) = q^{(n)}_{x})] < \infty,
\end{align}
where recall that $\widecheck{R}^{(n)}_{e}(0) = R^{(n)}_{e}(0) \wedge n^{1/2}$ and $\widecheck{R}^{(n)}_{d}(0-) = R^{(n)}_{d}(0-) \wedge n^{1/2}$.
\end{corollary}
\begin{proof}
From \eq{SDQ-Rd-1}, we have
\begin{align*}
  & \alpha^{(n)}_{e} \dd{E}^{(n)}_{e}[(\widecheck{R}^{(n)}_{d}(0-))^{j} 1(L^{(n)}(0-) = q^{(n)}_{x})] \nonumber\\
  & \quad = \alpha^{(n)}_{e} \dd{E}[(\widecheck{T}^{(n)}_{S})^{j}] \dd{P}^{(n)}_{d}[L^{(n)}(0) \le q^{(n)}_x] \nonumber\\
  & \qquad - j \dd{E}[\mu^{(n)}(L^{(n)}) (R^{(n)}_{d})^{j-1} (1-r^{(n)}_{d}) 1(0 < L^{(n)} \le q^{(n)}_{x})].
\end{align*}
Hence, \eq{wRnd} is proved by \lem{SDQ-moment}. Similarly, \eq{wRne} follows from \eq{SDQ-Re-1}.
\end{proof}

Note that two different types of probability measures, stationary and Palm, are involved in the BAR \eq{BAR1}. This causes difficulty in analyzing it. To relax this problem, we consider relations among them. To this end, we introduce the following auxiliary quantity. Define, for $x \in \dd{R}_{+}$,
\begin{align}
 \label{eq:SDQ-Dn-i}
 \Delta^{(n)}_{x} = \alpha^{(n)}_{e} \Big( & \dd{E}^{(n)}_{e}[\widecheck{R}^{(n)}_{d}(0-) 1(L^{(n)}(0-) = q^{(n)}_{x})] + \dd{E}^{(n)}_{d}[\widecheck{R}^{(n)}_{e}(0) 1(L^{(n)}(0) = q^{(n)}_{x})] \nonumber\\
 & - \dd{P}^{(n)}_{e}[L^{(n)}(0-) = q^{(n)}_{x}] \Big).
\end{align}
Then, $\Delta^{(n)}_{x}$ is well defined and finite by \cor{SDQ-Re-Rs 1}, and we have the following lemma, which gives relations the expectations under $\dd{P}$, $\dd{P}^{(n)}_{e}$ and $\dd{P}^{(n)}_{d}$.
\begin{lemma}
 \label{lem:SDQ-R2-1}
 \begin{align}
  \label{eq:SDQ-Dn-j+}
  & n^{-1/2} \dd{E}[\widehat{b}^{(n)}(\widehat{L}^{(n)}) 1(\widehat{L}^{(n)} \le n^{-1/2} q^{(n)}_{x})] + \mu^{(n)}(0) \dd{P}[\widehat{L}^{(n)} = 0] + o(n^{-1}) = \Delta^{(n)}_{x},\\
  \label{eq:SDQ-Dn-j-}
  & n^{-1/2} \dd{E}[\widehat{b}^{(n)}(\widehat{L}^{(n)}) 1(\widehat{L}^{(n)} > n^{-1/2} q^{(n)}_{x})] + o(n^{-1}) = - \Delta^{(n)}_{x},\\
  \label{eq:BAR-Pn-0}
  & \mu^{(n)}(0) \dd{P}[L^{(n)} = 0] = - n^{-1/2} \dd{E}[\widehat{b}^{(n)}(\widehat{L}^{(n)})] + o(n^{-1}).
 \end{align}
\end{lemma}

\begin{proof}
For making computations easy, we introduce the following notations.
\begin{align*}
 \sr{E}^{(n)}_{1,x} & = \dd{E}[\lambda^{(n)}(L^{(n)}) r^{(n)}_{e} 1(L^{(n)} \le q^{(n)}_{x})] - \dd{E}[\mu^{(n)}(L^{(n)}) r^{(n)}_{d}1(0<L^{(n)} \le q^{(n)}_{x})] \nonumber\\
 \sr{E}^{(n)}_{2,x} & = \dd{E}[\lambda^{(n)}(L^{(n)}) r^{(n)}_{e} 1(L^{(n)} > q^{(n)}_{x})] - \dd{E}[\mu^{(n)}(L^{(n)}) r^{(n)}_{d} 1(L^{(n)} > q^{(n)}_{x})], \nonumber\\
 \sr{E}^{(n)} & = \dd{E}[\lambda^{(n)}(L^{(n)}) r^{(n)}_{e}] - \dd{E}[\mu^{(n)}(L^{(n)}) r^{(n)}_{d} 1(L^{(n)} > 0)].
\end{align*}
Note that
\begin{align*}
  & \lambda^{(n)}(L^{(n)}) - \mu^{(n)}(L^{(n)}) = \widehat{\lambda}^{(n)}(\widehat{L}^{(n)}) - \widehat{\mu}^{(n)}(\widehat{L}^{(n)}) = n^{-1/2} \widehat{b}^{(n)}(\widehat{L}^{(n)}),
\end{align*}
and, by \lem{SDQ-t-moment}, 
  \begin{align*}
  & \dd{E}[\widecheck{T}^{(n)}_{A}] \dd{P}^{(n)}_{e}[L^{(n)}(0) \le q^{(n)}_{x}] - \dd{E}[\widecheck{T}^{(n)}_{S}] \dd{P}^{(n)}_{d}[L^{(n)}(0) \le q^{(n)}_{x}]\\
  & \quad = \dd{P}^{(n)}_{e}[L^{(n)}(0) \le q^{(n)}_{x}] - \dd{P}^{(n)}_{d}[L^{(n)}(0) \le q^{(n)}_{x}] + o(n^{-1}).
 \end{align*}
Subtracting both sides of \eq{SDQ-Rd-1} from those of \eq{SDQ-Re-1} for $j=1$, and applying the notation $\sr{E}^{(n)}_{1,x}$ and the above formulas, we have
 \begin{align}
  \label{eq:SDQ-Rn-1}
  & \dd{E}[n^{-1/2} \dd{E}[\widehat{b}^{(n)}(\widehat{L}^{(n)}) 1(L^{(n)} \le q^{(n)}_{x})]  + \mu^{(n)}(0) \dd{P}[L^{(n)} = 0] - \sr{E}^{(n)}_{1,x}\nonumber\\
  & \quad = \alpha^{(n)}_{e} \Big(\dd{E}^{(n)}_{e}[\widecheck{R}^{(n)}_{d}(0-) 1(L^{(n)}(0-) = q^{(n)}_{x})] + \dd{E}^{(n)}_{d}[\widecheck{R}^{(n)}_{e}(0) 1(L^{(n)}(0) = q^{(n)}_{x})] \nonumber\\
  & \hspace{11ex} +\dd{P}^{(n)}_{e}[L^{(n)}(0) \le q^{(n)}_{x}] - \dd{P}^{(n)}_{d}[L^{(n)}(0) \le q^{(n)}_{x}] + o(n^{-1}) \Big).  
\end{align}
  This implies \eq{SDQ-Dn-j+}
 because $\sr{E}^{(n)}_{1,x} = o(n^{-1})$ by \eq{SDQ-Re 1} and
 \begin{align*}
  \dd{P}^{(n)}_{e}[L^{(n)}(0) \le q^{(n)}_{x}] - \dd{P}^{(n)}_{d}[L^{(n)}(0) \le q^{(n)}_{x}] = - \dd{P}^{(n)}_{e}[L^{(n)}(0-) = q^{(n)}_{x}]
\end{align*}
by \lem{SDQ-basic 1}. Similarly, subtracting both sides of \eq{SDQ-Rd-K} from those of \eq{SDQ-Re-K} for $k=1$, we have \eq{SDQ-Dn-j-}. Substituting $\Delta^{(n)}_{x}$ of \eq{SDQ-Dn-j-} into \eq{SDQ-Dn-j+}, we have \eq{BAR-Pn-0}.
\end{proof}

Our next task is to expand the jump component $E^{(n)}_{\Delta,[0,x]}(\theta)$ of \eq{SDQ=<En1} using $\Delta^{(n)}_{x}$ for large $n$. For this, we will use the following Taylor expansions.

\begin{lemma}
 \label{lem:SDQ-asymp 1}
 For each fixed ${\theta}$ and $x$, as $n \to \infty$,
 \begin{align}
  \label{eq:SDQ-exp 1}
  & e^{n^{-1/2} \theta  q^{(n)}_{x}} = e^{{\theta} x} (1 + o({\theta} n^{-1/2})), \\
  \label{eq:SDQ-exp 2}
  & e^{n^{-1/2}{\theta} (q^{(n)}_{x}+1)} = e^{{\theta} x} (1 + {\theta}n^{-1/2} + o({\theta} n^{-1/2})).
 \end{align}
\end{lemma}

\begin{lemma}
 \label{lem:SDQ-asymp 2}
 Let $P_{n}(u) = \sum_{k=1}^{n} u^{k}$ for $u \in \dd{R}$ and integer $n \ge 1${. Then,} for each fixed ${\theta}$, as $n \to \infty$, for $i=1,2$,
 \begin{align}
  \label{eq:SDQ-exp-eta 1}
  e^{-\eta^{(n)}(n^{-1/2}{\theta}) \widecheck{R}^{(n)}_{e}(0)} & = 1 - n^{-1/2}{\theta} \widecheck{R}^{(n)}_{e}(0) \nonumber\\
  & \quad - 2^{-1} (n^{-1/2}{\theta})^{2} [ \sigma_{A}^{2} \widecheck{R}^{(n)}_{e}(0) - (\widecheck{R}^{(n)}_{e}(0))^{2} ] + P_{3}(\widecheck{R}^{(n)}_{e}(0)) O((n^{-1/2}{\theta})^{3}),\\
  \label{eq:SDQ-exp-zeta 1}
  e^{-\zeta^{(n)}(n^{-1/2}{\theta}) \widecheck{R}^{(n)}_{d}(0-)} & = 1 + n^{-1/2}{\theta} \widecheck{R}^{(n)}_{d}(0-) \nonumber\\
  & \quad - 2^{-1} (n^{-1/2}{\theta})^{2} [ \sigma_{S}^{2} \widecheck{R}^{(n)}_{d}(0-) - (\widecheck{R}^{(n)}_{d}(0-))^{2} ] + P_{3}(\widecheck{R}^{(n)}_{d}(0-)) O((n^{-1/2}{\theta})^{3}).
 \end{align}
\end{lemma}

\lem{SDQ-asymp 1} is essentially the same as Lemma 4.7 for $r = n^{-1/2}$ of \cite{Miya2025}, and \lem{SDQ-asymp 2} is obtained similarly to Lemma 4.8 by expanding those exponential functions up to the order $n^{-1/2}{\theta}$. From \eq{SDQ=<En1} and \eq{SDQ>En1} for $E^{(n)}_{\Delta,[0,x]}(\theta)$, we have the following lemma.

\begin{lemma}[Discontinuous part of the asymptotic BAR]
 \label{lem:BAR=<P2}
 For $x \in \dd{R}_{+}$ and each $\theta \le 0$, as $n \to \infty$,
\begin{align}
\label{eq:BAR=<P2}
  & \alpha^{(n)}_{e} E^{(n)}_{\Delta,[0,x]}(\theta) \nonumber\\
  & \quad = n^{-1} \theta  \left(e^{n^{-1/2} q^{(n)}_{x} {\theta}} \big\{\dd{E}[\widehat{b}^{(n)}(\widehat{L}^{(n)}) 1(\widehat{L}^{(n)} > n^{-1/2} q^{(n)}_{x})] + \theta H^{(n)}_{x}\big\} + \theta o(1)\right),\\
\label{eq:BAR>P2}
  & \alpha^{(n)}_{e} E^{(n)}_{\Delta,(x,\infty)}(\theta) = - \alpha^{(n)}_{e} E^{(n)}_{\Delta,[0,x]}(\theta),
\end{align}
 where recall that $H^{(n)}_{x}$ is defined by \eq{Hn-x}.
\end{lemma}

\begin{proof}
 We evaluate $E^{(n)}_{\Delta,[0,x]}(\theta)$ through \eq{SDQ=<En1}. For this, we note that, by \eq{SDQ-exp-eta 1} and \eq{SDQ-exp-zeta 1},
 \begin{align*}
  & \dd{E}^{(n)}_{d}[e^{-\eta^{(n)}(n^{-1/2}{\theta}) \widecheck{R}^{(n)}_{e}(0)} 1(L^{(n)}(0) = q^{(n)}_{x})] \nonumber\\
  & \quad = \dd{E}^{(n)}_{d}[(1 - n^{-1/2}{\theta} \widecheck{R}^{(n)}_{e}(0) - (2n)^{-1} \theta^{2} (\sigma_{A}^{2} \widecheck{R}^{(n)}_{e}(0) - (\widecheck{R}^{(n)}_{e}(0))^{2} + o(1)))1(L^{(n)}(0) = q^{(n)}_{x})], \nonumber\\
  & \dd{E}^{(n)}_{e}[e^{-\zeta^{(n)}(n^{-1/2}{\theta}) \widecheck{R}^{(n)}_{d}(0-)} 1(L^{(n)}(0-) = q^{(n)}_{x})] \nonumber\\
  & \quad = \dd{E}^{(n)}_{e}[(1 + n^{-1/2}{\theta} \widecheck{R}^{(n)}_{d}(0-) - (2n)^{-1} \theta^{2} (\sigma_{S}^{2} \widecheck{R}^{(n)}_{d}(0-) - (\widecheck{R}^{(n)}_{d}(0-))^{2} + o(1)))1(L^{(n)}(0-) = q^{(n)}_{x})].
 \end{align*}
 Substituting these formulas into \eq{SDQ=<En1} and recalling the definition of $\Delta^{(n)}_{x}$, we have
 \begin{align*}
  & E^{(n)}_{\Delta,[0,x]}(\theta) = e^{n^{-1/2}{\theta} q^{(n)}_{x}} n^{-1/2} \theta (1 + o(1)) \big\{\dd{E}^{(n)}_{d}[(1 - n^{-1/2}{\theta} \widecheck{R}^{(n)}_{e}(0)) 1(L^{(n)}(0) = q^{(n)}_{x})]  \nonumber\\
  & \quad \hspace{22ex} - 2^{-1} n^{-1} \theta^{2} \dd{E}^{(n)}_{d}[(\sigma_{A}^{2} \widecheck{R}^{(n)}_{e}(0) - (\widecheck{R}^{(n)}_{e}(0))^{2} + o(1))1(L^{(n)}(0) = q^{(n)}_{x})] \big\} \nonumber\\
  & \quad - e^{n^{-1/2}{\theta} q^{(n)}_{x}} n^{-1/2}{\theta} \big\{\dd{E}^{(n)}_{e}[\widecheck{R}^{(n)}_{d}(0-) 1(L^{(n)}(0-) = q^{(n)}_{x})] + \dd{E}^{(n)}_{d}[\widecheck{R}^{(n)}_{e}(0) 1(L^{(n)}(0) = q^{(n)}_{x})] \big\} \nonumber\\
  & \quad + e^{n^{-1/2}{\theta} q^{(n)}_{x}} 2^{-1} n^{-1} \theta^{2} \big\{\dd{E}^{(n)}_{e}[(\sigma_{S}^{2} \widecheck{R}^{(n)}_{d}(0-) - (\widecheck{R}^{(n)}_{d}(0-))^{2} + o(1)) 1(L^{(n)}(0-) = q^{(n)}_{x})] \nonumber\\
  & \hspace{24ex} - \dd{E}^{(n)}_{d}[(\sigma_{A}^{2} \widecheck{R}^{(n)}_{e}(0) - (\widecheck{R}^{(n)}_{e}(0))^{2} + o(1)) 1(L^{(n)}(0) = q^{(n)}_{x})] \big\}.
\end{align*}
Then, using the definition \eq{SDQ-Dn-i} of $\Delta^{(n)}_{x}$, this equation yields
\begin{align*}
  \alpha^{(n)}_{e} E^{(n)}_{\Delta,[0,x]}(\theta) & = e^{n^{-1/2}{\theta} q^{(n)}_{x}} \big\{ - n^{-1/2} \theta \Delta^{(n)}_{x} - n^{-1} \theta^{2} (1 + o(1)) \alpha^{(n)}_{e} \dd{E}^{(n)}_{d}[ \widecheck{R}^{(n)}_{e}(0) 1(L^{(n)}(0) = q^{(n)}_{x})]  \nonumber\\
  & \quad- 2^{-1} n^{-3/2} \theta^{3} \alpha^{(n)}_{e} \dd{E}^{(n)}_{d}[(\sigma_{A}^{2} \widecheck{R}^{(n)}_{e}(0) - (\widecheck{R}^{(n)}_{e}(0))^{2} + o(1))1(L^{(n)}(0) = q^{(n)}_{x})] \nonumber\\
  & \quad+ 2^{-1} n^{-1} \theta^{2} \alpha^{(n)}_{e} \dd{E}^{(n)}_{e}[(\sigma_{S}^{2} \widecheck{R}^{(n)}_{d}(0-) - (\widecheck{R}^{(n)}_{d}(0-))^{2} + o(1)) 1(L^{(n)}(0-) = q^{(n)}_{x})] \nonumber\\
  & \quad - 2^{-1} n^{-1} \theta^{2} \alpha^{(n)}_{e} \dd{E}^{(n)}_{d}[(\sigma_{A}^{2} \widecheck{R}^{(n)}_{e}(0) - (\widecheck{R}^{(n)}_{e}(0))^{2} + o(1)) 1(L^{(n)}(0) = q^{(n)}_{x})] \big\}.
\end{align*}
Hence, by the definition \eq{Hn-x} of $H^{(n)}_{x}$, we arrive at
\begin{align*}
  \alpha^{(n)}_{e} E^{(n)}_{\Delta,[0,x]}(\theta) & = e^{n^{-1/2}{\theta} q^{(n)}_{x}} \big\{ - n^{-1/2} \theta \Delta^{(n)}_{x} - \theta^{2} o(n^{-1}) \alpha^{(n)}_{e} \dd{E}^{(n)}_{d}[ \widecheck{R}^{(n)}_{e}(0) 1(L^{(n)}(0) = q^{(n)}_{x})]  \nonumber\\
  & \quad + n^{-1} \theta^{2} H^{(n)}_{x} \nonumber\\
  & \quad - 2^{-1} n^{-3/2} \theta^{3} \alpha^{(n)}_{e} \dd{E}^{(n)}_{d}[(\sigma_{A}^{2} \widecheck{R}^{(n)}_{e}(0) - (\widecheck{R}^{(n)}_{e}(0))^{2} + o(1))1(L^{(n)}(0) = q^{(n)}_{x})] \big\}.
 \end{align*}
Thus we have \eq{BAR=<P2} by \eq{SDQ-Dn-j-} for $\Delta^{(n)}_{x}$ and \eq{wRne} in \cor{SDQ-Re-Rs 1}. Finally, \eq{BAR>P2} is immediate from \eq{SDQ>En1}.
\end{proof}

\section{Proofs of the main results}
\label{sec:p-main}
\setnewcounter

In this section, we prove Theorems \thrt{tight-2}, \thrt{SDQ-main-1} and \thrt{MLQ-main-1}.

\subsection{Proof of \thr{tight-2}}
\label{sec:p-tight-2}

Recall that $\sr{S}_{sd} = \{\nu^{(n)}; n \ge 1\}$. By \lem{Helly}, for any subsequence of $\sr{S}_{sd}$, we can find a further subsequence $\{\nu^{(n_{k})}; k \ge 1\}$ and a sub-probability measure {$\nu$} on $(\dd{R}_{+},\sr{B})$ such that $\nu^{{(n_{k})}} \os{v}{\rightarrow} \nu$ as $k \to \infty$. We show that $\nu$ is a probability measure if $b_{\infty} < 0$. This proves \thr{tight-2} by \lem{tight-1}.

From the condition (\sect{preliminaries}.g), $b(u)$ converges to $b_{\infty} \in \dd{R}$ as $u \to \infty$. Hence, for any $\varepsilon > 0$, there {exists} $x_{0} \ge 1$ such that
\begin{align*}
 |b(u) - b_{\infty}| < \varepsilon, \qquad \forall u \ge x_{0}.
\end{align*}
For this $\varepsilon$, by {the condition} (\sect{preliminaries}.f), we can choose $k_{0} \ge 1$ such that
\begin{align*}
 \sup_{u \ge x_{0}}|\widehat{b}^{(n_{k})}(u) - b(u)| < \varepsilon, \qquad \forall k \ge k_{0}.
\end{align*}
Hence, we have
\begin{align}
 \label{eq:b0}
 \sup_{u \ge x_{0}}|\widehat{b}^{(n_{k})}(u) - b_{\infty}| < 2\varepsilon, \qquad \forall k \ge k_{0}.
\end{align}
On the other hand, by the condition (\sect{preliminaries}.a) and \lem{h-limit}, there is a $\sigma_{\infty}$ such that
\begin{align}
 \label{eq:s0}
 \sigma_{\infty} = \sup_{k \ge 1} \sup_{u \in \dd{R}_{+}} \sigma^{(n_{k})}(u) < \infty.
\end{align}

We now work on the asymptotic BAR in \lem{asym-BAR1} for $n=n_{k}$. By \eq{asym-BAR>1}, we have
\begin{align*}
 & \dd{E}\left[ \left(\widehat{b}^{(n_{k})}(\widehat{L}^{(n_{k})}) + \frac 12 {\theta} (\widehat{\sigma}^{(n_{k})})^{2}(\widehat{L}^{(n_{k})}) \right) e^{\theta \widehat{L}^{(n_{k})}} 1(\widehat{L}^{(n_{k})} > x) \right] \nonumber\\
 & \qquad = e^{\theta x} \dd{E}[\widehat{b}^{(n_{k})}(\widehat{L}^{(n_{k})}) 1(\widehat{L}^{(n_{k})} > x)] + \theta e^{\theta n_{k}^{-1/2} q^{(n_{k})}_{x}} H^{(n_{k})}_{x} + \theta o(1), \qquad \theta < 0, \ x \in \dd{R}_{+},
\end{align*}
or equivalently,
\begin{align}
\label{eq:asym-BAR>2}
 & \int_{x}^{\infty} \left(\widehat{b}^{(n_{k})}(u) + \frac 12 {\theta} (\widehat{\sigma}^{(n_{k})})^{2}(u) \right) e^{\theta u} d \nu^{(n_{k})}(u)  \nonumber\\
 & \quad = e^{\theta x} \int_{x}^{\infty} \widehat{b}^{(n_{k})}(u) d\nu^{(n_{k})}(u) + \theta e^{\theta n_{k}^{-1/2} q^{(n_{k})}_{x}} H^{(n_{k})}_{x} + \theta o(1), \qquad \theta < 0, \ x \in \dd{R}_{+}.
\end{align}
Letting $j \to \infty$ in this equation, by $\nu^{(n_{k})} \os{v}{\to} \nu$ and \lem{limit-hbn}, we have, for $\theta < 0$,
\begin{align}
\label{eq:BAR2}
 & \int_{x}^{\infty} \left(b(u) + \frac 12 {\theta} \sigma^{2}(u) \right) e^{\theta u} d \nu(u) = e^{\theta x} \int_{x}^{\infty} b(u) d\nu^{(n_{k})}(u) + \theta e^{\theta x} \lim_{k \to \infty} H^{(n_{k})}_{x}, \; x \in \dd{R}_{+}.
\end{align}
Hence, $\lim_{k \to \infty} H^{(n_{k})}_{x}$ exists and is finite. We denote this limit by $H_{x}$.

Applying \eq{b0} and \eq{s0} to \eq{BAR2} for $x=x_{0}$, we have, for $\theta < 0$ and $k \ge k_{0}$,
\begin{align}
\label{eq:asym-BAR4}
 & \left(b_{\infty} - 2\varepsilon + \frac 12 {\theta} \sigma_{\infty}^{2}\right) \int_{x_{0}}^{\infty} e^{\theta u} d \nu^{(n_{k})}(u) \le e^{\theta x_{0}} (b_{\infty} + 2\varepsilon) \int_{x_{0}}^{\infty} d\nu^{(n_{k})}(u) + \theta e^{\theta n_{k}^{-1/2} q^{(n_{k})}_{x_{0}}} H^{(n_{k})}_{x_{0}} + \theta o(1) \nonumber\\
 & \quad = e^{\theta x_{0}} (b_{\infty} + 2\varepsilon) ( 1 - \nu^{(n_{k})}([0,x_{0}])) + \theta e^{\theta n_{k}^{-1/2} q^{(n_{k})}_{x_{0}}} H^{(n_{k})}_{x_{0}} + \theta o(1),
\end{align}
because $\nu^{(n_{k})}(\dd{R}_{+}) = 1$. Since $b_{\infty} - 2\varepsilon + \frac 12 {\theta} \sigma_{\infty}^{2} \le 0$ for $\theta < 0$ and
\begin{align*}
 \int_{x_{0}}^{\infty} e^{\theta u} d \nu^{(n_{k})}(u) \le \int_{x_{0}}^{x} e^{\theta u} d \nu^{(n_{k})}(u) + e^{\theta x} \le \nu^{(n_{k})}((x_{0},x]) + e^{\theta x}, \qquad x \ge x_{0}, \ \theta < 0,
\end{align*}
letting $j \to \infty$ in \eq{asym-BAR4} yields, for any $x \ge x_{0}$,
\begin{align}
\label{eq:asym-BAR5}
 & \left(b_{\infty} - 2\varepsilon + \frac 12 {\theta} \sigma_{\infty}^{2}\right) \left(\nu((x_{0},x]) + e^{\theta x}\right) \le e^{\theta x_{0}} (b_{\infty} + 2 \varepsilon) ( 1 - \nu([0,x_{0}])) + \theta e^{\theta x_{0}} H_{x_{0}}.
\end{align}
For each fixed $\theta < 0$, \eq{asym-BAR5} for $x \to \infty$ yields
\begin{align*}
 \left(b_{\infty} - 2\varepsilon + \frac 12 {\theta} \sigma_{\infty}^{2}\right) \nu((x_{0},\infty)) \le e^{\theta x_{0}} (b_{\infty} + 2\varepsilon) ( 1 - \nu([0,x_{0}])) + \theta e^{\theta x_{0}} H_{x_{0}}.
\end{align*}
Then, letting $\theta \uparrow 0$ and letting $\varepsilon \downarrow 0$, we {arrive} at
\begin{align*}
 b_{\infty} \nu((x_{0},\infty)) \le b_{\infty} ( 1 - \nu([0,x_{0}])).
\end{align*}
Hence, if $b_{\infty} < 0$, then this implies that $\nu(\dd{R}_{+}) \ge 1$. Since $\nu$ is a sub-probability measure, $\nu(\dd{R}_{+}) = 1$, and therefore $\sr{S}_{sd}$ is tight by \lem{tight-1}.

\subsection{Proof of \thr{SDQ-main-1}}
\label{sec:p-SDQ-main-1}

To prove Theorem \thrt{SDQ-main-1}, we refines the asymptotic BAR \eq{asym-BAR=<1} for any vaguely convergent subsequence. This will be done in two ways. For this, we use the following condition.
\begin{mylist}{3}
\item [(\sect{p-main}.a)] If $\{\nu^{(n_{k})}; k \ge 1\}$ is a vaguely convergent subsequence of $\sr{S}_{sd}$, then
 \begin{align}
  \label{eq:Hn-x01}
  & \lim_{k \to \infty} H^{(n_{k})}_{x} = 0, \qquad x \ge 0.
 \end{align}
\end{mylist}
Note that, if \eq{Hn-x01} holds, then \eq{asym-BAR=<1} for the subsequence $\{\nu_{n_{k}}; k \ge 1\}$ is simplified to
\begin{align}
 \label{eq:asym-BAR=<2}
 & \frac 12 \dd{E}\left[ (\widehat{\sigma}^{(n_{k})})^{2}(\widehat{L}^{(n_{k})}) \left(\widehat{\beta}^{(n_{k})}(\widehat{L}^{(n_{k})}) + {\theta} \right) e^{\theta \widehat{L}^{(n_{k})}} 1(\widehat{L}^{(n_{k})} \le x) \right] + \mu^{(n_{k})}(0) n^{1/2} \dd{P}[\widehat{L}^{(n_{k})} = 0] \nonumber\\
 & \quad + e^{n^{-1/2} q^{(n_{k})}_{x} {\theta}} \dd{E}[\widehat{b}^{(n_{k})}(\widehat{L}^{(n_{k})}) 1(\widehat{L}^{(n_{k})} > x)] + \theta o(1) = 0, \quad x \in \dd{R}_{+},\ \theta \le 0,\ n \to \infty,
\end{align}

\begin{lemma}\rm
\label{lem:Ln-x01}
Assume the conditions (\sect{preliminaries}.a)--(\sect{preliminaries}.f) for the state-dependent queue. For a vaguely convergent subsequence $\{\nu^{(n_{k})}; k \ge 1\}$ of $\sr{S}_{sd}$, denote its vague limit by $\nu$. Then, For $x \in \dd{R}_{+}$, $\nu$ has a density in a neighborhood of $x$ if and only if \eq{Hn-x01} holds in a neighborhood of $x$.
\end{lemma}

\begin{proof}
For a vaguely convergent subsequence $\{\nu^{(n_{k})}; k \ge 1\}$ of $\sr{S}_{sd}$, denote its vague limit by $\nu$. We first assume that $\nu$ has a density  in a neighborhood of $x$, and prove that it implies \eq{Hn-x01} at $x$.

From the definition \eq{Hn-x} for $H^{(n_{k})}_{x}$, \eq{Hn-x01} is obtained if we show that, for $j=1,2$,
\begin{align}
\label{eq:RLn-e}
 & \lim_{n \to \infty} \dd{E}^{(n_{k})}_{d}[(\widecheck{R}^{(n_{k})}_{e}(0))^{j} 1(L^{(n_{k})}(0) = q^{(n_{k})}_{x})] = 0, \\
\label{eq:RLn-d}
 & \lim_{n \to \infty} \dd{E}^{(n_{k})}_{e}[(\widecheck{R}^{(n_{k})}_{d}(0-))^{j} 1(L^{(n_{k})}(0-) = q^{(n_{k})}_{x})] = 0. 
\end{align}
To prove these facts, we first show that if $\nu$ has a density in a neighborhood of $x$, then
\begin{align}
\label{eq:Lqx-delta}
  \lim_{n \to \infty} \dd{P}\left[L^{(n_{k})} - q^{(n_{k})}_{x} \in [-\delta,\delta]\right] = 0, \qquad \delta \ge 1, \ x > 0.
\end{align}
Since $q^{(n_{k})}_{x} = [n_{k}^{1/2} x]$, there is $k_{0} \ge 1$ for each $\varepsilon > 0$ such that
\begin{align*}
  [n_{k}^{-1/2} (q^{(n_{k})}_{x} - \delta), n_{k}^{-1/2} (q^{(n_{k})}_{x} + \delta)] \subset [x - \varepsilon, x + \varepsilon], \qquad k \ge k_{0}.
\end{align*}
This implies that
\begin{align*}
  \dd{P}\left[L^{(n_{k})} - q^{(n_{k})}_{x} \in [-\delta,\delta]\right] & \le \dd{P}\left[\widehat{L}^{(n_{k})} \in [x - \varepsilon, x + \varepsilon]\right]  \nonumber\\
  & = \nu^{(n_{k})}([x - \varepsilon, x + \varepsilon]) \to \nu([x - \varepsilon, x + \varepsilon]), \qquad k \to \infty,
\end{align*}
where $\nu$ is the vague limit of $\{\nu^{(n_{k})}; k \ge 1\}$. Since $\nu$ has a density in a neighborhood of $x$, we have \eq{Lqx-delta} by letting $\varepsilon \downarrow 0$ in this formla.

We now prove \eq{RLn-e}. This can be done in the same way as the proof of \cite[Lemma 13]{Miya2025}. We outline this proof. Applying the test function $f(u,v_{1},v_{2}) = (v_{1} \wedge n^{1/2}) (v_{2} \wedge n^{1/2})^{j} 1(u=q^{(n_{k})}_{x})$ to the BAR \eq{BAR1}, we have
\begin{align}
\label{eq:BAR-RdRe}
 &\dd{E}[\lambda^{(n_{k})}(q^{(n_{k})}_{x}) \widecheck{R}^{(n_{k})}_{d} 1(R^{(n_{k})}_{e} \le n^{1/2}, L^{(n_{k})} = q^{(n_{k})}_{x})] \nonumber\\
 & \qquad + \dd{E}[\mu^{(n_{k})}(q^{(n_{k})}_{x}) \widecheck{R}^{(n_{k})}_{e} 1(R^{(n_{k})}_{d} \le n^{1/2}, L^{(n_{k})} = q^{(n_{k})}_{x})] \nonumber\\
 & \quad = \alpha^{(n_{k})}_{e} \dd{E}[\widecheck{T}^{(n_{k})}_{A}] \dd{E}^{(n_{k})}_{e}[\widecheck{R}^{(n_{k})}_{d}(0-) 1(L^{(n_{k})}_{1}(0) = q^{(n_{k})}_{x})]  \nonumber\\
 & \qquad + \alpha^{(n_{k})}_{e} \dd{E}[\widecheck{T}^{(n_{k})}_{S}] \dd{E}^{(n_{k})}_{d}[\widecheck{R}^{(n_{k})}_{e}(0) 1(L^{(n_{k})}(0) = q^{(n_{k})}_{x})].
\end{align}
This BAR corresponds to the BAR \cite[(A.139)]{Miya2025}. Then, using \lem{SDQ-moment} and \cor{SDQ-Re-Rs 1} instead of \cite[Lemma 5 and Corollary 5]{Miya2025}, respectively, \eq{RLn-e} is proved.

We next prove \eq{RLn-d}. This can be essentially done similarly to the proof of \eq{RLn-e}, but we cannot directly use \eq{BAR-RdRe} because the $\dd{E}^{(n_{k})}_{e}[\widecheck{R}^{(n_{k})}_{d}(0-) 1(L^{(n_{k})}_{1}(0) = q^{(n_{k})}_{x})]$ in it is slightly different from $\dd{E}^{(n_{k})}_{e}[(\widecheck{R}^{(n_{k})}_{d}(0-))^{j} 1(L^{(n_{k})}(0-) = q^{(n_{k})}_{x})]$ in \eq{RLn-d}. We fill this gap by \eq{Lqx-delta}. Namely, we replace $q^{(n_{k})}_{x}$ by $q^{(n_{k})}_{x}+1$, then \eq{RLn-d} is obtained in the exactly same way as \eq{RLn-e} using the fact that
\begin{align*}
  \dd{E}^{(n_{k})}_{e}[\widecheck{R}^{(n_{k})}_{d}(0-) 1(L^{(n_{k})}_{1}(0) = q^{(n_{k})}_{x}+1)] = \dd{E}^{(n_{k})}_{e}[\widecheck{R}^{(n_{k})}_{d}(0-) 1(L^{(n_{k})}(0-) = q^{(n_{k})}_{x})],
\end{align*}
because $L^{(n_{k})}_{1}(0) = L^{(n_{k})}(0-) + 1$. Thus, (i) is proved.

We next assume that \eq{Hn-x01} holds in a neighborhood of $x$, and prove that it implies that $\nu$ has a density in a neighborhood of $x$. For this, applying \lem{limit-hbn} to \eq{asym-BAR=<2} for $k \to \infty$, we have, for $\theta \le 0$,
\begin{align}
 \label{eq:SDQ-BAR3}
 & \frac 12 \int_{0}^{x} \sigma^{2}(u) \left(\beta(u) + \theta \right) e^{\theta u} d\nu(u) \nonumber\\
 & \quad  + \lim_{k \to \infty} \mu^{(n_{k})}(0) n^{1/2} \nu^{(n)}(\{0\}) + e^{\theta x} \int_{x}^{\infty} b(u) d\nu(u) = 0.
\end{align}
Substituting $\theta = 0$ in this equation and using $\beta(u) \sigma^{2}(u) = 2b(u)$, we have
\begin{align}
 \label{eq:Ln-x03}
 \lim_{k \to \infty} \mu^{(n_{k})}(0) n^{1/2} \nu^{(n_{k})}_{0}(\{0\}) = - \int_{0}^{\infty} b(u) \nu(du).
\end{align}
Hence, \eq{SDQ-BAR3} can be written as
\begin{align}
 \label{eq:SDQ-BAR4}
 & \int_{0}^{x} \left( \beta(u) + \theta \right) \sigma^{2}(u)e^{\theta u} d\nu(u) = \int_{0}^{\infty} \beta(u) \sigma^{2}(u) d\nu(u) - e^{\theta x} \int_{x}^{\infty} \beta(u) \sigma^{2}(u) d\nu(u).
\end{align}

Define a measure $\xi$ on $(\dd{R}_{+},\sr{B}_{+})$ by
\begin{align}
\label{eq:xi-distribution-1}
 \xi([0,x]) \equiv \int_{0}^{x} \sigma^{2}(u) d\nu(u), \qquad x \ge 0,
\end{align}
then $\xi$ is a finite measure on $(\dd{R}_{+},\sr{B}_{+})$ because $\sigma^{2}(u)$ is bounded, and \eq{SDQ-BAR4} becomes
\begin{align}
 \label{eq:SDQ-BAR5}
 & \int_{0}^{x} \left( \beta(u) + \theta \right) e^{\theta u} d\xi(u) = \int_{0}^{\infty} \beta(u) d\xi(u) - e^{\theta x} \int_{x}^{\infty} \beta(u) d\xi(u).
\end{align}
From this equation and the assumption on \eq{Hn-x01}, choosing $y > x$  in the neighborhood of $x$ in which \eq{Hn-x01} holds, we have
\begin{align*}
 & \int_{x}^{y} \left( (e^{\theta u} - e^{\theta x}) \beta(u) + \theta e^{\theta u} \right) d\xi(u) = (e^{\theta x} - e^{\theta y}) \int_{y}^{\infty} \beta(u) d\xi(u).
\end{align*}
Dividing both sides of this equation and letting $\theta \uparrow 0$, we have
\begin{align}
 \label{eq:SDQ-BAR6}
 & \int_{x}^{y} \left( (u - x) \beta(u) + 1 \right) d\xi(u) = (x-y) \int_{y}^{\infty} \beta(u) d\xi(u),
\end{align}
which yields
\begin{align*}
 & \xi((x,y]) \le (y-x) ((y-x)/2 + 1) {\beta^{*}} \xi([0,\infty)), \qquad 0 \le x < y,
\end{align*}
where recall that $\beta^{*} = \sup_{u \ge 0} |\beta(u)| < \infty$. Hence, $\xi([0,x])$ is Lipschitz continuous in a neighborhood of $x$, so $\xi$ has a density in the neighborhood of $x$. By its definition \eq{xi-distribution-1}, $\nu$ must have a density in the neighborhood of $x$ as well.
\end{proof}

\begin{proof}[Proof of \thr{SDQ-main-1}]
By \lem{Helly}, for each subsequence of $\sr{S}_{sd}$, there is a further subsequence $\{\nu^{(n_{k})}; k \ge 1\}$ such that
\begin{align*}
  \nu^{(n_{k})} \os{v}{\to} \nu, \qquad k \to \infty. 
\end{align*}
In what follows, we show that $\nu$ has the density $h$ of \eq{SDQ-h} up to a constant multiplier for all vaguely converging subsequences.

As in the proof of \lem{Ln-x01}, we define $\xi$ by \eq{xi-distribution-1}. Then, $\xi$ has a density because $\nu$ has a density by (\sect{preliminaries}.h). Denote this density function by $g$. Furthermore, (\sect{preliminaries}.h) implies the condition (\sect{p-main}.a) by \lem{Ln-x01}, so we have the asymptotic BAR \eq{asym-BAR=<2}. Hence, \eq{SDQ-BAR5} is obtained by the exactly same arguments in the proof of \lem{Ln-x01}. 

Thus, we can differentiate \eq{SDQ-BAR5} on $x$ for each fixed $\theta \le 0$, which yields
\begin{align*}
 & (\beta(x) + \theta) e^{\theta x} g(x) = - \theta e^{\theta x} \int_{x}^{\infty} \beta(u) g(u) du + e^{\theta x} \beta(x) g(x).
\end{align*}
This equation can be written as
\begin{align}
 \label{eq:SDQ-BAR7}
 g(x) + \int_{x}^{\infty} \beta(u) g(u) du = 0.
\end{align}
Obviously, the solution $g$ of this integral equation is
\begin{align}
 \label{eq:g}
 g(x) = C' \exp\left(\int_{0}^{x} \beta(u) du\right),
\end{align}
where $C'$ is a constant.

Let $h(x)$ be the density function of $\nu$, then $h(x) = g(x)/\sigma^{2}(x)$ by \eq{xi-distribution-1}, and, by \eq{g},
\begin{align*}
 h(x) = \frac {C'}{\sigma^{2}(x)} \exp\left(\int_{0}^{x} \beta(u) du\right), \qquad x \ge 0.
\end{align*}
This density agrees with the density $h$ of \eq{SDQ-h} up to a constant multiplier. Hence, $\nu$ is a probability measure if only if 
\begin{align*}
  \int_{0}^{\infty} h(x) dx < \infty.
\end{align*}
Let $\ul{\sigma} = \inf_{u \in \dd{R}_{+}}$ and $\ol{\sigma} = \sup_{u \in \dd{R}_{+}}$. Since $0 < \ul{\sigma} \le \ol{\sigma} < \infty$ by (f1) of (\sect{preliminaries}.f), this condition is equivalent to
\begin{align}
\label{eq:g-stable}
  \int_{0}^{\infty} g(x) dx < \infty.
\end{align}
We show that \eq{g-stable} is equivalent to $b_{\infty} < 0$. First assume that $b_{\infty} < 0$. Then, there are $\delta > 0$ and $x_{0} \ge 0$ such that $b(u) < -\delta$ for $u \ge x_{0}$, so it follows from \eq{SDQ-BAR7} that
\begin{align*}
  \int_{x_{0}}^{\infty} g(u) du < \frac {\ol{\sigma}}{2 \delta} g(x_{0}) < \infty.
\end{align*}
Hence, we have \eq{g-stable}. We next assume that $b_{\infty} \ge 0$. In this case, there is $x_{1} \ge 0$ for any $\varepsilon > 0$ such that $b(u) > -\varepsilon$ for $u \ge x_{1}$, and, from \eq{SDQ-BAR7}, we have
\begin{align*}
  \int_{x_{0}}^{\infty} g(u) du > \frac {\ul{\sigma}}{2 \varepsilon} g(x_{0}) \to \infty, \qquad \varepsilon \downarrow 0.
\end{align*}
Hence, \eq{g-stable} does not hold. This completes the proof of \thr{SDQ-main-1}.
\end{proof}

\subsection{Proof of \thr{MLQ-main-1}}
\label{sec:p-MLQ-main-1}

We recall some notation for the multi-level queue. For its $n$-th system, its $i$-th level is $\ell^{(n)}_{i} \equiv n^{1/2} \ell_{i}$ for strictly increasing sequence $\{\ell_{i}; i =1,2,\ldots\}$, and $\lambda^{(n)}(u) = \lambda^{(n)}_{i}$, $\mu^{(n)}(u) = \mu^{(n)}_{i}$ and $\beta^{(n)}(u) = \beta^{(n)}_{i}$ for $u \in S^{(n)}_{i}$, where $S^{(n)}_{1} = [0,q^{(n)}_{1}]$ and $S^{(n)}_{i} = (q^{(n)}_{i-1}, \ell^{(n)}_{i}]$ for $i \ge 2$. For $n \to \infty$, $\widehat{b}^{(n)}_{i} \to b_{i}$, $\widehat{\sigma}^{(n)}_{i} \to \sigma_{i}$ and $\widehat{\beta}^{(n)}_{i} \to \beta_{i}$ (see \eq{2g-heavy}). Furthermore, the heavy traffic condition (\sect{preliminaries}.f) holds by \lem{MLQ-heavy}.

For the multi-level queue, we prove the condition (\sect{preliminaries}.h). Once (\sect{preliminaries}.h) is proved, we have \thr{MLQ-main-1} by \thr{SDQ-main-1}. To prove (\sect{preliminaries}.h), we start with the asymptotic BAR's \eq{asym-BAR=<1} and \eq{asym-BAR>1} for an arbitrarily chosen vaguely convergent subsequence $\{\nu_{n_{k}}; k \ge 1\}$ of $\sr{S}_{sd}$. That is, there is a subprobability measure $\nu$ such that
\begin{align*}
 \nu^{(n_{k})} \os{v}{\to} \nu, \qquad k \to \infty.
\end{align*}
Put $\theta = 0$ in \eq{asym-BAR=<1} for $\{\nu^{(n_{k})}; k \ge 1\}$, and let $k \to \infty$, then we have \eq{SDQ-BAR3} because $\beta(u) \sigma^{2}(u) = 2b(u)$ and $H^{(n_{k})}_{x}$ disappears by $\theta=0$. Hence, \eq{Ln-x03} is obtained again. Let
\begin{align*}
  B_{x} = - \int_{x}^{\infty} b(u) \nu(du), \qquad x \in \dd{R}_{+},
\end{align*}
then, applying the subsequence $\{\nu^{(n_{k})}; k \ge 1\}$ to \eq{asym-BAR=<1}, and, letting $k \to \infty$, we have
\begin{align*}
 & \frac 12 \int_{0}^{x} (\beta(u) + \theta) e^{\theta u} \sigma^{2}(u) \nu(du) = - B_{0} + e^{\theta x} B_{x} - \lim_{k \to \infty} \theta e^{\theta x} H^{(n_{k})}_{x}, \qquad x \in \dd{R}_{+}, \ \theta \le 0,
\end{align*}
by \eq{Ln-x03}. Hence, $\lim_{k \to \infty} H^{(n_{k})}_{x}$ exists and is finite. Denote this limit by $H_{x}$, then we have
\begin{align}
 \label{eq:BAR-nu=<x}
 & \frac 12 \int_{0}^{x} (\beta(u) + \theta) e^{\theta u} \sigma^{2}(u) \nu(du) = - B_{0} + e^{\theta x} B_{x} - \theta e^{\theta x} H_{x}, \qquad x \in \dd{R}_{+}, \ \theta \in \dd{R},
\end{align}
where the range of $\theta$ is extended from $\theta \le 0$ to $\theta \in \dd{R}$ because the lefthand side of this equation is the integral on a finite interval. For $0 < x < y$, from \eq{BAR-nu=<x}, we have
 \begin{align}
 \label{eq:BAR-nu=<x<y}
 & \frac 12 \int_{x}^{y} (\beta(u) + \theta) e^{\theta u} \sigma^{2}(u) \nu(du) \nonumber\\
 & \quad = e^{\theta y} B_{y} - e^{\theta x} B_{x} - \theta (e^{\theta y} H_{y} - e^{\theta x} H_{x}), \qquad x, y \in \dd{R}_{+}, \ \theta \in \dd{R},
\end{align}
Similarly, from \eq{asym-BAR>1}, we have
\begin{align}
 \label{eq:BAR-nu>x}
 & \frac 12 \int_{x}^{\infty} (\beta(u) + \theta) e^{\theta u} \sigma^{2}(u) \nu(du) = - e^{\theta x} B_{x} + \theta e^{\theta x} H_{x}, \quad \theta \le 0, \ x > 0,
\end{align}
where $\theta$ is limited to be nonpositive because the lefthand side of this equation is the integral over $(x,\infty)$.

Let $\nu_{i}$ be the restriction of $\nu$ on $S_{i}$, that is, $\nu_{i}(B) = \nu(B \cap S_{i})$ for $B \in \sr{B}$, and let $\varphi_{i}(\theta)$ be the mgf of $\nu_{i}$, where mgf means moment generating function. Namely,
\begin{align*}
  \varphi_{i}(\theta) = \int_{S_{i}} e^{\theta u} \nu_{i}(du), \qquad \theta \le 0,\ i \in J_{1,K+1},
\end{align*}
where recall that $i < \infty$ for $i \in J_{1,K+1}$, so $J_{1,K+1} = \{i \ge 1; i < \min(K+2,\infty)\}$. Note that $S_{K+1} = \emptyset$ for $K=\infty$.

Using \eq{BAR-nu=<x<y}, we compute $\varphi_{i}(\theta)$ for $i \in J_{1,K}$, and compute $\varphi_{K+1}(\theta)$ by \eq{BAR-nu>x} if $K < \infty$. We prove (\sect{preliminaries}.h) by these computations and the following fact that
\begin{align}
\label{eq:mgf-exp}
  \frac {e^{(\beta_{i} + \theta) \ell_{j}} - 1} {\beta_{i} + \theta} = \int_{0}^{\ell_{j}} e^{\theta x} e^{\beta_{i} x} dx, \qquad i, j \in J_{1,K}, 
\end{align}
is the mgf of the exponential function $e^{\beta_{i} x}$ if $\beta_{i} \not= 0$ and the constant function $1$ if $\beta_{i} = 0$ on $[0,\ell_{j}] = \cup_{j'=1}^{j} S_{j'}$. 
 
We first compute the mgf $\varphi_{1}(\theta)$ of $\nu_{1}$. For this, put $x=\ell_{1}$ in \eq{BAR-nu=<x}. Since $b(u) = b_{1}$, $\beta(u) = \beta_{1}$ and $\sigma^{2}(u) = \sigma^{2}_{1}$ for $u \in S_{1}$, we have
\begin{align}
\label{eq:phi-S10}
 & \frac 12 (\beta_{1} + \theta) \sigma^{2}_{1} \varphi_{1}(\theta) = - B_{0} + e^{\theta \ell_{1}} B_{\ell_{1}} - \theta e^{\theta \ell_{1}} H_{\ell_{1}}, \qquad \theta \in \dd{R}.
\end{align}
Letting $\theta = - \beta_{1}$ in this equation, we have
\begin{align*}
 & 0 = - B_{0} + e^{-\beta_{1} \ell_{1}} B_{\ell_{1}} + \beta_{1} e^{-\beta_{1} \ell_{1}} H_{\ell_{1}}.
\end{align*}
Subtracting this equation from \eq{phi-S10} side by side, we have
\begin{align*}
  \frac 12 (\beta_{1} + \theta) \sigma^{2}_{1} \varphi_{1}(\theta) & = (e^{\theta \ell_{1}} - e^{- \beta_{1} \ell_{1}}) (B_{\ell_{1}} + \beta_{1} H_{\ell_{1}}) - (\beta_{1} + \theta) e^{\theta \ell_{1}} H_{\ell_{1}} 
\end{align*}
Dividing both sides of this equation by $\beta_{1} + \theta$ for $\theta \not= - \beta_{1}$, we have
\begin{align}
\label{eq:phi-S11}
  \frac 12 \sigma^{2}_{1} \varphi_{1}(\theta) = \frac {e^{(\beta_{1} + \theta) \ell_{1}} - 1} {\beta_{1} + \theta} e^{- \beta_{1} \ell_{1}}( B_{\ell_{1}} + \beta_{1} H_{\ell_{1}}) - e^{\theta \ell_{1}} H_{\ell_{1}}.
\end{align}
By \eq{mgf-exp}, this shows that $\nu_{1}$ has a density on [$0,\ell_{1})$ while it has a point mass at $\ell_{1}$ if and only if $H_{\ell_{1}} \not= 0$.

We next consider the mgf of $\nu_{i}$ on $S_{i}$ for $2 \le i < \min(K+1,\infty)$. For this, put $x=\ell_{i-1}$ and $y=\ell_{i}$ in \eq{BAR-nu=<x<y}. Since $b(u) = b_{i}$, $\beta(u) = \beta_{i}$ and $\sigma^{2}(u) = \sigma^{2}_{i}$ for $u \in S_{i}$, we have
\begin{align*}
 & \frac 12 (\beta_{i} + \theta) \sigma^{2}_{i} \varphi_{i}(\theta) = e^{\theta \ell_{i}} B_{\ell_{i}} - e^{\theta \ell_{i-1}} B_{\ell_{i-1}} - \theta (e^{\theta \ell_{i}} H_{\ell_{i}} - e^{\theta \ell_{i-1}} H_{\ell_{i-1}}) \nonumber\\
 & \quad = e^{\theta \ell_{i}} (B_{\ell_{i}} + \beta_{i} H_{\ell_{i}}) - e^{\theta \ell_{i-1}} (B_{\ell_{i-1}} + \beta_{i} H_{\ell_{i-1}}) - (\beta_{i} + \theta) (e^{\theta \ell_{i}} H_{\ell_{i}} - e^{\theta \ell_{i-1}} H_{\ell_{i-1}}).
\end{align*}
Dividing both sides of this equation by $\beta_{i} + \theta$ for $\theta \not= - \beta_{i}$, we have
\begin{align}
\label{eq:phi-Si2}
  \frac 12 \sigma^{2}_{i} \varphi_{i}(\theta) & = \frac {e^{(\beta_{i} + \theta) \ell_{i}} - 1} {\beta_{i} + \theta} e^{- \beta_{i} \ell_{i}} (B_{\ell_{i}} + \beta_{i} H_{\ell_{i}}) - \frac {e^{(\beta_{i} + \theta) \ell_{i-1}} - 1} {\beta_{i} + \theta} e^{- \beta_{i} \ell_{i-1}} (B_{\ell_{i-1}} + \beta_{i} H_{\ell_{i-1}}) \nonumber\\
  & \quad + e^{\theta \ell_{i-1}} H_{\ell_{i-1}} - e^{\theta \ell_{i}} H_{\ell_{i}}.
\end{align}
Hence, by \eq{mgf-exp} and the fact that $\varphi_{i}(\theta)$ is the mgf of $\nu_{i}$ on $S_{i} = (\ell_{i-1},\ell_{i}]$, we must have
\begin{align}
\label{eq:Hi=0}
  H_{\ell_{i-1}} = 0, \qquad H_{\ell_{i}} \le 0, \qquad 2 \le i < \min(K+1,\infty).
\end{align}
Hence, $\nu_{i}$ has a density on $(\ell_{i-1}, \ell_{i})$ and a point mass at $\ell_{i}$ if $H_{\ell_{i}} < 0$. Furthermore, we can see that $H_{\ell_{j}} = 0$ for $1 \le j \le i-1$, and therefore $\nu_{j}$ has a density on $S_{j}$ for $1 \le j \le i-1$. Namely, $\nu$ has a density on $[0,\ell_{i-1}]$ for $2 \le i < \min(K+1,\infty)$.
This proves (\sect{preliminaries}.h) for $K=\infty$.

Thus, it remains to prove (\sect{preliminaries}.h) for $K < \infty$. In this case, $\nu$ has a density on $[0,\ell_{K-1}]$, so it remains to prove that $\nu_{i}$ has a density on $S_{K+1} = (\ell_{K}, \infty)$. To see this, we derive the following equation from \eq{BAR-nu>x} for $x=\ell_{K}$.
\begin{align}
 \label{eq:phi-SK1}
 & \frac 12 (\beta_{K+1} + \theta) \sigma^{2}_{K+1} \varphi_{K+1}(\theta) = - e^{\theta \ell_{K}} B_{\ell_{K}} + \theta e^{\theta \ell_{K}} H_{\ell_{K}}, \qquad \theta \le 0.
\end{align}
This can be written as
\begin{align}
 \label{eq:phi-SK2}
 & \frac 12 \sigma^{2}_{K+1} \varphi_{K+1}(\theta) = \frac {- e^{\theta \ell_{K}}}{\beta_{K+1} + \theta} (B_{\ell_{K}} + \beta_{K+1} H_{\ell_{K}}) + e^{\theta \ell_{K}} H_{\ell_{K}}, \qquad \theta \le 0.
\end{align}
Since the lefthand side of this equation is the mgf of $\nu_{K+1}$ on $S_{K+1} = (\ell_{K}, \infty)$, $\nu(\{\ell_{K}\}) = 0$. On the other hand, in its righthand side, $\frac {- e^{\theta \ell_{K}}}{\beta_{K+1} + \theta}$ is the mgf of $\exp(\beta_{K+1} x)$ on $(\ell_{K}, \infty)$ and $e^{\theta \ell_{K}} H_{\ell_{K}}$ is the mgf of the mass $H_{\ell_{H}}$ at $\ell_{K}$. Hence, we must have $H_{\ell_{K}} = 0$, and therefore $\nu_{K+1}$ on $S_{K+1}$ has a density for $K<\infty$.

Thus, the proof of \thr{MLQ-main-1} is completed. 

\appendix

\section*{Appendix}
\setnewcounter
\setcounter{section}{1}

\subsection{Proof of \lem{Ln-unique}}
\label{app:Ln-unique}

In the proof of this lemma, the index $n$ of the state-dependent queue is fixed. So, we remove the superscript ``$^{(n)}$'' from all the notations for the $n$-th state-dependent queue throughout this proof. For example, $\lambda^{(n)}(u)$, $\mu^{(n)}(u)$ and $L^{(n)}(u)$ are denoted by $\lambda(u)$, $\mu(u)$ and $L(u)$, respectively. In particular, \eq{ZLn-t} and \eq{Ln-t1} can be written as
\begin{align}
\label{eq:ZLn-tA}
 & Z(L,t) = A(U(L,t)) - S(V(L,t)),\\
 \label{eq:Ln-t1A}
 & L(t) = L(0-) + Z(L,t), \qquad t \ge 0.
\end{align}

We prove this lemma by induction over the times when either an arrival or a departure is counted. Denote its $k$-th time by $t_{k}$, and inductively construct $t_{k}$ and $L(s)$ for $s \in [0,t_{k}]$. Define
\begin{align*}
 & L_{k}(s) = L(s) 1(s \le t_{k}) + L(t_{k}) 1(s > t_{k}), \qquad s \ge 0, \qquad k=0,1,\ldots.
\end{align*}
Consider the case that $k=0$. In this case, from \eq{ZLn-tA} and \eq{Ln-t1A},
\begin{align*}
  L(0) = L(0-) + Z(L,0) = L(0-) + A(0) - S(0).
\end{align*}
We put $t_{0} = 0$, then $L(s)$ for $s \in [0,t_{0}]$ is uniquely determined. We consider this $t_{0}$ and $L(0)$ to be the initial state of $L(\cdot)$. For the next step, we purt $t_{e,0} = t_{d,0} = 0$ and $\ell_{e}(0) = \ell_{d}(0) = 0$, where the meaning of $\ell_{e}(i)$ and $\ell_{d}(i)$ for $i \ge 0$ will be explained later.

Then, for $k=1$, define $t_{e,1}$, $t_{d,1}$ and $t_{1}$ by
\begin{align}
\label{eq:t-step1}
\begin{split}
 & t_{e,1} = \inf \left\{t > t_{e,\ell_{e}(0)}; \int_{t_{e,\ell_{e}(0)}}^{t} \lambda(L_{0}(s)) ds = T_{A}(0) \right\},\\
 & t_{d,1} = \inf \left\{t > t_{d,\ell_{d}(0)} ; \int_{t_{d,\ell_{d}(0)}}^{t} \mu(L_{0}(s)) 1(L_{0}(s) > 0) ds = T_{S}(0) \right\},\\
 & t_{1} = t_{e,1} \wedge t_{d,1}. 
\end{split}
\end{align}
Note that these $t_{e,1}$ and $t_{d,1}$ may be different from $t^{(n)}_{e,1}$ and $t^{(n)}_{d,1}$, respectively, which are defined above \eq{Re-t2}. We can see that $L_{0}(s) = L_{0}(0)$ for $s \in [0, t_{1})$, so $t_{e,1}$ and $t_{d,1}$ are uniquely determined and $t_{1}$ is the first time when either an arrival or a departure is counted after time $0$. From \eq{ZLn-tA}, we have
\begin{align}
\label{eq:Z-step1}
 \Delta Z(L,t_{1}) = Z(L,t_{1}) - Z(L,t_{1}-) = \begin{cases}
  1,  & t_{e,1} < t_{d,1}, \\
  -1, & t_{e,1} > t_{d,1},  \\
  0, & t_{e,1} = t_{d,1},
 \end{cases}
\end{align}
and define $L(t_{1})$ by \eq{Ln-t1A}, then $L(s)$ for $s \in [0,t_{1}]$ is uniquely determined by \eq{Ln-t1A}. This completes the first step of the induction. 

To proceed with the induction, assume that $t_{k}$, $t_{e,k}$, $t_{d,k}$, $\ell_{e}(k-1)$, $\ell_{d}(k-1)$ and $L_{k}(s)$ for $s \in [0,t_{k}]$ are defined, where $t_{k} = t_{e,k} \wedge t_{d,k}$, and $\ell_{e}(k-1)$ and $\ell_{d}(k-1)$ are the last arrival and departure numbers at time $t_{k}$. Then, similarly to \eq{t-step1}, define
\begin{align*} 
 & t_{e,k+1} = \inf \left\{t > t_{e,\ell_{e}(k)}; \int_{t_{e,\ell_{e}(k)}}^{t} \lambda(L_{k}(s)) ds = T_{A}(k) \right\},\\
 & t_{d,k+1} = \inf \left\{t > t_{d,\ell_{d}(k)} ; \int_{t_{d,\ell_{d}(k)}}^{t} \mu(L_{k}(s)) 1(L_{k}(s) > 0) ds = T_{S}(k) \right\},\\
 & t_{k+1} = t_{e,k+1} \wedge t_{d,k+1},
\end{align*}
where
\begin{align*}
  & \ell_{e}(k) = \begin{cases}
 \ell_{e}(k-1) + 1, & \mbox{ if } t_{e,k} \le t_{d,k},  \\
 \ell_{e}(k-1), & \mbox{ if } t_{e,k} > t_{d,k},
\end{cases} \quad
  & \ell_{d}(k) = \begin{cases}
 \ell_{d}(k-1), & \mbox{ if } t_{e,k} < t_{d,k},  \\
 \ell_{d}(k-1)+1, & \mbox{ if } t_{e,k} \ge t_{d,k}.
\end{cases}
\end{align*}
Similarly to $t_{e,1}$ and $t_{d,1}$, these $t_{e,k+1}$ and $t_{d,k+1}$ may be different from $t^{(n)}_{e,k+1}$ and $t^{(n)}_{d,k+1}$, respectively. Then, the change of the queue length at time $t_{k+1}$ is given by
\begin{align*}
 \Delta Z(L,t_{k}) = \begin{cases}
  1,  & t_{e,k} < t_{d,k}, \\
  -1, & t_{e,k} > t_{d,k},  \\
  0, & t_{e,k} = t_{d,k}.
 \end{cases}
\end{align*}
Hence, it is not hard to see that $L(s)$ for $s \in [0,t_{k+1}]$ is uniquely determined by \eq{Ln-t1A}. Thus, repeating this procedure, we have $\{L(s); s \in [0,t_{k}]\}$ for all $k \ge 0$ as the unique solution of \eq{Ln-t1}. Since $\lim_{k \to \infty} t_{k} = \infty$, we have uniquely constructed $L(t)$ for all $t \ge 0$.

\subsection{Proof of \lem{Xn-stability}}
\label{app:Xn-stability}

Similarly to \app{Ln-unique}, in the proof of this lemma, we remove the superscript ``$^{(n)}$'' from all the notations for the $n$-th state-dependent queue. In particular, the condition (\sect{preliminaries}.e) can be written as
\begin{mylist}{3}
\item [($\ast$)] $\gamma_{\infty} \equiv \lim_{x \to \infty} \gamma(x)$ exists and $\gamma_{\infty} < 0$ for $n \ge 1$, where $\gamma(u) = \lambda(u) - \mu(u)$.
\end{mylist}
Furthermore, we simply {denote} $U^{(n)}(L^{(n)},t)$ and $V^{(n)}(L^{(n)},t)$ by $U(t)$ and $V(t)$, respectively. We also denote by $t_{e,1}$ the first time when a customer arrives.

Let $|X|(t) = L(t) + R_{e}(t) + R_{d}(t)$ for $t \ge 0$. We show that there is a $\delta > 0$ such that
\begin{align}
 \label{eq:fluid-1}
 \lim_{x \to \infty} \frac 1x \dd{E}_{\vc{x}}[|X|(x \delta)] = 0,
\end{align}
where $\dd{E}_{\vc{x}}(Z) = \dd{E}\left[Z \left|X(0) = \vc{x} \right.\right]$ for a random variable $Z$ and $\vc{x} = ([x],x,x) \in \dd{Z}_{+} \times \dd{R}_{+}^{2}$. Thus, the initial state of $X(\cdot)$ is given by
\begin{align*}
 L(0) = [x], \qquad R_{e}(0) = T_{A}(0) = x, \qquad R_{d}(0) = T_{S}(0) = x,
\end{align*}
under $\dd{P}_{\vc{x}}$. Define the fluid scaled process of $L(\cdot)$ as
\begin{align*}
 \ol{L}_{x}(t) = \frac 1x L(t x), \qquad x > 0, t \ge 0.
\end{align*}
{If} $\lim_{x \to \infty} \ol{L}_{x}(t)$ exists for $t \ge 0$, we denote it by $\ol{L}(t)$, and call $\ol{L}(\cdot) \equiv \{\ol{L}(t); t \ge 0\}$ a fluid limit of $L(\cdot)$. If \eq{fluid-1} holds and if $\ol{L}(\cdot)$ is uniquely obtained and
\begin{align}
\label{eq:fluid-vanish}
  \exists \delta > 0 \; \mbox{ such that } \; \ol{L}(t) = 0, \quad \forall t \ge \delta,
\end{align}
then \lem{Xn-stability} is proved in a similar way to the proof of Theorem 4.2 of \cite{Dai1995}.

To show \eq{fluid-1}, we prove that there are $\delta_{e}, \delta_{d}, \delta_{L} > 0$ such that
\begin{align}
 \label{eq:fluid-Re}
 & \lim_{x \to \infty} \frac 1x \dd{E}_{\vc{x}}[R_{e}(x \delta_{e})] = 0,\\
 \label{eq:fluid-Rd}
 & \lim_{x \to \infty} \frac 1x \dd{E}_{\vc{x}}[R_{d}(x \delta_{d})] = 0,\\
 \label{eq:fluid-L1}
 & \lim_{x \to \infty} \dd{E}_{\vc{x}}[\ol{L}_{x}(\delta_{L})] = 0.
\end{align}
Since $\ol{L}_{x}(\delta_{L}) = \frac 1x L(x \delta_{L})$, these facts {imply} \eq{fluid-1} for $\delta = \max(\delta_{e}, \delta_{d}, \delta_{L})$.

We first prove \eq{fluid-Re}. From the definition \eq{Re-t1} of $R_{e}(\cdot)$, we have
\begin{align}
 \label{eq:Re-t3}
 R_{e}(t) = A(U(t)) - \sum_{i=1}^{A(U(t))} (1-T_{A}(i)) - U(t) + T_{A}(0).
\end{align}
Define $M_{e}(\cdot)$ as
\begin{align*}
 M_{e}(t) = \sum_{i=1}^{A(U(t))} (1-T_{A}(i)), \qquad t \ge 0,
\end{align*}
then $M_{e}(\cdot)$ is shown to be an $\dd{F}$-martingale similarly to Lemma 3.1 of \cite{AtarMiya2026}, and \eq{Re-t3} becomes
\begin{align}
 \label{eq:Re-Mt}
 R_{e}(t) = A(U(t)) - M_{e}(t) - U(t) + T_{A}(0).
\end{align}
\begin{remark}
 \label{rem:DM-decomposition-A}
 \eq{Re-Mt} can be written as
\begin{align}
\label{eq:DM-decomposition-A}
  A(U(t)) = U(t) + R_{e}(t) - T_{A}(0) + M_{e}(t).
\end{align}
Similarly, let $M_{d}(t) = \sum_{i=1}^{S(V(t))} (1-T_{S}(i))$, then $M_{d}(\cdot)$ is a martingale, and
\begin{align}
\label{eq:DM-decomposition-S}
  S(V(t)) = U(t) + R_{d}(t) - T_{S}(0) + M_{d}(t).
\end{align}
Thus, we have semi-martingale decompositions for counting processes $A(U(\cdot))$ and $S(V(\cdot))$, which are detailed in \cite[Section 3]{AtarMiya2026}.
\end{remark}

Since $\dd{E}_{\vc{x}}[M(t)] = 0$ and $T_{A}(0)=x$ under $\dd{P}_{\vc{x}}$, it follows from \eq{Re-Mt} that
\begin{align}
 \label{eq:Re-t4}
 \dd{E}_{\vc{x}}[R_{e}(t)] & = \dd{E}_{\vc{x}}[A(U(t)) - U(t)] + x.
\end{align}
Let $U_{1}(t) = U(t) - U(t_{e,1})$ for $t \ge t_{e,1}(L)$ and let $A_{1}(t) = 1 + \inf \{k \ge 0; t_{A,k+1} - T_{A}(0) > t\}$, then $U(t)$ and $A(U(t))$ can be written as
\begin{align*}
 & U(t) = U(t \wedge t_{e,1}) + U_{1}(t) 1(t \ge t_{e,1}),\\
 & A(U(t)) = A_{1}(U_{1}(t)) 1(t \ge t_{e,1}).
\end{align*}
Note that $ A_{1}(\cdot)$ is a non-delayed renewal process with {inter-counting} times subject to $T_{A}$. Hence, \eq{Re-t4} becomes
\begin{align}
 \label{eq:Re-t5}
 \dd{E}_{\vc{x}}[R_{e}(t)] & = \dd{E}[A_{1}(U_{1}(t)) - U_{1}(t))1(t \ge t_{e,1})] \nonumber\\
 & \quad + \dd{E}_{\vc{x}}[x - U(t \wedge t_{e,1})],
\end{align}
where $U(t_{e,1}) = T_{A}(0) = x$.

By (\sect{preliminaries}.c) and the strong law of large numbers, $\lim_{y \to \infty} A(y)/y = 1$ $a.s.$, so there is a $y_{0}$ for any $\varepsilon > 0$ such that $|A(y) - y| < \varepsilon y$ for all $y \ge y_{0}$ $a.s.$, where $y_{0}$ may depend on the sample $\omega \in \Omega$, that is, it is a finite-valued random variable. Hence,
\begin{align*}
 |A(U(t)) - U(t)| < \varepsilon U(t) \le \varepsilon \lambda_{\sup} t, \qquad U(t) \ge y_{0}, \qquad {a.s.}
\end{align*}
This implies that, for any $\delta > 0$,
\begin{align*}
 \left|\dd{E}[(A_{1}(U_{1}(\delta y)) - U_{1}(\delta y))1(y_{0} \le U_{1}(\delta y))]\right| < \varepsilon \lambda_{\sup} \delta y,  \qquad \forall y > t_{e,1}/\delta.
\end{align*}
Hence,
\begin{align}
 \label{eq:A1U1}
 \lim_{y \to \infty} \frac 1y \left|\dd{E}[(A_{1}(U_{1}(\delta y)) - U_{1}(\delta y))1(y_{0} \le U_{1}(\delta y))]\right| = 0.
\end{align}
On the other hand, 
\begin{align*}
 \lim_{y \to \infty} \frac 1y \dd{E}[(A_{1}(U_{1}(\delta y)) - U_{1}(\delta y))1(U_{1}(\delta y) < y_{0})] = 0,
\end{align*}
because $U_{1}(\delta y) > \delta \lambda_{\inf} y$ and $y_{0} < \infty$. Combining this with \eq{A1U1}, we have
\begin{align*}
 \lim_{y \to \infty} \frac 1y \dd{E}[(A_{1}(U_{1}(\delta y)) - U_{1}(\delta y))] = 0,
\end{align*}
and therefore
\begin{align}
\label{eq:A1U2}
 & \lim_{y \to \infty} \frac 1y \dd{E}[(A_{1}(U_{1}(\delta y)) - U_{1}(\delta y))1(\delta y \ge t_{e,1})] = 0,
\end{align}
because
\begin{align*}
 & \lim_{y \to \infty} \frac 1y \dd{E}[(A_{1}(U_{1}(\delta y)) - U_{1}(\delta y))1(\delta y < t_{e,1})] = 0.
\end{align*}
Applying \eq{A1U2} to \eq{Re-t5}, we have \eq{fluid-Re} for $\delta_{e} = \delta$ if
\begin{align}
\label{eq:Un-tne}
  \lim_{x \to \infty} \frac 1x \dd{E}_{x}[x - U(\delta x \wedge t_{e,1})] = 0.
\end{align}
Since $t_{e,1} \lambda_{\inf} \le T_{A}(0) = x$ under $\dd{P}_{\vc{x}}$, let $\delta = 1/\lambda_{\inf}$, then $U(\delta x \wedge t_{e,1}) = U(t_{e,1}) = x$ under $\dd{P}_{\vc{x}}$. Hence, choosing $\delta_{e} = 1/\lambda_{\inf}$, we have \eq{Un-tne}, and therefore \eq{fluid-Re} is proved. By exactly the same arguments, we also have
\eq{fluid-Rd} for $\delta_{d} = 1/\mu_{\inf}$ .

It remains to prove \eq{fluid-vanish} and \eq{fluid-L1}. In what follows, we prove them in five steps.\medskip

\noindent {\it Step 1}: We derive a flow balance equation for the fluid scaled process $\ol{L}_{y}(\cdot)$ from the equation \eq{Ln-t1} of the queue length process. For $y \ge 1$, let $L(0) = [y]$, then it follows from \eq{Ln-t1} that
\begin{align}
\label{eq:fluid-2}
 & \ol{L}_{y}(t) = \frac 1y L(y\, t) = 1 + \frac 1y \left( A(U(y\, t)) - S(V(y\, t))\right), \qquad t \ge 0.
\end{align}
Let $\ol{A}_{y}(t) = \frac 1y A(yt)$ and $\ol{S}_{y}(t) = \frac 1y S(yt)$, then \eq{fluid-2} can be written as
\begin{align}
\label{eq:fluid-3}
	\ol{L}_{y}(t) & = 1 + \ol{A}_{y}\left(\int_{0}^{t} \lambda(y \ol{L}_{y}(s)) ds \right) \nonumber\\
	& \quad - \ol{S}_{y}\left(\int_{0}^{t} \mu(y \ol{L}_{y}(s)) 1(\ol{L}_{y}(s) > 0) ds \right),
\end{align}
because
\begin{align}
\label{eq:Un-yt}
  U(y\, t) = \int_{0}^{y\, t} \lambda(L(s)) ds = y \int_{0}^{t} \lambda(L(y\, s)) ds = y \int_{0}^{t} \lambda(y \ol{L}_{y}(s)) ds
\end{align}
and similarly
\begin{align}
\label{eq:Vn-yt}
  V(y\, t) = y \int_{0}^{t} \mu(y \ol{L}_{y}(s))1( \ol{L}_{y}(s) > 0) ds.
\end{align}
Thus, we have the equation \eq{fluid-3} for the fluid scaled process $\ol{L}_{y}(\cdot)$.\medskip

\noindent {\it Step 2}: We show that, for any sequence in $\dd{R}_{+}$ which diverges, there is its subsequence $\{y_{k}; k \ge 1\}$ such that $\ol{L}_{y_{k}}(\cdot)$ converges weakly to a continuous process, which is denoted by $\ol{L}(\cdot)$ as we already defined. We prove this weak convergence following the arguments in \cite{AtarMiya2026}. Let $\sr{K} = \{y_{k}; k \ge \}$, and let
\begin{align*}
  U_{y}(t) = \int_{0}^{t} \lambda(y \ol{L}_{y}(s)) ds, \quad V_{y}(t) = \int_{0}^{t} \mu(y \ol{L}_{y}(s)) 1(\ol{L}_{y}(s) > 0) ds, 
\end{align*}
then $\{\ol{A}_{y}(U_{y}(\cdot)); y \in \sr{K}\}$ and $\{\ol{S}_{y}(V_{y}(\cdot)); y \in \sr{K}\}$ are shown to be $C$-tight for any countable set $\sr{K} \subset \dd{N}$ such that $\sup_{y \in \sr{K}} y = \infty$ by similar arguments in Step 1 in the proof of \cite[Lemma 4.4]{AtarMiya2026}. Here, a sequence of processes, $\{X_{y}(\cdot); y \in \sr{K}\}$, is called $C$-tight if the sequence of the probability laws of them is tight, and there is its subsequence $\{y_{k}; k \ge 1\} \subset \sr{K}$ such that $X_{y_{k}}(\cdot)$ converges weakly to a continuous process as $k \to \infty$ (see \cite[Section VI.3.25]{JacoShir2003} and \cite[Lemma 4.8]{AtarMiya2026}). Hence, the righthand side of \eq{fluid-3} is $C$-tight, and threfore $\{\ol{L}_{y}(\cdot); y \in \sr{K}\}$ is also $C$-tight.

Thus, we can find a subsequence $\{y'_{k}; n \ge 1\} \subset \sr{K}$ such that
\begin{align*}
  \left\{\ol{L}_{y'_{k}}(t), \ol{A}_{y'_{k}}\left(\int_{0}^{t} \lambda(y'_{k} \ol{L}_{y'_{k}}(s)) ds\right), \ol{S}_{y'_{k}}\left(\int_{0}^{t} \mu(y'_{k} \ol{L}_{y'_{k}}(s)) 1(\ol{L}_{y'_{k}}(s) > 0) ds \right); t \ge 0\right\}
\end{align*}
converge weakly to continuous processes. Hence, we can find a further subsequence $\{y''_{n}; k \ge 1\} \subset \sr{K}$, by which they converges almost surely to these continuous processes for each fixed $t \ge 0$. We will work on this sub-sub sequence $\{y''_{k}; k \ge 1\}$, but, in what follows, we simply write $y''_{k}$ as $y$ for convenience.
\medskip

\noindent {\it Step 3}: We rewrite \eq{fluid-2} so that it is tractable for computing the limit of $\ol{L}_{y}(\cdot)$ as $y \to \infty$, and derive an equation for the fluid process $\ol{L}_{y}(\cdot)$. To this end, we use the semi-martingale decompositions \eq{DM-decomposition-A} and \eq{DM-decomposition-S} in \rem{DM-decomposition-A}, from which we have.
\begin{align}
\label{eq:A-DM-decomposition-A}
  \frac 1y A(U(yt)) = \frac 1y \left(U(yt) + R_{e}(yt) - T_{A}(0) + M_{e}(yt)\right),\\
\label{eq:S-DM-decomposition-A}
  \frac 1y S(V(yt)) = \frac 1y \left(V(yt) + R_{d}(yt) - T_{S}(0) + M_{d}(yt)\right).
\end{align}
Substituting these formulas into \eq{fluid-2} and using \eq{Un-yt} and \eq{Vn-yt}, we have
\begin{align}
\label{eq:fluid-6}
  \ol{L}_{y}(t) & = 1 + \int_{0}^{t} \lambda(y \ol{L}_{y}(s)) ds - \int_{0}^{t} \mu(y \ol{L}_{y}(s))1( \ol{L}_{y}(s) > 0) ds \nonumber\\
  & \quad + \frac 1y \left(D_{e}(yt) + D_{d}(yt)\right), \quad t \ge 0,
\end{align}
where
\begin{align*}
 & D_{e}(t) = R_{e}(t) - T_{A}(0) + M_{e}(t), \qquad D_{d}(t) = R_{d}(t) - T_{S}(0) + M_{d}(t).
\end{align*}
We here note that $y^{-1} R_{e}(yt)$ and $y^{-1} R_{d}(yt)$ vanish almost surely as $y \to \infty$ because there is a constant $c > 0$ such that $y^{-1} \sup_{s \le cy^{2}} R_{e}(s) \to 0$ in probability as $y \to \infty$ by the arguments in Step 3 of \cite[Section 4.5]{AtarMiya2026}. Furthermore, $y^{-1} M_{e}(yt)$ and $y^{-1} M_{d}(yt)$ vanish almost surely as $y \to \infty$ by the strong law of large numbers for renewal processes $A(\cdot)$ and $S(\cdot)$. Hence, $y^{-1} (D_{e}(yt) + D_{e}(yt))$ vanishes as $y \to \infty$, and therefore it follows from \eq{fluid-6} that
\begin{align}
\label{eq:fluid-4}
  \ol{L}(t) = \lim_{y \to \infty} \ol{L}_{y}(t) = 1 + \lim_{y \to \infty} \frac 1y \left(U(yt) - V(yt)\right) \ge 0.
\end{align}

\noindent {\it Step 4}: We further compute the right side of \eq{fluid-4}, by which $\ol{L}(\cdot)$ will be uniquely determined in Step 5. Recall the drift $\gamma(u) = \lambda(u) - \mu(u) 1(u > 0)$ for $u \in \dd{R}_{+}$, and then
\begin{align}
\label{eq:UVn-yt}
 & \frac 1y \left(U(yt) - V(yt)\right) = \frac 1y \int_{0}^{yt} \gamma(y \ol{L}_{y}(s)) ds = \int_{0}^{t} \gamma(y \ol{L}_{y}(s)) ds \nonumber\\
  & \quad = \int_{0}^{t} \gamma(y \ol{L}_{y}(s)) 1(\ol{L}(s) > 0) ds + \int_{0}^{t} \gamma(y \ol{L}_{y}(s)) 1(\ol{L}(s) = 0) ds.
\end{align}

Since $\gamma_{\infty} = \lim_{y \to \infty} \gamma(y) < 0$ by (\sect{preliminaries}.e) (see ($\ast$) at the bigining of this proof) and $\ol{L}(s) > 0$ implies that $\ol{L}(s) > \varepsilon$ for some $\varepsilon > 0$, we have
\begin{align*}
 & \lim_{y \to \infty} \gamma(y \ol{L}_{y}(s)) 1(\ol{L}(s) > 0) = \gamma_{\infty} 1(\ol{L}(s) > 0) < 0.
\end{align*}
Because $\gamma(\cdot)$ is uniformly bounded, applying these limits to \eq{UVn-yt} as $y \to \infty$, we have
\begin{align*}
  \lim_{y \to \infty} \frac 1y \left(U(yt) - V(yt)\right) = \gamma_{\infty} \int_{0}^{t} 1(\ol{L}(s) > 0) ds + \ol{\gamma}_{0}(t),
\end{align*}
where $\ol{\gamma}_{0}(t) \equiv \lim_{y \to \infty} \int_{0}^{t} \gamma(y \ol{L}_{y}(s)) 1(\ol{L}(s) = 0) ds$ exists because the lefthand side exists. Hence, it follows from \eq{fluid-4} that
\begin{align}
 \label{eq:fluid-5}
 \ol{L}(t) & = 1 + \gamma_{\infty} \int_{0}^{t} 1(\ol{L}(s) > 0) ds + \ol{\gamma}_{0}(t) \ge 0, \qquad t \ge 0,
\end{align}

\noindent {\it Step 5}: We show that $\ol{L}(\cdot)$ is uniquely determined by \eq{fluid-5} and satisfies \eq{fluid-vanish}. To this end, we note that
\begin{align}
\label{eq:fluid-boundary}
  \int_{0}^{t} 1(\ol{L}(u) > 0) d\ol{\gamma}_{0}(u) = 0,
\end{align}
because
\begin{align*}
  |\ol{\gamma}_{0}(t) - \ol{\gamma}_{0}(s)| & \le  \left|\lim_{y \to \infty} \int_s^{t} \gamma(y \ol{L}_{y}(u)) 1(\ol{L}(u) = 0) \right| du \nonumber\\
  & \le |\lambda_{\sup} - \mu_{\inf}| \int_s^{t} 1(\ol{L}(u) = 0) du.
\end{align*}
Hence, $\ol{}\gamma_{0}(t)$ may have non-zero increments only when $\ol{L}(t) = 0$. Furthermore, for each $s_{0} > 0$, if $\ol{L}(s_{0}) = 0$ and if there is a $t_{1} > s_{0}$ such that $\ol{L}(t_{1}) > 0$, then there must be some $s_{1} \in (s_{0},t_{1})$ such that $\ol{L}(s_{1}) = 0$ and $\ol{L}(s) > 0$ for $s \in (s_{1},t_{1}]$ because $\ol{L}(\cdot)$ is a continuous and nonnegative function by Step 2. Under these conditions, from \eq{fluid-5},
\begin{align*}
  \ol{\gamma}_{0}(t) - \ol{\gamma}_{0}(s_{1}) = - \gamma_{\infty} \int_{s_{1}}^{t} 1(\ol{L}(s) > 0) ds + \ol{L}(t) > 0, \qquad t \in (s_{1},t_{1}].
\end{align*}
Because $\gamma_{\infty} < 0$ by (\sect{preliminaries}.e) (see ($\ast$)), this contradicts \eq{fluid-boundary}. Hence, once $\ol{L}(t)$ attains $0$ at time $t$, then $\ol{L}(s) = 0$ for all $s \ge t$. Since $\ol{L}(0) = 1$, this time $t$ equals $\delta_{L} =  1/|\gamma_{\infty}|$ by \eq{fluid-5}, and $\ol{L}(\cdot)$ is uniquely obtained as $\ol{L}(t) = (1 - |\gamma_{\infty}|t)^{+}$. Thus, \eq{fluid-vanish} and therefore \eq{fluid-L1} are proved. This completes the proof of \lem{Xn-stability}.


\subsection{Proof of \lem{h-limit}}
\label{app:h-limit}

We start to prove the first equation in \eq{limit-la-mu-n}. It follows from \eq{heavy-t-la1} in (\sect{preliminaries}.f) that there is an $n_{0} \ge 1$ for any $\varepsilon > 0$ such that
\begin{align*}
 \sup_{u \in \dd{R}_{+}} \left|\widehat{\lambda}^{(n)}(u) - \lambda(u) - n^{-1/2} \lambda^{*}(u)\right| < n^{-1/2} \varepsilon, \qquad \forall n \ge n_{0}.
\end{align*}
The triangle inequality and this yield
\begin{align*}
 \sup_{u \in \dd{R}_{+}} \left|\widehat{\lambda}^{(n)}(u) - \lambda(u)\right| & \le \sup_{u \in \dd{R}_{+}} \left|\widehat{\lambda}^{(n)}(u) - \lambda(u) - n^{-1/2} \lambda^{*}(u)\right| + \sup_{u \in \dd{R}_{+}} \left|n^{-1/2} \lambda^{*}(u)\right| \nonumber\\
 & \le n^{-1/2} \left(\varepsilon + \left|\sup_{u \in \dd{R}_{+}} \lambda^{*}(u)\right|\right) \to 0, \qquad n \to \infty,
\end{align*}
because $\lambda^{*}(u)$ is uniformly bounded in $u \in \dd{R}_{+}$. Thus, the first equation in \eq{limit-la-mu-n} is proved, and its second equation is similarly proved.

We next prove the first equation of \eq{limit-b-sg-n}. Also using the triangle formula, \eq{heavy-t-la1}, \eq{heavy-t-mu1} and $\lambda(u) = \mu(u)$ yield
\begin{align*}
 & \sup_{u \in \dd{R}_{+}}|\widehat{b}^{(n)}(u) - b(u)| = \sup_{u \in \dd{R}_{+}} \left|n^{1/2}(\widehat{\lambda}^{(n)}(u) - \widehat{\mu}^{(n)}(u)) - (\lambda^{*}(u) - \mu^{*}(u)) \right| \nonumber\\
 & \quad \le \sup_{u \in \dd{R}_{+}}\left|n^{1/2} (\widehat{\lambda}^{(n)}(u) - \lambda(u)) - \lambda^{*}(u)\right| \nonumber\\
 & \qquad + \sup_{u \in \dd{R}_{+}}\left|n^{1/2} (\widehat{\mu}^{(n)}(u) - \mu(u) ) - \mu^{*}(u)\right| \to 0, \qquad n \to \infty.
\end{align*}
Thus, the first equation in \eq{limit-b-sg-n} is proved. Its second equation is obtained by applying \eq{limit-b-sg-n} to the definitions of $\widehat{\sigma}^{(n)}(u)$ and $\sigma(u)$.

Finally, \eq{beta-limit} follows from \eq{limit-b-sg-n} because $\liminf_{n \ge 1} \inf_{u \in \dd{R}_{+}} \widehat{\sigma}^{(n)}(u) > 0$ and $\inf_{u \in \dd{R}_{+}} \sigma(u) > 0$.

\subsection{Proof of \lem{limit-hbn}}
\label{app:limit-hbn}

By (f2) in the condition (\sect{preliminaries}.f) and \eq{bsber-u}, $b(u)$ may have countably many discontinuous points which are finite in each finite interval. Denote the sequence of these discontinuous points in increasing order by $\{x_{i}; i \in \dd{N}\}$, and let $x_{0} = 0$ and $x_{\infty} = \lim_{i \to \infty}x_{i}$. Define intervals $U_{i}$ as $U_{1} = [0,x_{1}]$, $U_{i} = (x_{i-1}, x_{i}]$ for $i \ge 2$ and $U_{\infty} = (x_{\infty}, \infty)$. Note that $U_{\infty}$ is empty if $x_{\infty} = \infty$. Let $\ol{\dd{N}} = \dd{N} \cup \{\infty\}$. Obviously, $\cup_{i \in \ol{\dd{N}}} U_{i} = \dd{R}_{+}$. Since $b(\cdot)$ is uniformly continuous on each $U_{i}$ for $i \in \ol{\dd{N}}$ and bounded on $\dd{R}_{+}$, we have
\begin{align}
 \label{eq:bn-b}
 & \int_{0}^{\infty} \widehat{b}^{(n)}(u) d\nu^{(n)}(u) - \int_{0}^{\infty} b(u) d\nu(u) = \sum_{i=1}^{\infty} \left(\int_{U_{i}} \widehat{b}^{(n)}(u) d\nu^{(n)}(u) - \int_{U_{i}} b(u) d\nu(u)\right) \nonumber\\
 & \quad = \sum_{i=1}^{\infty} \left(\int_{U_{i}} (\widehat{b}^{(n)}(u) - b(u)) d\nu^{(n)}(u)\right) + \sum_{i=1}^{\infty} \left(\int_{U_{i}} b(u) d\nu^{(n)}(u) - \int_{U_{i}} b(u) d\nu(u) \right) \nonumber\\
 & \quad = \int_{0}^{\infty} (\widehat{b}^{(n)}(u) - b(u)) d\nu^{(n)}(u) + \sum_{i=1}^{\infty} \left(\int_{U_{i}} b(u) d\nu^{(n)}(u) - \int_{U_{i}} b(u) d\nu(u) \right).
\end{align}
By \lem{h-limit}, there is $n_{0} \ge 1$ for each $\varepsilon> 0$ such that $\sup_{u \in \dd{R}_{+}} |\widehat{b}^{(n)}(u) - b(u)| < \varepsilon$ for $\forall n \ge n_{0}$, and therefore
\begin{align*}
 & \left|\int_{0}^{\infty} (\widehat{b}^{(n)}(u) - b(u)) d\nu^{(n)}(u)\right| < \varepsilon, \qquad n \ge n_{0}.
\end{align*}
Similarly, since $\nu^{(n)} \os{v}{\to} \nu$ as $n \to \infty$, there are $n_{1} \ge 1$ and $y_{0}$ for each $\varepsilon> 0$ such that
\begin{align*}
 & \left| \int_{y_{0}}^{\infty} b(u) d\nu^{(n)}(u) - \int_{y_{0}}^{\infty} b(u) d\nu(u)\right| \le \sup_{u \ge y_{0}} |b(u)| {((\nu^{(n)}((y_{0},\infty)) + \nu((y_{0},\infty))} < \varepsilon \quad n \ge n_{1}.
\end{align*}
Furthermore, for each $i \in \ol{\dd{N}}$,
\begin{align*}
 & \int_{U_{i} \cap [0,y_{0}]} b(u) d\nu^{(n)}(u) - \int_{U_{i} \cap [0,y_{0}]} b(u) d\nu(u) \to 0, \qquad n \to \infty.
\end{align*}
Hence, the right-hand side of \eq{bn-b} vanishes as $n \to \infty$. This proves \eq{limit-hbn} for $x=0$ and $y=\infty$. Similarly, \eq{limit-hbn} is proved for $0 < x < y < \infty$, replacing $\widehat{b}^{(n)}(u)$ and $b(u)$ by $\widehat{b}^{(n)}(u) 1_{(x,y]}(u)$ and $b(u)1_{(x,y]}(u)$, respectively. The other cases are similarly proved for \eq{limit-hbn}. The proof for \eq{limit-hsn} is essentially the same as that of \eq{limit-hbn}.

\subsection*{Acknowledgement} This paper was originally submitted for the study of the level dependent queue. We thank the editor and two referees for their encouragement, which significantly improves the original submission \cite{KobaMiyaSaku2025} including its title. We are grateful to Rami Atar for discussions about the process and distributional limits for the multi-level queue. Yutaka Sakuma is supported by JSPS KAKENHI Grant Number JP22K11927.


\def\cprime{$'$} \def\cprime{$'$} \def\cprime{$'$} \def\cprime{$'$}
  \def\cprime{$'$} \def\cprime{$'$} \def\cprime{$'$}

\end{document}